\documentclass[10pt]{amsart}
\usepackage{geometry}                % See geometry.pdf to learn the layout options. There are lots.
\usepackage{graphicx}
\usepackage{dsfont}
\usepackage{amssymb}
\usepackage{amsmath}
\usepackage{epstopdf}
\usepackage{fullpage}
\usepackage{enumerate}
\usepackage{color,subfig}
\usepackage{amsthm}
\usepackage{placeins}
\usepackage{array}
\usepackage{blkarray}
\usepackage{multirow}
\usepackage{float}
\usepackage{morefloats}
\usepackage{pgfplots}
\input{macros.sty}

\usepackage{amsopn}
\let\oldtocsection=\tocsection

\let\oldtocsubsection=\tocsubsection

\let\oldtocsubsubsection=\tocsubsubsection

\renewcommand{\tocsection}[2]{\hspace{0em}\oldtocsection{#1}{#2}\textbf}
\renewcommand{\tocsubsection}[2]{\hspace{1em}\oldtocsubsection{#1}{#2}}
\renewcommand{\tocsubsubsection}[2]{\hspace{2em}\oldtocsubsubsection{#1}{#2}}

\newcommand{\dv}{\text{\rm div}}
\newcommand{\dmin}{d_\text{\rm min}}
\newcommand{\rinj}{r_\text{\rm inj}}
\renewcommand{\o}{\text{\rm o}}
\renewcommand{\O}{{\mathcal O}}
\renewcommand{\d}{\text{\rm d}}
\newcommand{\capa}{\text{\rm cap}}

\newcommand{\e}{\varepsilon}
\renewcommand{\exp}{\text{\rm exp}}
\newcommand{\I}{\text{\rm I}}
\newcommand{\Id}{\text{\rm Id}}

\newcommand{\Ker}{\text{\rm Ker}}
\newcommand{\Ran}{\text{\rm Ran}}

\newcommand{\R}{{\mathbb R}}
\newcommand{\dist}{\text{\rm dist}}

\newcommand{\D}{{\mathbb D}}
\newcommand{\N}{{\mathbb N}}
\newcommand{\calU}{{U}}

\begin{document}
\newtheorem{theorem}{Theorem}[section]
\newtheorem{remark}{Remark}[section]
\newtheorem{definition}{Definition}[section]
\newtheorem{lemma}{Lemma}[section]
\newtheorem{corollary}{Corollary}[section]
\newtheorem{proposition}{Proposition}[section]
\numberwithin{equation}{section}

\title{Small perturbations in the type of boundary conditions for an elliptic operator}
\author{
E. Bonnetier\textsuperscript{1}, C. Dapogny \textsuperscript{2} and M.S. Vogelius \textsuperscript{3}
}
\maketitle
%%%%%%%%%%%%%%%%%%%%%%%%%%%%%%%%%%%%%%%%%%%%%%%%%%%%%%%%%%%%%%%%%%%%%%%%%%%%%%%%%%%%%%%%%%%%%%%%
\begin{center}
\emph{\textsuperscript{1} Institut Fourier, Universit\'e Grenoble-Alpes, BP 74, 38402 Saint-Martin-d'H\`eres Cedex, France}.\\
\emph{\textsuperscript{2} Univ. Grenoble Alpes, CNRS, Grenoble INP, LJK, 38000 Grenoble, France}.\\
\emph{\textsuperscript{3} Rutgers University,  Department of Mathematics, New Brunswick, NJ, USA.}
\end{center}
%%%%%%%%%%%%%%%%%%%%%%%%%%%%%%%%%%%%%%%%%%%%%%%%%%%%%%%%%%%%%%%%%%%%%%%%%%%%%%%%%%%%%%%%%%%%%%%%

\begin{abstract}
In this article, we study the impact of a change in the type of boundary conditions of an elliptic boundary value problem. 
In the context of the conductivity equation we consider a reference problem with mixed homogeneous Dirichlet and Neumann boundary conditions.
Two different perturbed versions of this ``background'' situation are investigated, when 
(i) The homogeneous Neumann boundary condition is replaced by a homogeneous Dirichlet boundary condition 
on a ``small'' subset $\omega_\e$ of the Neumann boundary; and when
(ii) The homogeneous Dirichlet boundary condition is replaced by a homogeneous Neumann boundary condition on 
a ``small'' subset $\omega_\e$ of the Dirichlet boundary. 
The relevant quantity that measures the ``smallness'' of the subset $\omega_\e$ differs in the two cases: while it is 
the harmonic capacity of $\omega_\e$ in the former case, we introduce a notion of ``Neumann capacity'' to handle 
the latter. 
In the first part of this work we derive representation formulas that catch the structure of the first non trivial term 
in the asymptotic expansion of the voltage potential, for a general $\omega_\e$, under the sole assumption that it 
is ``small''  in the appropriate sense.
In the second part, we explicitly calculate the first non trivial term in the asymptotic expansion of the voltage potential, in the 
particular geometric situation where the subset $\omega_\e$ is a vanishing surfacic ball.
\end{abstract}

%%%%%%%%%%%%%%%%%%%%%%%%%%%%%%%%%%%%%%%%%%%%%%%%%%%%%%%%%%%%%%%%%%%%%%%%%%%%%%%%%%%%%%%%%%%%%%%%
\bigskip
\hrule
\tableofcontents
\vspace{-0.5cm}
\hrule
\bigskip
\bigskip
%%%%%%%%%%%%%%%%%%%%%%%%%%%%%%%%%%%%%%%%%%%%%%%%%%%%%%%

%%%%%%%%%%%%%%%%%%%%%%%%%%%%%%%%%%%%%%%%%%%%%%%%%%%%%%%
\section{General setting of the problem}
%%%%%%%%%%%%%%%%%%%%%%%%%%%%%%%%%%%%%%%%%%%%%%%%%%%%%%%

\noindent Understanding the perturbations in physical fields caused by the presence of small inhomogeneities in a known ambient medium is  crucial for a variety of purposes.
For instance, it allows one to appraise the robustness of the behavior of a body with respect to alterations of its constituent material, 
or to reconstruct ``small'' inclusions with unknown locations, shapes and properties
inside this body; see \cite{ammari2004reconstructionbook} for an overview of such applications. 
From the mathematical point of view, this task translates into the asymptotic analysis of the solution $u_\e$ to a ``physical'' partial differential equation, whose defining domain or material coefficients are perturbed at a small scale, parametrized by the vanishing parameter $\e$.
Many instances of this general question have been investigated: beyond the model setting of the conductivity equation, addressed for instance in \cite{capdeboscq2003general,ammari2003accurate,cedio1998identification},
let us mention the studies \cite{ammari2002complete,beretta2012small} in the context of the linearized elasticity system, or the works \cite{ammari2001asymptotic,griesmaier2011general} devoted to the Maxwell equations. 

Here we investigate, in the physical context of the conductivity equation, an interesting variant of the aforementioned  problems, namely the variant when the type of the boundary condition is changed on small sets.

 Throughout this article, $\Omega \subset \R^d$ is a smooth, bounded domain ($d=2$ or $3$), whose boundary is decomposed as follows 
\begin{equation}\label{eq.decpOm}
 \partial \Omega = \overline{\Gamma_D} \cup \overline{\Gamma_N}, \text{ where } \Gamma_D, \: \Gamma_N \text{ are disjoint, non empty, open Lipschitz subsets of } \partial \Omega~,
 \end{equation}
where we refer to \cref{def.lipset} below for the definition of an open Lipschitz subset of $\partial \Omega$.
The regions $\Gamma_D$ and $\Gamma_N$ correspond to homogeneous Dirichlet, and homogeneous Neumann conditions for the voltage potential, respectively; see \cref{fig.illussetting} for an illustration of this setting (in the case $d=3$).
The domain $\Omega$ is occupied by a medium with smooth isotropic conductivity $\gamma \in {\mathcal C}^\infty(\overline\Omega)$, satisfying the bounds
\begin{equation}\label{eq.ellconduc}
 \forall x \in \Omega, \:\: \alpha \leq \gamma(x) \leq \beta,
 \end{equation}
for some fixed constants $0 < \alpha \leq \beta$. 
The ``background'' voltage potential $u_0$, in response to a smooth external source $f \in {\mathcal C}^\infty(\overline\Omega)$, is the unique
$H^1(\Omega)$ solution  to the mixed boundary value problem
\begin{equation}\label{eq.bg}
\left\{ 
\begin{array}{cl}
-\dv(\gamma \nabla u_0) = f & \text{in } \Omega, \\
u_0 = 0 & \text{on } \Gamma_D, \\
\gamma\frac{\partial u_0}{\partial n} = 0 & \text{on } \Gamma_N. 
\end{array}
\right.
\end{equation}
We notice that, as a consequence of the classical regularity theory for elliptic partial differential equations, 
$u_0$ is smooth except at the interface, $\Sigma$, between $\Gamma_D$ and $\Gamma_N$, where the boundary condition changes type; 
see e.g. \cite{brezis2010functional,gilbarg2015elliptic}. 

%%In two dimension, if $f$ vanishes near $\Sigma$, then simple considerations will show that $u_0$ globally is in $H^{3/2-\epsilon}(\Omega)$.

In this paper we analyze perturbed versions of \cref{eq.bg}, where the boundary conditions are modified on a ``small'', open Lipschitz subset  $\omega_\e$ of the boundary $\partial \Omega$. 
More precisely, we are interested in two different situations: 
\begin{itemize}
\item The case where the homogeneous Neumann boundary condition is replaced by a homogeneous Dirichlet boundary condition
on a ``small'' open Lipschitz subset $\omega_\e$, lying strictly inside the region $\Gamma_N$. In this situation, the voltage potential $u_\e$ is the unique $H^1(\Omega)$ solution to the boundary value problem
\begin{equation}\label{eq.uepsdir}
\left\{ 
\begin{array}{cl}
-\dv(\gamma \nabla u_\e) = f & \text{in } \Omega, \\
u_\e = 0 & \text{on } \Gamma_D \cup \omega_\e, \\
\gamma\frac{\partial u_\e}{\partial n} = 0 & \text{on } \Gamma_N \setminus \overline{\omega_\e}. 
\end{array}
\right.
\end{equation}
\item The case where the homogeneous Dirichlet boundary condition on $\Gamma_D$ is replaced by a homogeneous Neumann boundary condition on a ``small'' open Lipschitz
subset $\omega_\e$, located strictly inside $\Gamma_D$. The voltage potential $u_\e$ is then the unique $H^1(\Omega)$  solution to the boundary value problem 
\begin{equation}\label{eq.uepsneu}
\left\{ 
\begin{array}{cl}
-\dv(\gamma \nabla u_\e) = f & \text{in } \Omega, \\
u_\e = 0 & \text{on } \Gamma_D \setminus \overline{\omega_\e}, \\
\gamma\frac{\partial u_\e}{\partial n} = 0 & \text{on } \Gamma_N \cup \omega_\e. 
\end{array}
\right.
\end{equation}
\end{itemize}
In either case, we assume that $\omega_\e$ lies ``far'' from the transition region $\Sigma$, in the sense that
\begin{equation}\label{assum.far}
\text{There exists a constant } \dmin >0 \text{ such that, for all } \e >0, \:\: \dist(\omega_\e, \Sigma) \geq \dmin.
\end{equation}

\begin{figure}[!ht]
\centering
\begin{minipage}{0.6\textwidth}
\includegraphics[width=1.0\textwidth]{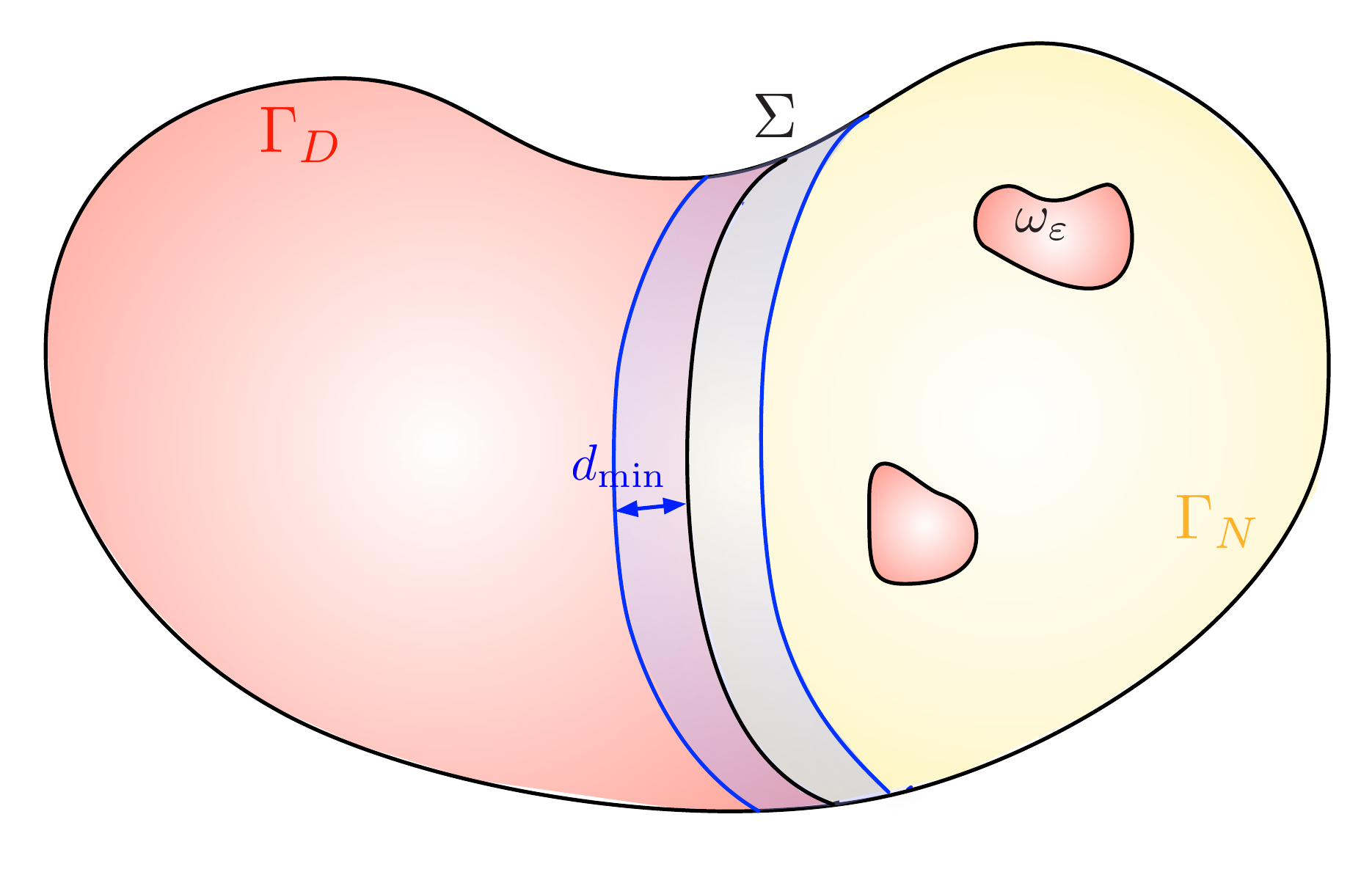}
\end{minipage} 
\caption{\it The considered setting when the Neumann region $\Gamma_N$ is perturbed by a ``small'' subset $\omega_\e$ bearing homogeneous Dirichlet boundary conditions.}
\label{fig.illussetting}
\end{figure} 

Such problems show up in multiple physical applications. 
The former situation, where homogeneous Neumann boundary conditions are replaced by Dirichlet boundary conditions, is sometimes referred to as the ``narrow escape problem'' in the literature. Originating from acoustics, it has recently attracted much attention due to its significance in the field of biology. 
In this setting indeed, $\Omega$ represents a cavity whose boundary is reflecting except on the small absorbing window $\omega_\e \subset \partial \Omega$. 
The particles inside $\Omega$ are guided by a Brownian motion; they may only leave through the region $\omega_\e$ and 
the solution $u_\e$ to \cref{eq.uepsdir} then represents their mean exit time. 
We refer to \cite{holcman2014narrow} and the references therein for an overview of the physical relevance of this problem and for an account of recent developments.
In this context, the asymptotic behavior of $u_\e$ in the limit where $\e$ vanishes has been analyzed in \cite{cheviakov2010asymptotic,pillay2010asymptotic} 
by means of formal matched asymptotic expansions; the rigorous proofs of these results were later provided in \cite{chen2011asymptotic,ammari2012layer} for ``simple" sets $\omega_\e$. 
Let us also mention the interesting variant of this ``narrow escape problem'', tackled in \cite{li2014matched,li2017asymptotic}, 
where the vanishing exit region $\omega_\e$ is connected to a thin, elongated channel, whose presence is modeled through the replacement of homogeneous Neumann conditions by Robin (and not Dirichlet) boundary conditions on $\omega_\e$.

The case \cref{eq.uepsneu} where homogeneous Dirichlet boundary conditions are replaced by homogeneous Neumann boundary conditions 
on the vanishing region $\omega_\e \subset \partial \Omega$ seems to have been more rarely considered. 
Let us however mention the early investigations conducted in \cite{gadyl1986asymptotics,gadyl1996perturbation}, 
where the asymptotics of the eigen elements of the Laplace operator are examined as the boundary condition passes from Dirichlet to Neumann type 
on a small surfacic ball $\omega_\e \subset \partial \Omega$.
An analogous study is found in \cite{planida2004asymptotics}, where, in three space dimensions, 
the small subset $\omega_\e$ is shaped as a thin neighborhood of a curve on $\partial \Omega$. 
An interesting physical motivation for this problem
was recently provided in the work \cite{kaina2014hybridized}, devoted to the construction of metasurfaces capable of affecting such changes in boundary conditions. This mechanism was analyzed from the mathematical point of view, and the corresponding asymptotic behavior of $u_\e$ was derived in 2d in \cite{ammari2019mathematical}, under the technical assumption that the boundary of $\Omega$ is completely flat in a neighborhood of the $\omega_\e$; 
these results were then used in \cite{ammari2020wave} so as to determine the optimal placement of such metasurfaces.

The present paper addresses both situations \cref{eq.uepsdir,eq.uepsneu}: our goal is to understand the asymptotic behavior of $u_\e$ as $\e \to 0$, a limiting regime in which 
the small inclusions $\omega_\e$, where boundary conditions are changed, ``vanish'' in an appropriate sense. As we shall see, the relevant measure of ``smallness'' for the set $\omega_\e$ depends on which one  of the above situations we are in. 
Our investigations go in two complementary directions. In the first part of this paper, we work from a quite abstract point of view, 
making minimal assumptions about the inclusion set $\omega_\e$, apart from ``smallness''.
We derive the general structure of the lowest order terms in the asymptotic expansion of the perturbation $u_\e-u_0$. In the second part of this paper, we consider a more specific situation 
as far as the geometry of the inclusion set $\omega_\e$ is concerned: we assume that $\omega_\e$ is a surfacic ball of radius $\e$ on $\partial \Omega$. 
In the two- and three-dimensional instances of \cref{eq.uepsdir,eq.uepsneu}, 
we precisely calculate the lowest order terms in the asymptotic expansion of the perturbation  $u_\e-u_0$, thus offering four non trivial examples of our more abstract formulas. 
As we shall see, our mathematical treatment of these four cases, based on an integral equation method, 
displays some similarities but also important differences. To emphasize both aspects, we shall use the same notation for corresponding quantities, as far as possible.

This paper is organized as follows. 
In \cref{sec.prel}, we recall some background material from functional analysis and potential theory, 
which is essential for the rest of our investigation. 
In \cref{sec.dir} we analyze, from an abstract point of view, the general structure of the (lowest order terms of the) perturbation $u_\e-u_0$, when the homogeneous Neumann boundary condition is replaced with a homogeneous Dirichlet boundary condition on a small subset $\omega_\e \Subset \Gamma_N$. 
In \cref{sec.neu} we investigate the case  when the homogeneous Dirichlet boundary condition is replaced with a homogeneous Neumann boundary condition on a small subset $\omega_\e \Subset \Gamma_D$. 
\cref{sec.calcdir,sec.calcneu} are then devoted to the explicit asymptotic expansion of $u_\e-u_0$ for  both scenarios in the particular case where $\omega_\e$ is a small surfacic ball, 
lying in $\Gamma_N$ or $\Gamma_D$, respectively. 
In \cref{sec.persp} we outline a few natural ideas for future work, suggested by the present study.
At the end this article are four appendices, collecting several useful results from the litterature, as well as some technical calculations in close connection with the topics discussed in the main parts of the text.

%%%%%%%%%%%%%%%%%%%%%%%%%%%%%%%%%%%%%%%%%%%%%%%%%%%%%%%
\section{Preliminary material}\label{sec.prel}
%%%%%%%%%%%%%%%%%%%%%%%%%%%%%%%%%%%%%%%%%%%%%%%%%%%%%%%

\noindent We initiate our study by collecting some essential background material. 
In \cref{sec.Hs}, we outline classical results about fractional Sobolev spaces defined on the boundary of a smooth domain $\Omega$, or on a relatively open Lipschitz subset $\Gamma \subset \partial \Omega$;
in the latter case we emphasize the difference between the spaces $H^s(\Gamma)$ and $\widetilde H^s(\Gamma)$.  
In \cref{sec.pot} we summarize the main properties of layer potential operators, and in \cref{sec.fundasol} we make a few remarks about the construction of fundamental solutions to boundary value problems with variable coefficients. 
Finally, in \cref{sec.capacity} we introduce and discuss the notion of $H^1$ capacity, 
which turns out to be the relevant measure of smallness for sets supporting Dirichlet boundary conditions. 

%%%%%%%%%%%%%
\subsection{The Sobolev spaces $H^s(\partial \Omega)$, $H^s(\Gamma)$ and $\widetilde{H}^{s}(\Gamma)$}\label{sec.Hs}
%%%%%%%%%%%%%

\noindent As is customary in the literature, for an arbitrary integer $n \geq 0$, 
$H^n(\partial\Omega)$ stands for the Sobolev space of functions $u  \in L^2(\partial \Omega)$ defined on the boundary of $\Omega$, 
whose tangential derivatives up to order $n$ also belong to $L^2(\partial \Omega)$, and the space $H^{-n}(\partial \Omega)$ is the topological dual of $H^n(\partial \Omega)$. 

The definition of Sobolev spaces with fractional exponents on the closed hypersurface $\partial \Omega$, 
or on an open Lipschitz subset $\Gamma \subset \partial \Omega$ gives rise to some subtleties, 
which we briefly describe in this section, referring to \cite{mclean2000strongly} and \cite{grisvard2011elliptic,lions1968problemes} for more details. 

Let us first consider Sobolev spaces of functions attached to the whole boundary $\partial \Omega$. Given a real number $0<s<1$, there are several equivalent ways of defining a norm on the fractional Sobolev Space $H^s(\partial\Omega)$; we use the following definition
$$
\Vert v \Vert^2_{H^s(\partial\Omega)} = \Vert v \Vert_{L^2(\partial\Omega)}^2 + \int_{\partial\Omega}{\int_{\partial\Omega}{\frac{|v(x)-v(y)|^2}{|x-y|^{d-1+2s}}\:\d s(x)}\:\d s(y)}~.
$$
Note that, in the literature the geodesic distance $d^{\partial \Omega}(x,y)$ between two points $x,y \in \partial \Omega$ is often used in place of the Euclidean one $|x-y|$ in the above formula. 
However, since $\Omega$ is smooth and compact, the resulting norms are equivalent (with a constant depending on $\Omega$); see \cref{lem.eqdist}. 

When $-1< s \leq 0$, $H^s(\partial \Omega)$ is the topological dual of $H^{-s}(\partial \Omega)$.

\par\medskip

We next turn to Sobolev spaces $H^s(\Gamma)$, defined on a proper region $\Gamma \subset \partial \Omega$, 
and to this end, we introduce a definition. 
\begin{definition}\label{def.lipset}
An open, connected subset $\Gamma \subset \partial \Omega$ is called a Lipschitz subdomain if locally at its boundary, $\Gamma$ consists of all points located on one side of the graph of a Lipschitz function. A Lipschitz subset $\Gamma \subset \partial \Omega$ is then defined to be the reunion of a finite number of Lipschitz subdomains, the closures of which do not intersect.
\end{definition}
%This minimal regularity assumption ensures that any function  $u\in H^1(\R^d)$ has a well-defined trace $u|_{\Gamma} \in H^{1/2}(\Gamma)$ on a given  Lipschitz set $\omega \partial \Omega$. This trace also coincides with the restriction to $\Gamma$ of $u|_{\partial \Omega}$ (the trace of $u$ on $\partial \Omega$).

Let then $\Gamma \subset \partial \Omega$ be a Lipschitz subset of $\partial \Omega$.
For any real number $0<s<1$ we introduce the following two classes of Sobolev spaces on $\Gamma$: 
\begin{itemize}
\item $\widetilde{H}^s(\Gamma)$ denotes the space of (restrictions to $\Gamma$ of) functions in $H^s(\partial \Omega)$ with compact support inside $\overline{\Gamma}$. This space is equipped with the norm $|| \cdot ||_{H^{s}(\partial \Omega)}$; it is the closure in $H^s(\partial \Omega)$ of the set of ${\mathcal C}^\infty$ functions on $\partial \Omega$ with compact support inside $\Gamma$. Equivalently, $u$ belongs to $\widetilde H^s(\Gamma)$ if and only if its extension by $0$ to all of $\partial \Omega$, 
which we throughout the following still denote by $u$, belongs to $H^s(\partial \Omega)$.
\item  $H^s(\Gamma)$ is the space of the restrictions to $\Gamma$ of functions in $ H^s(\partial \Omega)$. This space is equipped with the norm:
\begin{equation}\label{eq.normHs}
 \Vert v \Vert^2_{H^s(\Gamma)} = \Vert v \Vert_{L^2(\Gamma)}^2 + | v |^2_{H^s(\Gamma)}, \text{ where } | v |^2_{H^s(\Gamma)}:=\int_{\Gamma}{\int_{\Gamma}{\frac{|v(x)-v(y)|^2}{|x-y|^{d-1+2s}}\:\d s(x)}\:\d s(y)}~, 
\end{equation}
which is equivalent to the quotient norm induced by that of $H^s(\partial \Omega)$, up to constants that may depend on $\Gamma$.
\end{itemize}
Let us point out a few facts about the relation between both types of spaces: 
\begin{itemize}
\item When $0<s<1/2$, the spaces  $\widetilde{H}^s(\Gamma)$ and  $H^s(\Gamma)$ are identical, with equivalent norms. 
On the other hand, when $\frac{1}{2} \leq s<1$,  $\widetilde{H}^s(\Gamma)$ is a proper subspace of  $H^s(\Gamma)$.
\item When $\frac{1}{2}<s<1$, the space $\widetilde H^s(\Gamma)$ coincides with $H^s_0(\Gamma)$, the closure in $H^s(\Gamma)$ (for the natural norm \cref{eq.normHs}) of the set of ${\mathcal C}^\infty$ functions with compact support $K \Subset \Gamma$. 
\end{itemize}

For any real number $-1<s<0$, $H^s(\Gamma)$ is still defined as the space of restrictions to $\Gamma$ of distributions in $H^s(\partial \Omega)$ (equipped with the quotient norm). This space can be identified with the topological  
dual of  $\widetilde{H}^{-s}(\Gamma)$, using as a pairing the natural extension of the $L^2(\Gamma)$ inner product, that we denote by: 
$$ \langle u ,v \rangle, \:\: u \in H^s(\Gamma), \:\: v \in \widetilde H^{-s}(\Gamma).$$ 
Similarly, $\widetilde{H}^{s}(\Gamma)$  is the space of distributions in $H^s(\partial \Omega)$ 
with compact support inside $\overline\Gamma$. It is identified with
the dual space of  $H^{-s}(\Gamma)$, using the same pairing (with the same notation).\par\medskip

The case when $s= 1/2$ is particular: $\widetilde H^{1/2}(\Gamma)$ is a proper subspace of $H^{1/2}(\Gamma)$, with a strictly stronger norm, while the latter space, incidentally, coincides with $H^{1/2}_0(\Gamma)$. To better appraise this distinction between $\widetilde{H}^{1/2}(\Gamma)$ and $H^{1/2}(\Gamma)$, we calculate the norm $\Vert u \Vert_{\widetilde H^{1/2}(\Gamma)}= \Vert u \Vert_{H^{1/2}(\partial \Omega)}$ of an arbitrary function $u \in \widetilde{H}^{1/2}(\Gamma)$:
$$
\begin{array}{>{\displaystyle}cc>{\displaystyle}l}
 \Vert u \Vert ^2_{\widetilde H^{1/2}(\Gamma)}  &=&  || u||^2_{L^2(\Gamma)} + \int_{\Gamma}{\int_{\Gamma}{\frac{|u(x)-u(y)|^2}{|x-y|^{d}} \:\d s(x)}\:\d s(y)} + 2 \int_{\Gamma}{\rho_{\Gamma}(x) | u(x)|^2 \:\d s(x)} \\
 &= &  \Vert u \Vert^2_{H^{1/2}(\Gamma)} + 2 \int_{\Gamma}{\rho_{\Gamma}(x) | u(x)|^2 \:\d s(x)}~.
\end{array}
$$
The weight $\rho_{\Gamma}$ is here defined by 
$$\forall x \in \Gamma,\quad \rho_{\Gamma}(x) = \int_{\partial \Omega \setminus \Gamma}{\frac{1}{|x-y|^d}\:\d s(y)}~.$$
The above norm on the space $\widetilde{H}^{1/2}(\Gamma)$ is stronger than that on $H^{1/2}(\Gamma)$, and in particular
\begin{equation}
\label{normcalc}
\left( \int_{\Gamma}{\rho_{\Gamma}(x) | u(x)|^2 \:\d s(x)}\right)^\frac12 \le \frac{1}{\sqrt 2} \Vert u \Vert _{\widetilde H^{1/2}(\Gamma)}~.
\end{equation}

\par\medskip

The spaces with exponents $\pm \frac12$ are particularly relevant in the context of variational solutions to  boundary value problems like \cref{eq.bg}.
By a variational  solution to  \cref{eq.bg} we understand a function $u_0$ in the functional space
$$ H^1_{\Gamma_D}(\Omega) := \left\{ u \in H^1(\Omega), \:\: u = 0\text{ on } \Gamma_D \right\}$$
of $H^1(\Omega)$ functions with vanishing trace on $\Gamma_D$ (in other words $u_0|_{\Gamma_N} \in \widetilde H^{1/2}(\Gamma_N)$), and for which
$$
\int_{\Omega}\gamma \nabla u_0 \cdot \nabla v~\d x = \int_\Omega f v~\d x~,
$$
for all $v\in H^1_{\Gamma_D}(\Omega)$ (i.e., $v|_{\Gamma_N} \in \widetilde H^{1/2}(\Gamma_N)$). Using integration by parts, this identity asserts that:
$$
\gamma \frac{\partial u_0}{\partial n} =0 \text{ as an element in } H^{-1/2}(\Gamma_N), \text{ and so } \gamma \frac{\partial u_0}{\partial n} \in \widetilde H^{-1/2}(\Gamma_D)~.
$$

%\noindent The space of functions (in $H^\frac12(\partial \Omega)$) that vanish outside $\omega$ is commonly referred to as $H^{1/2}_{00}(\omega)$. It is wellknown that the quantity $\left( \int_{\omega}{\rho_{\omega}(x) | u(x)|^2 \:\d s(x)}\right)^{1/2}$ in combination with  $\Vert u \Vert_{H^\frac12(\omega)}$ constitutes a norm on $H^{1/2}_{00}(\omega)$, cf. \cite{grisvard2011elliptic,lions1968problemes}. For our purposes an essential feature of the inequality (\ref{normcalc}) is that the constant (here $1/\sqrt 2$) is independent of $\omega$. 

%%%%%%%%%%%%%
\subsection{A short review of layer potentials} \label{sec.pot}
%%%%%%%%%%%%%

\noindent In the present section, we denote by $D \subset \R^d$ a smooth bounded domain, and 
we briefly recall some background material about layer potential operators associated with $\partial D$; 
we refer to \cite{ammari2007polarization,folland1995introduction,mclean2000strongly,sauter2011boundary} for more details about such operators.

Let $G(x,y)$ be the fundamental solution of the operator $-\Delta$ in the free space $\R^d$ 
\begin{equation}\label{eq.GreenLap}
G(x,y) = \left\{
\begin{array}{cl}
-\frac{1}{2\pi} \log|x -y | & \text{if } d= 2~, \\
\frac{1}{4\pi |x-y |} & \text{if } d= 3~. 
\end{array}
\right.
\end{equation}
For $x \in \R^d$, the function $y \mapsto G(x,y)$ satisfies 
$$ -\Delta_y G(x,y) = \delta_{y=x}~,$$
in the sense of distributions in $\R^d$, where $\delta_{y=x}$ is the Dirac distribution at $x$. 

For a smooth density function $\phi \in {\mathcal C}^\infty(\partial D)$, the single layer potential associated with $\phi$ is defined by 
\begin{equation}\label{eq.SD}
 {\mathcal S}_D\phi(x) = \int_{\partial D}{G(x,y) \phi(y) \: \d s(y)}, \quad x \in \R^d \setminus \partial D,
 \end{equation}
and the corresponding double layer potential is defined by
$$ {\mathcal D}_D\phi(x) = \int_{\partial D}{\frac{\partial G}{\partial n_y}(x,y) \phi(y) \: \d s(y)}, \quad x \in \R^d \setminus \partial D.$$
The operators ${\mathcal S}_D$ and ${\mathcal D}_D$ extend to bounded operators 
from $H^{-1/2}(\partial D)$ into $H^1_{\text{loc}}(\R^d)$, and from $H^{1/2}(\partial D)$ into $H^1(D) \cup H^1_{\text{loc}}(\R^d \setminus \overline D)$, respectively. In addition, the functions ${\mathcal S}_D \phi$ and ${\mathcal D}_D \phi$ are both harmonic on $D$ and $\R^d \setminus \overline{D}$. 
Of particular interest are their behavior at the interface $\partial D$. 
Let us denote by 
\begin{equation}\label{eq.oslimits}
 g^\pm(x) := \lim\limits_{t \downarrow 0} g(x \pm t n(x)), \quad x \in \partial D
 \end{equation}
the one-sided limits of a function $g$ which is smooth enough from either side of $\partial D$ and by $[g](x) := g^+(x) - g^-(x)$ the corresponding jump across $\partial D$. 
The functions ${\mathcal S}_D \phi$ and ${\mathcal D}_D \phi$ satisfy the well-known jump relations: 
\begin{equation}\label{eq.jumpSD}
 [{\mathcal S}_D \phi] = 0 \text{ and } \left[ \frac{\partial }{\partial n}({\mathcal S}_D \phi)\right] = -\phi, 
 \end{equation}
and 
\begin{equation}\label{eq.jumpDD}
[{\mathcal D}_D \phi] = \phi \text{ and } \left[ \frac{\partial }{\partial n}({\mathcal D}_D \phi)\right] = 0. 
\end{equation}

The first and the last of these four jump relations allow to introduce the integral operators $S_D$ and $R_D$, defined for a smooth density function $\phi \in {\mathcal C}^\infty(\partial D)$ by:
$$ S_D \phi = ({\mathcal S}_D \phi) \lvert_{\partial D}, \quad S_D \phi(x)  = \int_{\partial D}{G(x,y) \phi(y) \: \d s(y)}~,~~~ x \in \partial D, $$
and 
$$ R_D \phi =\frac{\partial }{\partial n} \left( {\mathcal D}_D \phi  \right), \quad R_D \phi(x) =  \underset{\eta \downarrow 0}{\text{\rm f.p.} }\int_{\partial D \setminus B_\eta(x) }{\frac{\partial ^2 G}{\partial n_x \partial n_y}(x,y) \phi(y) \: \d s(y)}~, ~~~ x\in \partial D, $$
% In fact, one defines them through mapping properties of the solution to elliptic systems, with these function spaces, 
% and then realizes that for smooth densities, the above expressions hold.
where f.p. refers to a finite part integral in the sense of Hadamard.
These operators extend as bounded mappings $S_D : H^{-1/2}(\partial D) \to H^{1/2}(\partial D)$ and $R_D : H^{1/2}(\partial D) \to H^{-1/2}(\partial D)$.
 
Lastly, we recall the decay properties of the single and double layer potentials at infinity. 
For a given density $\phi \in H^{-1/2}(\partial D)$, it follows from the explicit expression \cref{eq.GreenLap} of the fundamental solution $G(x,y)$ that, for $d= 3$
\begin{equation}\label{eq.decayS3d}
 {\mathcal S}_D \phi(x) = {\mathcal O}\left( \frac{1}{|x|}\right), \text{ and } \lvert \nabla {\mathcal S}_D \phi(x)\lvert = {\mathcal O}\left( \frac{1}{|x|^{2}}\right),
 \end{equation}
 where we have used the convenient notation ${\mathcal O}\left( \frac{1}{|x|^{2}}\right)$ to represent a function whose modulus is bounded by $\frac{C}{|x|}$ when $|x|$ is large enough, 
 for some constant $C>0$. 
 
The case $d= 2$ is a little more subtle, and in general one only has
$$ {\mathcal S}_D \phi(x) = {\mathcal O}\left(|\log|x| |\right), \text{ and } \lvert \nabla {\mathcal S}_D \phi(x)\lvert = {\mathcal O}\left( \frac{1}{|x|}\right)~,$$
however, if $\int_{\partial D}{\phi \: \d s} = 0$, then it holds additionally
\begin{equation} \label{eq.decaySDmean0}
{\mathcal S}_D \phi(x) = {\mathcal O}\left( \frac{1}{|x|} \right), \text{ and } \lvert \nabla {\mathcal S}_D \phi(x)\lvert = {\mathcal O}\left( \frac{1}{|x|^2}\right)~.
\end{equation}
As far as the double layer potential is concerned, one has for $\phi \in H^{1/2}(\partial D)$ and $d= 2, 3$
\begin{equation}\label{eq.decayDD}
 {\mathcal D}_D \phi(x) = {\mathcal O}\left( \frac{1}{|x|^{d-1}}\right), \text{ and } \lvert \nabla {\mathcal D}_D \phi(x)\lvert = {\mathcal O}\left( \frac{1}{|x|^{d}}\right)~.
 \end{equation}

 %%%%%%%%%%%%%
\subsection{The fundamental solution $N(x,y)$ to the background equation \cref{eq.bg}}\label{sec.fundasol}
%%%%%%%%%%%%%

\noindent We turn our attention to the case when the reference problem under consideration is not the free-space Laplace equation, but rather the background boundary value problem \cref{eq.bg}. 
The fundamental solution $N(x,y)$ to the latter is constructed from that $G(x,y)$ associated to the operator $-\Delta$ in the free space, given by \cref{eq.GreenLap}, in a way which we now briefly describe. We refer, e.g., to \cite{friedman1989identification} or \cite{ammari2009layer} for similar results.

For any point $x \in \Omega$, the function $y \mapsto N(x,y)$ satisfies the following equation
 \begin{equation}\label{fundsol}
  \left\{
 \begin{array}{cl}
 -\dv_y(\gamma(y) \nabla_y N(x,y)) = \delta_{y=x} &\text{in } \Omega~,\\
 N(x,y) = 0 & \text{for } y \in \Gamma_D~, \\
 \gamma(y) \frac{\partial N}{\partial n_y}(x,y) = 0 & \text{for } y \in \Gamma_N~.
 \end{array}
 \right.
 \end{equation}
 This means that, for any function $\varphi \in {\mathcal C}^1(\overline\Omega)$ such that $\varphi = 0 $ on $\Gamma_D$, one has
 $$ \varphi(x) = \int_\Omega{\gamma(y) \nabla_y N(x,y) \cdot \nabla \varphi (y) \: \d y }, \quad x \in \Omega.$$
 By an easy adaptation of the proof of Lemma 2.36 in \cite{folland1995introduction}, one sees that the function $N(x,y)$ is symmetric in its arguments. Furthermore, it is related to the fundamental solution $G(x,y)$ to the Laplace equation in free space via the relation 
 $$ N(x,y) = \frac{1}{\gamma(x)}G(x,y) + R(x,y)~,$$
 where for given $x \in \Omega$, $y \mapsto R(x,y)$ is the solution to the equation 
 $$
   \left\{
 \begin{array}{cl}
 -\dv_y(\gamma(y) \nabla_y R(x,y)) =  \frac{1}{\gamma(x)}  \nabla \gamma(y) \cdot \nabla_y G(x,y)&\text{in } \Omega~,\\
 R(x,y) = - \frac{1}{\gamma(x)} G(x,y) & \text{for } y \in \Gamma_D~, \\
 \gamma(y) \frac{\partial R}{\partial n_y}(x,y) = -\frac{\gamma(y)}{\gamma(x)} \frac{\partial G}{\partial n_y}(x,y) & \text{for } y \in \Gamma_N~.
 \end{array}
 \right.
  $$
The precise functional characterization of $R(x,y)$ follows from standard elliptic regularity theory, 
depending on the singularity of $G(x,y)$, see \cite{brezis2010functional,gilbarg2015elliptic}. 
Without entering into technicalities, let us just mention that, for fixed $x \in \Omega$, 
the function $y \mapsto R(x,y)$ belongs (at least) to $H^1(\Omega)$.
Moreover, for every open subset $U \Subset \R^d \setminus (\Sigma \cup \{x\})$, it is of class ${\mathcal C}^\infty$ on $\overline \Omega \cap U$. 

%For $d=2$, and fixed $x\in \Omega$, one may show that the function $y \mapsto R(x,y)$ belongs to the space $H^{3/2-\eta}(\Omega)$ for any $\eta >0$. 
%
%\vskip 10pt
%
%******   WHAT DO WE NEED THE ABOVE SECTION FOR?  *******
%
%\vskip 10pt

%\begin{remark}\label{rem.otherkernel}
%The use of the fundamental solution $N(x,y)$ instead of the free space Green's function paves the way to adapted %notions of layer potentials, 
%which satisfy in particular similar jump relations to \cref{eq.jumpSD,eq.jumpDD}. 
%\end{remark}

%%%%%%%%%%%%%
\subsection{The capacity of a subset in $\R^d$}\label{sec.capacity}
%%%%%%%%%%%%%

\noindent
In one of the two scenarios studied in this article, namely when $\omega_\e$ accounts for Dirichlet boundary conditions being imposed inside the Neumann region $\Gamma_N$ (cf.  \cref{sec.dir}), the key quantity to measure the ``smallness'' of the set $\omega_\e$ will be the $H^1(\R^d)$ capacity.
For the convenience of the reader, we briefly recall the definition 
and two simple results related to this notion, referring to \cite{henrot2018shape} for further details.

\begin{definition}\label{def.capa}
The capacity $\capa(E)$ of an arbitrary subset $E \subset \R^d$ is defined by: 
\begin{equation}\label{eq.cap}
\capa (E ) = \inf \left\{ || v ||^2_{H^1(\R^d)}, \:\: v(x) \geq 1 \text{ a.e. on an open neighborhood of } E  \right\}.
\end{equation}
\end{definition}

%
%\begin{remark}
%Alternative definitions of the capacity exist, where notably the infimum in \cref{eq.cap} is taken over $H^1$ functions which are ``grounded'' with respect to a fixed subset $A\subset \R^d$, i.e., which vanish on $A$. 
%As we shall see, both quantities are equivalent, up to constants depending on the subset $A$ and the distance to it. 
%So that our asymptotic formulas feature an absolute rate, we stick to the above definition of the ``absolute'' $H^1$ capacity.
%\end{remark}

A slightly different  formula for $\capa(E)$ is that of the following lemma.
\begin{lemma}
For an arbitrary subset $E \subset \R^d$, it holds
\begin{equation}\label{eq.cap1}
\capa (E ) = \inf \left\{ || v ||^2_{H^1(\R^d)}, \:\: v(x) = 1 \text{ a.e. on an open neighborhood of } E  \right\}.
\end{equation}
\end{lemma}
\begin{proof}
On the one hand, it follows immediately from the definition \cref{eq.cap} that
$$\capa (E ) \leq \inf \left\{ || v ||^2_{H^1(\R^d)}, \:\: v(x) = 1 \text{ a.e. on an open neighborhood of } E  \right\}~.$$
Conversely, if $v_n \in H^1(\R^d)$ is a sequence of functions such that
$$ v_n \geq 1 \text{ a.e. on an open neighborhood of } E, \:\: || v_n ||^2_{H^1(\R^d)} \xrightarrow{n \to \infty} \capa(E)~,$$
then $w_n = \min(v_n,1)$ defines a sequence of functions in $H^1(\R^d)$ which satisfies
$$ || w_n ||_{H^1(\R^d) }^2 \leq  || v_n ||_{H^1(\R^d) }^2~;$$
see for instance \cite{henrot2018shape}, Proposition 3.1.11. As a result,
$$ \inf \left\{ || v ||^2_{H^1(\R^d)}, \:\: v(x) = 1 \text{ a.e. on an open neighborhood of } E  \right\} \leq || w_n ||^2_{H^1(\R^d)} \leq || v_n ||^2_{H^1(\R^d)} \xrightarrow{n \to \infty} \capa(E)~,$$
which proves \cref{eq.cap1}.
\end{proof}

%This notion of capacity paves the way to that of of quasi-continuity:

%\begin{definition}\label{def.qc}
%One function $f : \R^d \to \R$ is quasi-continuous if there exists a non increasing sequence of open sets $\omega_n \subset \R^d$ such that:
%\begin{enumerate}[(i)]
%\item $\lim\limits_{n\to \infty} \capa(\omega_n)  = 0$; 
%\item The restriction of $f$ to $\R^d \setminus \omega_n$ is continuous. 
%\end{enumerate}  
%\end{definition}

%The relevance of this concept in the context of functions in Sobolev spaces is exemplified by the following property.

%\begin{theorem}\label{th.qcrep}
%Any function $f \in H^1(\R^d)$ has a quasi-continuous representative, which is unique up to equality quasi-everywhere.
%\end{theorem}

We now provide a useful lemma, whereby
the capacity of a subset $\omega$ of the boundary $\partial \Omega$ of a smooth domain $\Omega \subset \R^d$ can be estimated in terms of the energy norm of a function whose trace equals $1$ on $\omega$.

\begin{lemma}\label{lem.capatrace}
Let $\Omega$ be a smooth bounded domain in $\R^d$, $\omega$ be a Lipschitz open subset of $\partial\Omega$, and let $u$ be a function in $H^1(\R^d)$. 
If $u=1$ on $\omega$ in the sense of traces in $H^{1/2}(\omega)$, then
$$ \capa(\omega) \leq || u ||^2_{H^1(\R^d)}.$$ 
\end{lemma}
\begin{proof}
Let $\calU$ be an open neighborhood of $\omega$ in $\R^d$, and let $\phi$ be a $C^{\infty}_c(\R^d)$ function which equals $1$ identically on $\calU$. 
We decompose the function $u$ as:
$$
u= \phi + u-\phi ~.
$$
Since the set of functions $v \in H^1(\R^d)$ with vanishing trace on $\omega$ is exactly the closure of $C^{\infty}_c(\R^d\setminus \overline \omega)$ in $H^1(\R^d)$ (see \cref{sec.Hs}), and since the trace of $u-\phi$ equals $0$ on $\omega$, there exists a sequence $v_n \in C^{\infty}_c(\R^d\setminus \overline \omega)$ such that
$$
\Vert u -\phi -v_n  \Vert_{H^1(\R^d)} \le \frac{1}{n}~.
$$
We now estimate
\begin{eqnarray*} 
\Vert u \Vert_{H^1(\R^d)}&=& \Vert \phi + v_n +  u - \phi - v_n  \Vert_{H^1(\R^d)} \\
&\ge& \Vert  \phi +v_n \Vert_{H^1(\R^d)}- \Vert u - \phi -v_n \Vert_{H^1(\R^d)} \\
&\ge&\Vert  \phi +v_n \Vert_{H^1(\R^d)} -\frac{1}{n}~.
\end{eqnarray*}
The function $ \phi +v_n$ lies in $H^1(\R^d)$ (actually it lies in $C_c^\infty(\R^d)$) and it equals $1$ on an open neighborhood of $\omega$ in $\R^d$. From the definition of $ \capa(\omega)$ it follows that 
$$
\Vert  \phi +v_n \Vert_{H^1(\R^d)} \ge\left( \capa(\omega)\right)^{1/2} ~,
$$
and so by combination with the previous estimate we get:
$$
\Vert u \Vert_{H^1(\R^d)}\ge \left(\capa(\omega)\right)^{1/2} -\frac{1}{n}~.
$$
By passing to the limit as $n \rightarrow \infty$ we arrive at the desired conclusion.
%Let $\overline u$ be the quasi-continuous representative of $u$ supplied by \cref{th.qcrep}, and let $\omega_n$ be a sequence of open subsets of $\R^d$ as in \cref{def.qc}.
%The hypothesis of the lemma implies in particular that $u(x)= 1$ for all $x \in \Gamma \setminus E$, where the subset $E \subset \Gamma$ has zero ${\mathcal H}^{d-1}$ Hausdorff measure. 

%Since $\overline u$ is continuous on $\Gamma \setminus \omega_n$, it holds $\overline u = 1$ everywhere on $\Gamma \setminus \omega_n$. Hence, 
%$$ \Gamma \subset \left\{ x \in \R^d, \:\: \overline{u}(x) \geq 1 \right\} \cup \omega_n. $$
%We now use Lemma 3.3.34 in \cite{henrot2018shape}, which states that, for any function $z \in H^1(\R^d)$ with quasi-continuous representative $\overline z$,
%the following inequality holds:
%$$ \capa \left\{ x \in \R^d \text{ s.t. } |\overline{z}(x)| > 1 \right\} \leq || z||^2_{H^1(\R^d)}.$$
%This leads to:
%$$\begin{array}{>{\displaystyle}cc>{\displaystyle}l}
% \capa(\Gamma ) &\leq& \capa \left\{ x \in \R^d,\: \:\overline{u}(x) \geq 1 \right\} + \capa(\omega_n),\\
% &\leq&  \capa \left\{ x \in \R^d,\: \: \left( 1+\frac{1}{n}\right) | \overline{u} (x)| > 1 \right\} + \capa(\omega_n), \\
% &\leq& \left( 1+\frac{1}{n}\right)^2 ||u||_{H^1(\R^d)}^2 + \capa(\omega_n);
% \end{array}
% $$
%letting $n\to \infty$ yields the desired result. 
\end{proof}

\begin{remark}\label{rem.capadir}
From the physical point of view, the capacity of a compact subset $E \subset \R^d$
is the total energy of the electric field in the whole ambient space $\R^d$, in the equilibrium regime where the potential is constant on $E$ (equal to $1$). 
Different notions of capacity are found in the literature, depending on the kernel relating the charge distribution (i.e., the source term) to the induced potential. 
A very natural notion of capacity is attached to the fundamental solution to the Laplace operator with homogeneous Dirichlet boundary conditions on a ``ground surface'' $A$ (where the potential is set to $0$); this concept is often associated with the name condenser capacity. 
The version proposed in \cref{def.capa} is convenient for our purpose, since it is somehow ``universal'' (it does not depend on the choice of a fixed grounding subset $A$ for the potential), and it  involves the Bessel kernel.
It is equivalent to the notion of condenser capacity, up to constants depending on the subset $A$; see for instance  \cref{lem.capas} for a result in this direction.
We refer to \cite{adams2012function} for an extensive discussion of the concept of capacity; see also \cite{landkof1972foundations,doob2012classical}.
\end{remark}

\begin{remark}\label{rem.explicitcapa}
In \cref{sec.calcdir} we shall be particularly interested in subsets of $\R^d$ of the form
\begin{equation}\label{eq.defDecapa}
 \D_\e := \left\{ x = (x_1,\ldots,x_{d-1},0 ) \in \R^d, \:\: | x | < \e \right\}~.
 \end{equation}
$\D_\e$ is a line segment of length $2\e$, when $d=2$, and a planar disk with radius $\e$, when $d=3$.  
The following estimates of the capacity of $\D_\e$ will come in handy
\begin{equation}\label{eq.capaDe}
 \capa(\D_\e) \leq \frac{C_2}{|\log \e|} \text{ when } d= 2, \quad \capa(\D_\e) \leq C_3\e \text{ when } d= 3, 
 \end{equation}
where $C_2$ and $C_3$ are universal constants; see for instance Chap. 2 in \cite{landkof1972foundations}.
\end{remark}

%%%%%%%%%%%%%%%%%%%%%%%%%%%%%%%%%%%%%%%%%%%%%%%%%%%%%%%
\section{Replacing Neumann conditions by Dirichlet conditions on a ``small set"}\label{sec.dir} 
%%%%%%%%%%%%%%%%%%%%%%%%%%%%%%%%%%%%%%%%%%%%%%%%%%%%%%%

\noindent In this section, we consider an arbitrary sequence $\omega_\e$ of open, Lipschitz subsets of the Neumann region $\Gamma_N$, 
which are all ``well-separated'' from the Dirichlet region $\Gamma_D$ in the sense that the assumption \cref{assum.far} holds. 
The homogeneous Neumann boundary condition 
satisfied on $\Gamma_N$ by the background potential $u_0$ (see \cref{eq.bg})
is dropped on $\omega_\e$,
where it is replaced by a homogeneous Dirichlet condition. The perturbed potential $u_\e$ in this situation is the solution to the equation \cref{eq.uepsdir}.  

As we shall see, the potential $u_\e$ converges to $u_0$ as $\e \to 0$, when the set $\omega_\e$ vanishes in an appropriate sense.
In this general setting, where no additional hypothesis is made about $\omega_\e$, 
our aim is to establish an abstract representation formula
for the first non-trivial term in the limiting asymptotics of $u_\e-u_0$. 

%%%%%%%%%%%%%%%%%%%%%%%%%%%%%%%%%%%%%%%%%%%%%%%%%%%%%%%
\subsection{Some preliminary estimates}\label{sec.prelestdir}
%%%%%%%%%%%%%%%%%%%%%%%%%%%%%%%%%%%%%%%%%%%%%%%%%%%%%%%

\noindent We start with some a priori estimates related to modified versions of the perturbed boundary value problem \cref{eq.uepsdir}. 
The first of these results 
is concerned with the unique solution $\chi_\e \in H^1(\Omega)$ to the problem  
\begin{equation}\label{eq.chieps}
\left\{ 
\begin{array}{cl}
-\Delta \chi_\e = 0 & \text{in } \Omega~, \\
\chi_\e = 1 & \text{on } \omega_\e~, \\
\chi_\e = 0 & \text{on } \Gamma_D~, \\
\frac{\partial \chi_\e}{\partial n} = 0 & \text{on } \Gamma_N \setminus \overline{\omega_\e}~. 
\end{array}
\right.
\end{equation}
Let us recall that, by a solution to \cref{eq.chieps} we understand a function $\chi_\e \in H^1_{\Gamma_D}(\Omega)$ such that $\chi_\e = 1$ on $\omega_\e$ and 
\begin{equation}\label{eq.varfchiepsdir}
 \forall v \in H^1_{\Gamma_D}(\Omega) \text{ with } v = 0 \text{ on } \omega_\e, \quad \int_\Omega{\nabla \chi_\e\cdot \nabla v \: \d x} = 0~.
 \end{equation}

\begin{lemma}\label{lem.capas}
Let $\omega_\e$ be an open Lipschitz subset of the region $\Gamma_N \subset \partial \Omega$, 
which lies ``far'' from $\Gamma_D$ in the sense that \cref{assum.far} holds,
and let $\chi_\e$ be the solution to \cref{eq.chieps}.
There exist two constants $ 0 < m \leq M$ which depend only on $\Omega$, $\Gamma_D$ and the lower bound $\dmin$ on the distance from $\omega_\e$ to $\Gamma_D$, but are otherwise independent of $\omega_\e$,  such that 
\begin{equation}\label{eq.capas}
 m \: \capa(\omega_\e) \leq || \chi_\e ||_{H^1(\Omega)}^2 \leq M \capa(\omega_\e)~.
 \end{equation}
\end{lemma}
\begin{proof}
\noindent \textit{We start with the proof of the inequality $ m \:  \capa(\omega_\e) \leq || \chi_\e ||_{H^1(\Omega)}^2$.}
Because of the smoothness of  $\Omega$, there exists an extension $\widetilde{\chi_\e} \in H^1(\R^d)$ to the whole space $\R^d$ 
such that 
\begin{equation}\label{eq.poincarechiA}
 || \widetilde{\chi_\e} ||_{H^1(\R^d)} \leq C || \chi_\e||_{H^1(\Omega)}~,
\end{equation}
where the constant $C$ depends only on $\Omega$ (see e.g. Appendix A in \cite{mclean2000strongly}). 
Using \cref{lem.capatrace} we may estimate the capacity of the subset $\omega_\e \subset \partial \Omega$, where the trace of $\widetilde{\chi_\e}$ equals $1$, by
$$ \capa(\omega_\e)  \leq  ||\widetilde{\chi_\e}||_{H^1(\R^d)}^2~.$$
This inequality, combined with \cref{eq.poincarechiA}, yields the desired result.\par\medskip 

\noindent \textit{We now prove that $  || \chi_\e ||_{H^1(\Omega)}^2 \leq M \capa(\omega_\e)$}. Let us first observe that, 
due to a classical variation of the Poincar\'e inequality, there exists a constant $C >0$ which depends only on $\Omega$ and $\Gamma_D$, such that
\begin{equation}\label{eq.Poincarecapa} 
|| \chi_\e ||_{H^1(\Omega)}^2 \leq C || \nabla \chi_\e ||^2_{L^2(\Omega)^d}~.
\end{equation}

Because of the separation assumption \cref{assum.far},
there exists a function $h \in C^1(\overline \Omega)$ such that, for any $\e >0$ small enough, one has
$$h(x)=1\text{ for } x \in \omega_\e, \:\: h(x)=0 \text{ for } x \in  \Gamma_D, \text{ and }\Vert h \Vert_{C^1(\overline \Omega)} \le C~,$$
where $C$ depends on $\dmin$, but is otherwise independent of $\omega_\e$. 
For any function $\chi \in H^1(\R^d)$ such that $\chi =1$ on an open neighborhood of $\omega_\e$, we now have
$$  \begin{array}{>{\displaystyle}cc>{\displaystyle}l}
|| \nabla \chi_\e ||^2_{L^2(\Omega)^d}  &=& \int_{\Omega}{\nabla \chi_\e \cdot \nabla \chi_\e \: \d x} \\[2ex]
%& =&  \int_{\partial \Omega}{\gamma \frac{\partial \chi_\e}{\partial n} \chi_\e \: \d s},  \\[2ex]
%&=& C \int_{\partial \Omega}{\gamma \frac{\partial \chi_\e}{\partial n} \chi h \: \d s},\\[2ex]
& =&  \int_{\Omega}{ \nabla \chi_\e \cdot \nabla (\chi h)\: \d x}\\[2ex]
&\leq& C || \nabla \chi_\e||_{L^2(\Omega)^d }|| \chi||_{H^1(\R^d) }~.
\end{array}
$$
Here we have used the fact that $\chi_\e -\chi h$ vanishes on $\Gamma_D \cup \omega_\e$, together with the variational formulation \cref{eq.varfchiepsdir}, 
to pass from the first line to the second. We immediately conclude that
$$
|| \nabla \chi_\e ||^2_{L^2(\Omega)^d}  \leq C || \chi||^2_{H^1(\R^d) }~.
$$
Since this holds for any function $\chi \in H^1(\R^d)$ which equals identically $1$ on an open neighborhood of $\omega_\e$, 
the desired upper bound for $|| \chi_\e ||_{H^1(\Omega)}^2$  follows by taking the infimum over all such functions $\chi \in H^1(\R^d)$
and using the formula \cref{eq.cap1} for the capacity, as well as the Poincar\'e inequality \cref{eq.Poincarecapa}.
\end{proof}

\noindent The second result in this section is concerned with solutions to a slight generalization of \cref{eq.chieps}, where the prescribed Dirichlet data on $\omega_\e$ is given by a function $g$ (and not constantly equal to $1$) and the conductivity $1$ is replaced by $\gamma$. More precisely, we now consider the unique solution $v_\e \in H^1(\Omega)$ to the boundary value problem:
\begin{equation}\label{eq.veps}
\left\{ 
\begin{array}{cl}
-\dv(\gamma \nabla v_\e) = 0 & \text{in } \Omega, \\
v_\e = g & \text{on } \omega_\e, \\
v_\e = 0 & \text{on } \Gamma_D, \\
\gamma\frac{\partial v_\e}{\partial n} = 0 & \text{on } \Gamma_N \setminus \overline{\omega_\e}~, 
\end{array}
\right.
\end{equation}
where $g$ is a given function in $C^1(\overline\Omega)$. By a solution to \cref{eq.veps} we understand a function $v_\e \in H^1_{\Gamma_D}(\Omega)$
such that $v_\e = g$ on $\omega_\e$ and 
\begin{equation}\label{eq.varfvedirth}
\forall v \in H^1_{\Gamma_D}(\Omega) \text{ with } v = 0 \text{ on } \omega_\e, \quad  \int_\Omega \gamma \nabla v_\e \cdot \nabla v \: \d x=0~.
\end{equation}

\begin{lemma}\label{lem.firstest}
Suppose $d=2$ or $d=3$. Let $\omega_\e$ be an open Lipschitz subset of the region $\Gamma_N \subset \partial \Omega$, 
satisfying \cref{assum.far}, and let $v_\e \in H^1(\Omega)$ be the solution to \cref{eq.veps}.
There exists a constant $M$ which depends only on $\alpha$, $\beta$, the ellipticity constants of $\gamma$, $\Omega$, $\Gamma_D$ and $\dmin$, 
but is otherwise independent of $\omega_\e$, such that
\begin{equation}\label{eq.H1estve}
 || v_\e ||_{H^1(\Omega)} \leq M || g ||_{{\mathcal C}^1(\overline \Omega)}  \capa(\omega_\e)^{\frac{1}{2}}~.
 \end{equation}
In addition, $v_\e$ satisfies the following improved $L^2$ estimate
\begin{equation}\label{eq.L2estve}
|| v_\e ||_{L^2(\Omega)}  \leq M ||g ||_{{\mathcal C}^1(\overline \Omega)} \capa(\omega_\e)^{\frac{3}{4}}~.
  \end{equation}
\end{lemma}
\begin{proof}

\noindent \textit{We first prove \cref{eq.H1estve}.} Notice that, due to (a modified version of) the Poincar\'e inequality, it suffices to show that the term $|| \nabla v_\e ||_{L^2(\Omega)^d}^2$ satisfies the desired upper bound. To this end we introduce the solution $\chi_\e \in H^1(\Omega)$ to \cref{eq.chieps}. 
Since $v_\e - g \chi_\e = 0$ on $\omega_\e\cup \Gamma_D$, the variational formulation \cref{eq.varfvedirth} yields
$$  \begin{array}{>{\displaystyle}cc>{\displaystyle}l}
 || \nabla v_\e ||_{L^2(\Omega)^d}^2 &\leq& C \int_{\Omega}{\gamma \nabla v_\e \cdot \nabla v_\e \: \d x} \\[2ex]
 %&=& C \int_{\partial \Omega}{\gamma \frac{\partial v _\e}{\partial n} v_\e \: \d s} \\[2ex]
%&=&  C \int_{\partial \Omega}{\gamma \frac{\partial v _\e}{\partial n} g \chi_\e  \: \d s} \\[2ex]
&=&C \int_{\Omega}{\gamma \nabla v_\e \cdot \nabla (g \chi_\e )\: \d x} \\[2ex]
&\leq&  C \Vert g \chi_\e  \Vert_{H^{1}(\Omega)} \Vert  v _\e \Vert_{H^{1}( \Omega)} \\[2ex]
&\leq&C \Vert g \Vert_{C^1(\overline \Omega)} \Vert \chi_\e \Vert_{H^1(\Omega)}  \Vert  v_\e \Vert_{H^1(\Omega)}~.
 \end{array}
 $$
Using the upper bound for $\Vert \chi_\e \Vert_{H^1(\Omega)}$ supplied by \cref{lem.capas}, and the Poincar\'e inequality for $v_\e$ we conclude that
$$
|| \nabla v_\e ||_{L^2(\Omega)^d} \leq M \Vert g \Vert_{C^1(\overline \Omega)} \capa(\omega_\e)^{\frac{1}{2}} ~,
$$
which is the desired estimate \cref{eq.H1estve}. \par\medskip 

\noindent \textit{We proceed to prove  \cref{eq.L2estve}.} To this end we rely on a variant of the ``classical'' Aubin-Nitsche trick \cite{aubin1967behavior,nitsche1968kriterium,ciarlet2002finite}.
Let $w_\e$ denote the unique solution in  $H^1_{\Gamma_D}(\Omega)$  to the boundary value problem
$$ 
\left\{\begin{array}{ccl}
-\dv(\gamma\nabla w_\e) = v_\e & \text{in } \Omega~, \\
w_\e = 0 & \text{on } \Gamma_D~, \\
\gamma\frac{\partial w_\e}{\partial n} = 0 & \text{on } \Gamma_N~,
\end{array}
\right.
$$
or rather, in its variational form:
$$w_\e \in H^1_{\Gamma_D}(\Omega), \text{ and } \int_{\Omega}{\gamma \nabla w_\e \cdot \nabla v\: \d x} = \int_\Omega{v_\e v \: \d x}~~ \hbox{ for all } v \in H^1_{\Gamma_D}(\Omega)~.$$
Since \cref{assum.far} holds, there exists a cut-off function $\eta \in {\mathcal C}^\infty_c(\R^d)$ with the property
$$\eta = 1 \text{ on a neighborhood of all the } \omega_\e  \text{ and } \eta = 0 \text{ on an open set } \calU \text{ of } \R^d \text{ with } \Gamma_D \Subset \calU~.$$
The key ingredient of the following derivation is that $w_\e$ shows improved regularity with respect to $v_\e$ (away from the interface between $\Gamma_D$ and $\Gamma_N$). 
In particular, standard interior elliptic regularity results, discussed e.g. in \cite{brezis2010functional,gilbarg2015elliptic}, give  
$$ || \eta w_\e ||_{H^3(\Omega)} \leq C || v_\e ||_{H^1(\Omega)}~.$$
In addition, since $d=2$ or $3$, the classical Sobolev Embedding Theorem ensures that
$$
H^{3}(\Omega) \subset C^{1}(\overline \Omega) ~~\text{ and for all } v \in H^3(\Omega), \:\: \Vert v \Vert_{C^{1}(\overline \Omega)} \le C \Vert v \Vert_{H^3(\Omega)}~,
$$
see e.g. \cite{adams2003sobolev}. 
It follows immediately from this and the previous regularity estimate that
\begin{equation}
\label{ellregwe}
\Vert  \eta w_\e  \Vert_{C^{1}(\overline \Omega)} \le C  || \eta w_\e ||_{H^3(\Omega)} \leq C || v_\e ||_{H^1(\Omega)}~.
\end{equation}
We now calculate
$$  \begin{array}{>{\displaystyle}cc>{\displaystyle}l}
\int_\Omega{v_\e ^2 \: \d x} &=&  \int_\Omega{\gamma \nabla w_\e \cdot \nabla v_\e \:\d x} \\[0.75em]
%&=&  \int_{\partial\Omega}{\gamma \frac{\partial v_\e}{\partial n} w_\e \:\d s}, \\
%&=&  \int_{\partial\Omega}{\gamma \frac{\partial v_\e}{\partial n} \chi_\e \eta w_\e \:\d s},
&=& \int_\Omega{\gamma \nabla (\chi_\e \eta w_\e) \cdot \nabla v_\e \:\d x}, 
\end{array}$$
where we have introduced the solution $\chi_\e  \in H^1(\Omega)$ to \cref{eq.chieps}, as well as the fixed cut-off function $\eta$ from above. We have also used that $w_\e- \chi_\e \eta w_\e=0$ on $\omega_\e \cup \Gamma_D$ and the variational formulation \cref{eq.varfvedirth}.
It now follows that
\begin{equation}\label{eq.estvelem2prel}
  \begin{array}{>{\displaystyle}cc>{\displaystyle}l}
|| v_\e ||_{L^2(\Omega)}^2 &\leq& C ||  v_\e ||_{H^1(\Omega)}  || \chi_\e ||_{H^1(\Omega)} || \eta w_\e||_{C^{1}(\overline \Omega)} \\
&\leq & C   ||  v_\e ||_{H^1(\Omega)}^2 || \chi_\e ||_{H^1(\Omega)}~,
\end{array}
\end{equation}
where we have employed \cref{ellregwe} for the last inequality.
Finally, using the estimate 
$$
|| \chi_\e ||_{H^1(\Omega)} \leq M \capa(\omega_\e)^{\frac12}
$$
from \cref{lem.capas}, together with the already established $H^1$ estimate \cref{eq.H1estve} for $v_\e$, it follows from \cref{eq.estvelem2prel} that
$$|| v_\e ||_{L^2(\Omega)}^2  \leq C ||g ||_{{\mathcal C}^1(\overline \Omega)}^2 \capa(\omega_\e)^{\frac{3}{2}}~, $$
as desired.
\end{proof}

\begin{remark}\label{rem.prelestdiraniso}
We observe that the conclusions of \cref{lem.firstest}, and their proofs, extend verbatim 
to the case where the scalar conductivity $\gamma$ is replaced by a smooth conductivity matrix $A \in {\mathcal C}^\infty(\overline \Omega)^{d\times d}$ satisfying the bounds
\begin{equation}\label{eq.anisomat}
 \forall \xi \in \R^d, \quad \alpha |\xi |^2  \leq A(x) \xi\cdot \xi \leq \beta |\xi|^2, \quad x \in \overline \Omega.
 \end{equation}
More precisely, the $H^1$ and $L^2$ estimates \cref{eq.H1estve,eq.L2estve} still hold true when $v_\e$ is the solution to the following anisotropic counterpart of \cref{eq.veps}
$$\left\{ 
\begin{array}{cl}
-\dv(A \nabla v_\e) = 0 & \text{in } \Omega, \\
v_\e = g & \text{on } \omega_\e, \\
v_\e = 0 & \text{on } \Gamma_D, \\
(A\nabla v_\e)\cdot n= 0 & \text{on } \Gamma_N \setminus \overline{\omega_\e}~.
\end{array}
\right.
$$
\end{remark}

%%%%%%%%%%%%%%%%%%%%%%%%%%
\subsection{The representation formula}
%%%%%%%%%%%%%%%%%%%%%%%%%%

\noindent The deviation $r_\e := u_\e - u_0$ between the perturbed potential and the background potential is the unique $H^1(\Omega)$ solution  
to 
\begin{equation}\label{eq.reps}
\left\{ 
\begin{array}{cl}
-\dv(\gamma \nabla r_\e) = 0 & \text{in } \Omega~, \\
r_\e = -u_0 & \text{on } \omega_\e~, \\
r_\e = 0 & \text{on } \Gamma_D~, \\
\gamma\frac{\partial r_\e}{\partial n} = 0 & \text{on } \Gamma_N \setminus \overline{\omega_\e}~. 
\end{array}
\right.
\end{equation}

\noindent Because of our separation assumption \cref{assum.far}, there exists a smooth compact subset $K \Subset \Gamma_N$ such that
$\omega_\e \subset K$ for all $\e$. 
Owing to local elliptic regularity estimates for the background problem \cref{eq.bg},  we have 
$$
\Vert u_0\Vert_{C^1(K)} \le C \Vert f \Vert_{H^m(\Omega)}~,
$$
for a sufficiently large integer $m$ (again, see e.g. \cite{brezis2010functional,gilbarg2015elliptic}). 
Hence, we may construct a $C^1$ function $g_0$ on all of  $\overline \Omega$ with the properties that
$$
g_0 = -u_0 \text{ on } K~~, \text{ and } \Vert g_0 \Vert_{C^1(\overline \Omega)} \le C \Vert u_0\Vert_{C^1(K)} \le C \Vert f \Vert_{H^m(\Omega)}~.
$$
With this notation, $r_\e $ is the unique $H^1(\Omega)$ solution to
$$
\left\{ 
\begin{array}{cl}
-\dv(\gamma \nabla r_\e) = 0 & \text{in } \Omega~, \\
r_\e = g_0 & \text{on } \omega_\e~, \\
r_\e = 0 & \text{on } \Gamma_D~, \\
\gamma\frac{\partial r_\e}{\partial n} = 0 & \text{on } \Gamma_N \setminus \overline{\omega_\e}~. 
\end{array}
\right.
$$

\noindent As a straightforward consequence of \cref{lem.firstest}, it follows that
\begin{equation}
\label{reh1est}
 || r_\e ||_{H^1(\Omega)} \leq C || f ||_{H^m(\Omega)}  \capa(\omega_\e)^{1/2}~,
\end{equation}
and we now search for the next term in the asymptotic expansion of $u_\e$. Our main result is the following.

\begin{theorem}\label{th.repdir}
Suppose $d=2$ or $d=3$, and suppose  $\omega_\e$ is a sequence of non-empty, open Lipschitz subsets of $\partial \Omega$, which are all contained in $\Gamma_N$ and well-separated from $\Gamma_D$ in the sense that \cref{assum.far} holds. Let $u_\e$ denote the solution to \cref{eq.uepsdir}.
Assume that  the capacity $\capa(\omega_\e)$ of $\omega_\e$ tends to $0$ as $\e \rightarrow 0$. 
Then there exists a subsequence, still labeled by $\e$, and a non-trivial distribution $\mu$ in the dual space of ${\mathcal C}^1(\partial \Omega)$, such that 
for any fixed point $x \in \Omega$, and any $\eta \in C^\infty(\partial \Omega)$ with $\eta =1$ on $\{ y \in \partial \Omega,~\dist(y,\Gamma_D)>\dmin/2\}$ and
$\eta =0$ on $\{ y \in \partial \Omega,~\dist(y,\Gamma_D)<\dmin/3\}$, it holds
\begin{equation}\label{eq.1storderdir} 
u_\e(x) = u_0(x) - \capa(\omega_\e) \mu_y\left[\eta(y) \gamma(y) u_0(y) N(x,y)\right] + \o(\capa(\omega_\e))~, ~~\hbox{ as } \e \to 0~.
\end{equation}
The term $\o(\capa(\omega_\e))$ goes to zero faster than $\capa(\omega_\e)$ uniformly for $x \in K$, where $K$ is any compact subset of $\Omega$. The distribution $\mu$ depends only on the subsequence $\omega_\e$, $\Omega$, and $\Gamma_N$.
\end{theorem}
\begin{proof}
Introducing the fundamental solution $N(x,y)$ of the background operator defined in \cref{sec.fundasol}, we obtain for any $x\in \Omega$
$$ \begin{array}{>{\displaystyle}cc>{\displaystyle}l}
r_\e(x) &=& \int_{\Omega}{r_\e(y) (-\dv_y(\gamma(y) \nabla_y N(x,y))) \: \d y} \\
&=& \int_{\Omega}{\gamma(y) \nabla r_\e(y) \cdot \nabla_y N(x,y) \: \d y} - \int_{\partial \Omega}{\gamma(y) \frac{\partial N}{\partial n_y}(x,y) r_\e(y) \: \d s(y)} \\
&=& \int_{\partial\Omega}{\gamma(y) \frac{\partial r_\e}{\partial n}(y) N(x,y) \: \d s(y)} - \int_{\partial \Omega}{\gamma(y) \frac{\partial N}{\partial n_y}(x,y) r_\e(y) \: \d s(y)}~.
\end{array}$$
Since $y \mapsto \gamma(y) \frac{\partial N}{\partial n_y}(x,y)$ vanishes on $\Gamma_N$ (i.e. as an element in $H^{-1/2}(\Gamma_N)$) and $r_\e$ vanishes on $\Gamma_D$ (i.e., $r_\e \in \widetilde{H}^{1/2}(\Gamma_N)$), the second integral in the above right-hand side equals $0$, and so
\begin{equation}\label{eq.rerepN}
 r_\e(x) =  \int_{\partial \Omega}{\gamma(y) \frac{\partial r_\e}{\partial n} (y)N(x,y) \: \d s(y)}~.
 \end{equation}
To proceed, we now use the same ``compensated compactness'', or ``clever integration by parts'' technique as in \cite{capdeboscq2003general}, see also \cite{murat2018h}.
Let $\phi \in {\mathcal C}^1(\partial \Omega)$ be an arbitrary function, which vanishes on the set $\{ y \in \partial \Omega~:~\dist(y,\Gamma_D)<\dmin/3\}$. Since $\Omega$ is smooth, it is easy to  construct a function $\psi \in {\mathcal C}^1(\overline{\Omega})$ such that
\begin{equation}\label{eq.extphidir} 
\psi = \phi \text{ on } \partial \Omega, \text{ and } || \psi ||_{{\mathcal C}^1(\overline{\Omega})} \leq C || \phi ||_{{\mathcal C}^1(\partial \Omega)}~,
\end{equation}
where the constant $C$ depends only on $\Omega$. 
As before, let $\chi_\e$ denote the solution to \cref{eq.chieps}. Since the function $(\phi- \chi_\e \psi)$ belongs to $H^{1/2}(\partial \Omega)$, and vanishes on $\Gamma_D \cup \omega_\e$, we have
$$  \int_{\partial\Omega}{\gamma \frac{\partial r_\e}{\partial n} \phi \: \d s}  =  \int_{\partial\Omega}{\gamma \frac{\partial r_\e}{\partial n}  \chi_\e \psi \: \d s}~.$$ 
An integration by parts now yields
\begin{equation}\label{eq.repdir1erest} 
 \begin{array}{>{\displaystyle}cc>{\displaystyle}l}
\int_{\partial\Omega}{\gamma \frac{\partial r_\e}{\partial n} \phi \: \d s} & =&  \int_\Omega{\gamma \nabla r_\e \cdot \nabla (\chi_\e \psi) \: \d y}\\
&=&  \int_\Omega{\psi \gamma \nabla r_\e \cdot \nabla \chi_\e \: \d y} +  \int_\Omega{\chi_\e \gamma \nabla r_\e \cdot \nabla \psi  \: \d y}~.
\end{array}
\end{equation}

\noindent Using \cref{reh1est} and the estimate \cref{eq.L2estve} applied to $\chi_\e$, we may control the second term in the above right-hand as follows 
\begin{equation}\label{eq.ctr2termrepdir}
 \begin{array}{>{\displaystyle}cc>{\displaystyle}l}
 \left\lvert  \int_\Omega{\chi_\e \gamma \nabla r_\e \cdot \nabla \psi  \: \d y} \right\lvert &\leq& C || \chi_\e ||_{L^2(\Omega)} || r_\e||_{H^1(\Omega)}  || \psi ||_{{\mathcal C}^1(\overline{\Omega})} \\
 & \leq &  C\capa(\omega_\e )^{\frac54} \Vert f\Vert_{H^m(\Omega)} \Vert \phi \Vert_{C^1(\partial \Omega)}~.
 \end{array}
 \end{equation}
A similar argument makes it possible to rewrite the first term in the right-hand side of \cref{eq.repdir1erest} as
\begin{eqnarray*}
\int_\Omega{\psi \gamma \nabla r_\e \cdot \nabla \chi_\e \: \d y} &=& \int_\Omega{ \nabla (\gamma \psi  r_\e) \cdot \nabla \chi_\e \: \d y} - \int_\Omega{  r_\e \nabla (\gamma \psi ) \cdot \nabla \chi_\e \: \d y} \\
&=&\int_\Omega{\nabla (\gamma \psi  r_\e) \cdot \nabla \chi_\e \: \d y} +\O(\capa(\omega_\e )^{\frac54}) \Vert f\Vert_{H^m(\Omega)} \Vert \phi \Vert_{C^1(\partial \Omega)}~.
\end{eqnarray*}
Inserting these two facts into \cref{eq.repdir1erest} we get
$$  
\int_{\partial\Omega}{\gamma \frac{\partial r_\e}{\partial n} \phi \: \d s} =  \int_\Omega{ \nabla (\gamma \psi r_\e) \cdot \nabla \chi_\e \: \d y}  \: + \:  \O(\capa(\omega_\e )^{\frac54}) \Vert f\Vert_{H^m(\Omega)} \Vert \phi \Vert_{C^1(\partial \Omega)}~,
$$
and so, after another integration by parts
$$  
\int_{\partial\Omega}{\gamma \frac{\partial r_\e}{\partial n} \phi \: \d s} =  \int_{\partial \Omega}{ \frac{\partial \chi_\e}{\partial n} \gamma \phi r_\e \: \d s} \: + \: \O(\capa(\omega_\e )^{\frac54}) \Vert f\Vert_{H^m(\Omega)} \Vert \phi \Vert_{C^1(\partial \Omega)}~.
$$
Since $\frac{\partial \chi_\e}{\partial n} = 0$ on $\Gamma_N \setminus \overline \omega_\e$, 
and since $r_\e = -u_0 \chi_\e$ on $\omega_\e \cup \Gamma_D$, we may replace $r_\e$ with $-u_0 \chi_\e$ in the integral of the above right-hand side, thus obtaining
\begin{equation} 
\label{eq.reapresipp}  
\int_{\partial\Omega}{\gamma \frac{\partial r_\e}{\partial n} \phi \: \d s} =  -\int_{\partial \Omega}{ \frac{\partial \chi_\e}{\partial n} \chi_\e u_0 \gamma \phi \: \d s} \: + \: \O(\capa(\omega_\e )^{\frac54}) \Vert f\Vert_{H^m(\Omega)} \Vert \phi \Vert_{C^1(\partial \Omega)}~.
\end{equation}
Now let $\eta$ be a function as introduced in the statement of this theorem: 
$$\eta \in C^\infty(\partial \Omega), \:\:  \eta =1 \text{ on } \{ y \in \partial \Omega,~\dist(y,\Gamma_D)>\dmin/2\} \text{ and }
\eta =0 \text{ on } \{ y \in \partial \Omega,~\dist(y,\Gamma_D)<\dmin/3\}.$$ 
Then, $N(x,\cdot) \eta(\cdot)$ is a $C^\infty$ function on $\partial \Omega$ which vanishes on the set $\{ y \in \partial \Omega~:~\dist(y,\Gamma_D)<\dmin/3\}$ and coincides with  $N(x,\cdot)$ on $\omega_\e \cup \Gamma_D$.  
By a combination of \cref{eq.rerepN} and  \cref{eq.reapresipp}, with $\phi(\cdot) =N(x,\cdot) \eta(\cdot)$, it follows that 
\begin{eqnarray}
\label{finaleq}
r_\e(x) &=&\int_{\partial \Omega}{\gamma(y) \frac{\partial r_\e}{\partial n} (y)N(x,y) \: \d s(y)} \nonumber \\
&=& \int_{\partial \Omega}{\gamma(y) \frac{\partial r_\e}{\partial n} (y)N(x,y)\eta(y) \: \d s(y)} \\
&=&  -\int_{\partial \Omega} { \frac{\partial \chi_\e}{\partial n}(y) \chi_\e(y) u_0(y)\gamma(y)  N(x,y) \eta(y) \: \d s(y)}+  {\mathcal O}(\capa(\omega_\e )^{\frac54}) \Vert f\Vert_{H^m(\Omega)} \Vert N(x,\cdot)\eta(\cdot) \Vert_{C^1(\partial \Omega)}~ . \nonumber
\end{eqnarray}
Finally, the upper bound in \cref{lem.capas} reveals that $\capa(\omega_\e)>0$ (since the $\omega_\e$ are non-empty), and that  for any function $\phi \in {\mathcal C}^1(\partial \Omega)$
$$ \left\lvert \frac{1}{\capa(\omega_\e)}\int_{\partial\Omega}{\frac{\partial \chi_\e}{\partial n} \chi_\e \phi \: \d s}  \right\lvert = \left\lvert \frac{1}{\capa(\omega_\e)} \int_{\Omega}{\nabla \chi_\e \cdot \nabla ( \chi_\e \phi ) \: \d y}  \right\lvert \leq  C || \phi ||_{{\mathcal C}^1(\partial\Omega)}~.$$
It follows from the Banach-Alaoglu theorem that, up to extraction of a subsequence, which we still label by $\e$, 
there exists a bounded linear functional $\mu$ on ${\mathcal C}^1(\partial \Omega)$ such that, for any $\phi \in {\mathcal C}^1(\partial\Omega)$: 
$$ \frac{1}{\capa(\omega_\e)} \int_{\partial\Omega}{\frac{\partial \chi_\e}{\partial n} \chi_\e \phi \: \d s} \xrightarrow{\e \to 0} \mu(\phi)~.$$
The lower bound in \cref{lem.capas}, in combination with Poincar\'e's inequality, reveals that $\mu(1)>0$, in other words that $\mu$ is non-trivial.
A combination of \cref{finaleq} and the above convergence result (with $\phi(\cdot) = u_0(\cdot) \gamma(\cdot) N(x,\cdot) \eta(\cdot )$) yields the desired representation formula
$$r_\e(x) = - \capa(\omega_\e) \mu_y\left[\eta(y) u_0(y)\gamma(y) N(x,y)\right] + \o(\capa(\omega_\e))~.$$
The uniformity of the convergence of the remainder term, when the point $x$ is confined to a fixed compact subset $K \Subset\Omega$, follows from the fact that the set of functions $\{ u_0(\cdot) \gamma(\cdot) N(x,\cdot) \eta(\cdot )\}_{x \in K} \subset {\mathcal C}^1(\partial \Omega)$ is compact in the ${\mathcal C}^1$ topology. 
\end{proof}

%%%%%%%%%%%%%%%%%%%%%%%%%%
\subsection{Properties of the limiting distribution $\mu$}
%%%%%%%%%%%%%%%%%%%%%%%%%%

\noindent The limiting distribution $\mu$ introduced in  \cref{th.repdir}
is a priori a distribution of order one on $\partial \Omega$, and as such it may depend on first-order derivatives of the argument function $\phi$. 
We now show that this is not the case, and that $\mu$ is actually a non negative Radon measure on $\partial\Omega$. 

\begin{proposition}\label{prop.caracmudir}
The distribution $\mu$ in \cref{eq.1storderdir} is a non-trivial, non negative Radon measure on $\partial \Omega$.
Moreover, the support of $\mu$ is contained in any compact subset $K \subset \partial \Omega$
such that $\omega_\e \subset K$ for $\e >0$ small enough.
\end{proposition}
\begin{proof}
We recall from the proof of \cref{th.repdir} that the distribution $\mu$ is defined by: 
$$ \forall \phi \in {\mathcal C}^1(\partial \Omega), \:\: \mu(\phi) = \lim\limits_{\e \to 0}{\frac{1}{\capa(\omega_\e)}} \int_{\partial \Omega}{ \frac{\partial \chi_\e}{\partial n} \chi_\e \phi \: \d s}~,$$
where the limit is taken along a subsequence, and $\chi_\e \in H^1(\Omega)$ is the solution to the equation \cref{eq.chieps}. Let $\phi$ be an arbitrary function $\phi \in {\mathcal C}^1(\partial \Omega)$. Since $\partial \Omega$ is smooth, it is easy to construct a function $\widetilde \psi \in {\mathcal C}^1(\overline \Omega)$ 
such that
\begin{equation}
\label{psidef}
\widetilde \psi = \phi \text{ on } \partial \Omega, \text{ and } || \widetilde \psi||_{C^0(\overline \Omega)} = || \phi ||_{C^0(\partial \Omega)}~.
\end{equation}
Green's formula then yields
$$
\int_{\partial\Omega}{\frac{\partial \chi_\e}{\partial n} \chi_\e \phi \: \d s} = \int_\Omega{( \nabla \chi_\e \cdot \nabla \chi_\e )\widetilde \psi \: \d x} + \int_\Omega{(\nabla \chi_\e \cdot \nabla \widetilde \psi) \chi_\e \:\d x}~.
$$
As in the proof of \cref{th.repdir} (see \cref{eq.ctr2termrepdir}), the estimates of \cref{lem.firstest} show that
$$
\lim\limits_{\e \to 0}  \frac{1}{\capa(\omega_\e)} \int_\Omega{(\nabla \chi_\e \cdot \nabla \widetilde \psi) \chi_\e \:\d x} =0~,
$$ 
and as a consequence
\begin{equation} 
\label{mupart}
\mu(\phi) =\lim\limits_{\e \to 0} \frac{1}{\capa(\omega_\e)} \int_\Omega{(\nabla \chi_\e \cdot \nabla \chi_\e) \widetilde \psi \:\d x}~,
\end{equation}
for any function $\phi \in C^1(\partial \Omega)$, where $\widetilde \psi \in C^1(\overline \Omega)$ is related to $\phi$ by \cref{psidef}. On the other hand, using \cref{lem.capas}, there exists a constant $C >0$ such that
\begin{equation}\label{eq.boundedBAdir}
\forall \psi \in {\mathcal C}^0(\overline \Omega), \:\: \frac{1}{\capa(\omega_\e)} \left\lvert \int_\Omega ( \nabla \chi_\e \cdot \nabla \chi_\e) \psi \: \d x \right\lvert \leq C || \psi ||_{{\mathcal C}^0(\overline\Omega)}. 
\end{equation}
Hence, using again the Banach-Alaoglu theorem, there exists a subsequence of the $\e$'s and a non negative Radon measure $\nu$ on $\overline\Omega$ such that
$$ \forall \psi \in {\mathcal C^0}(\overline\Omega), \:\: \frac{1}{\capa(\omega_\e)} \int_\Omega{( \nabla \chi_\e \cdot \nabla \chi_\e) \psi \: \d x} \rightarrow \int_\Omega{\psi \:\d \nu}~.$$
Combining this with \cref{mupart} we conclude that:
$$\mu(\phi) = \int_\Omega{\widetilde \psi \:\d \nu}$$
for any $\phi \in C^1(\partial \Omega)$, where $\widetilde \psi \in C^1(\overline\Omega)$ is related to $\phi$ by \cref{psidef}. Moreover,
$$
\left| \int_\Omega{\widetilde \psi \:\d \nu} \right| \le C \Vert \widetilde \psi \Vert_{C^0(\overline \Omega)} = C || \phi ||_{C^0(\partial \Omega)}, 
$$
and we have thus proved that, for any $\phi \in {\mathcal C}^1(\partial \Omega)$
$$
|\mu(\phi)| \le C || \phi ||_{C^0(\partial \Omega)}~.
$$
This shows that $\mu$ is a Radon measure on $\partial \Omega$, the non negativity of which follows from that of $\nu$. 
Moreover, the proof of \cref{th.repdir} has already revealed that $\mu$ is non trivial since $\mu(1) > 0$.

Finally, let $K \Subset \partial \Omega$ be a compact subset of $\partial \Omega$ such that $\omega_\e \subset K$ for $\e >0$ small enough. Let 
$\phi\in C^1(\partial \Omega)$ be an arbitrary function with support in the relatively open subset $U := \partial \Omega \setminus K$.
Then, $\chi_\e \phi$ belongs to $H^{1/2}(\partial \Omega)$ and vanishes on $\omega_\e \cup \Gamma_D$, so that
$$
\int_{\partial\Omega}{\frac{\partial \chi_\e}{\partial n} \chi_\e \phi \: \d s}=0~.
$$
It follows that
$$
\mu(\phi) = \lim\limits_{\e \to 0}  \frac{1}{\capa(\omega_\e)} \int_{\partial\Omega}{\frac{\partial \chi_\e}{\partial n} \chi_\e \phi \: \d s} =0~.
$$
Since this holds true for any $\phi\in C^1(\partial \Omega)$ with support in $U$, the desired result about the support of $\mu$ follows.
\end{proof}

\cref{prop.caracmudir} immediately leads to the following Corollary to \cref{th.repdir}.

\begin{corollary}\label{repdir2}
Suppose $d=2$ or $d=3$. Let $\omega_\e$ be a sequence of non-empty, open Lipschitz subsets of $\partial \Omega$, which are all contained in $\Gamma_N$ and are well-separated from $\Gamma_D$ in the sense that \cref{assum.far} holds. Let $u_\e$ denote the solution to \cref{eq.uepsdir}.
Assume furthermore that  the capacity $\capa(\omega_\e)$ of $\omega_\e$ goes to $0$ as $\e \rightarrow 0$. 
Then there exists a subsequence, still denoted by $\e$, and a non-trivial, non negative Radon measure $\mu$ on $\partial \Omega$, such that for any fixed point $x \in \Omega$
$$ 
u_\e(x) = u_0(x) -\capa(\omega_\e) \int_{\partial \Omega}{u_0(y) \gamma(y) N(x,y) \:\d \mu(y)}  + \o(\capa(\omega_\e)).
$$
The measure $\mu$ depends only on the subsequence $\omega_\e$, $\Omega$, and $\Gamma_N$.
The support of $\mu$ lies inside any compact subset $K \subset \partial \Omega$ containing the $\omega_\e$ for $\e >0$ small enough, and the term $\o(\capa(\omega_\e))$ goes to zero faster than $\capa(\omega_\e)$ uniformly (in $x$) on compact subsets of $\Omega$. 
\end{corollary}

\begin{remark}\label{rem.cply}
Let us comment about the physical meaning of the representation formula of \cref{repdir2}.
\begin{itemize}
\item The first order term in this expansion arises as the superposition of the potentials $u_0(y) \gamma(y) N(x,y)$ created at $x$ by point sources (monopoles) 
which are distributed on the ``limiting location'' of the vanishing subsets $\omega_\e$. The negative sign in front of this term indicates that these point sources have been replaced by a ``ground'' (homogeneous Dirichlet boundary condition) when passing from the background physical situation to the perturbed one. 
\item Assuming for simplicity that $f$ has compact support inside $\Omega$, the fact that the term $\o(\capa(\omega_\e))$ (in \cref{repdir2})  is uniformly small on compact subsets of $\Omega$ leads to the following asymptotic expansion for the compliance (or power consumption) of $\Omega$:
$$
\int_{\Omega}f \, u_\e \: \d x = \int_{\Omega} f \, u_0 \: \d x -\capa(\omega_\e) \int_{\Omega} f(x) \int_{\partial \Omega}{\gamma(y) u_0(y) N(x,y) \:\d \mu(y)}\: \d x  + \o(\capa(\omega_\e))~.
$$
Due to the symmetry of the fundamental solution (see \cref{sec.fundasol}), we have
$$ u_0(y) = \int_\Omega{N(x,y) f(x) \: \d x}~, $$
and this now implies
$$ \int_{\Omega}f \, u_\e \: \d x = \int_{\Omega} f \, u_0 \: \d x -\capa(\omega_\e) \int_{\partial\Omega} { \gamma(x) u_0^2(x)\: \d \mu(x)}  + \o(\capa(\omega_\e))~.
$$
In particular, the emergence of a small Dirichlet region within the homogeneous Neumann zone $\Gamma_N$ always decreases the value of the compliance, 
which is consistent with physical intuition, 
since it amounts to enlarging the region of the boundary $\partial \Omega$ where the voltage potential is grounded.
\end{itemize}
\end{remark}

% Pas s\9Er qu'on puisse esp\8Erer des bornes plus pr\8Ecises sur cette mesure, malheureusement... elle va probablement d\8Ependre de la distance de \omega_\e
% \88 l'ensemble \Gamma_D o\9D on impose Dirichlet 0. Par contre, c'est born\8E par la constante qui appara\94t dans le Lemme 3.2.

%%%%%%%%%%%%%%%%%%%%%%%%%%%%%%%%%%%%%%%%%%%%%%%%%%%%%%%
\section{Replacing Dirichlet conditions by Neumann conditions on a ``small set"}\label{sec.neu}
%%%%%%%%%%%%%%%%%%%%%%%%%%%%%%%%%%%%%%%%%%%%%%%%%%%%%%%

\noindent We presently turn to the opposite situation of that considered in \cref{sec.dir}. 
The considered sequence $\omega_\e$ of ``small'', open Lipschitz subsets of $\partial \Omega$ is now included in $\Gamma_D$, 
and it is well-separated from $\Gamma_N$ in the sense that \cref{assum.far} holds.
The homogeneous Dirichlet boundary condition satisfied by the ``background'' voltage potential $u_0$ on $\Gamma_D$ (see \cref{eq.bg})
is dropped on $\omega_\e$, where it is  replaced by a homogeneous Neumann boundary condition: the perturbed voltage potential $u_\e$ 
is then the solution to the equation \cref{eq.uepsneu}.
Like in \cref{sec.dir}, without any further assumption on $\omega_\e$, we aim to derive a representation formula for $u_\e -u_0$ as $\e\to 0$.\par\medskip

Let us start by defining the quantity $e(\omega_\e)$ which will measure the ``smallness'' of a set $\omega_\e$ in the present setting. 
When $\omega \subset \mathbb{R}^d$ is an arbitrary finite collection of disjoint Lipschitz hypersurfaces, we introduce:
\begin{equation} \label{eq.capaneu}
e(\omega) = \max\limits_{\kappa \in {\mathcal C}^\infty_c(\R^d), \atop \kappa(x) = \pm 1 \text{ for }  x\in\overline{\omega}} \Bigg\{  \int_{\R^d \setminus \overline\omega}{(z^2 + |\nabla z |^2 ) \: \d x}, \:\: z \in H^1(\R^d \setminus \overline \omega ) \text{ s.t. } \left\{
\begin{array}{cl}
- \Delta z + z = 0 & \text{in } \R^d \setminus \overline{\omega} , \\
\frac{\partial z}{\partial n} = \kappa & \text{on }\omega
\end{array}
\right. \Bigg\}~.
\end{equation}
In the above formulation, $n$ stands for any smooth unit normal vector field on (each connected component of) $\omega$, and the value of $e(\omega)$ does not depend on the choice of the particular direction(s) of $n$, due to the presence of the maximum. 
More precisely, when $\omega$ has only one connected component, $e(\omega)$ is the energy of the unique $H^1(\R^d \setminus \overline \omega)$ solution $z$ to the equation  
\begin{equation}\label{eq.funccapneu}
 \left\{
\begin{array}{cl}
- \Delta z + z = 0 & \text{in } \R^d \setminus \overline{\omega}~ , \\
\frac{\partial z}{\partial n} = 1 & \text{on }\omega~,
\end{array}
\right. 
\end{equation}
and the choice of an orientation for the normal vector $n$ to $\omega$ only affects the sign of $z$ and not the value of the energy $e(\omega)$. 
When $\omega$ has several connected components, a direction for $n$ can be set independently on each connected component of $\omega$; 
the possible choices for $\kappa$ in \cref{eq.capaneu} correspond to all possible configurations of the field $n$, and the quantity $e(\omega)$ captures the configuration with maximum energy. 

In view of the discussion in \cref{sec.capacity} (see notably \cref{rem.capadir}), it is very tempting to interpret $e(\omega)$ as a sort of ``capacity'' of the set $\omega$,
which, in a Neumann context, measures the energy of the potential in an ``equilibrium'' situation where the current passing through $\omega$ is constant, with amplitude equal to $1$. 

\begin{remark}\label{rem.explcapaneu}
In spite of its intuitive physical interpretation, the quantity $e(\omega)$ is not very explicit, 
since it involves the solution of a boundary value problem posed on the whole ambient space $\R^d$. 
For this reason, we derive in \cref{app.eom} several interesting surrogate quantities, depending only on the geometry of $\omega$, 
which in some particular cases are equivalent to $e(\omega)$. 

In \cref{sec.calcneu}, we shall conduct explicit calculations of the solution $u_\e$ to \cref{eq.uepsneu}, in the particular case where the inclusion set $\omega_\e$ is a ``surfacic ball'' on $\partial \Omega$. 
The following estimates for the ``smallness'' of the planar disk $\D_\e$ defined in \cref{eq.defDecapa}, which follow straighforwardly from \cref{app.eom} (see in particular \cref{rem.eDe}), will be used repeatedly:
\begin{equation}\label{eq.capaneuDe}
 e(\D_\e) \leq C_2 \e^2 \text{ if } d =2, \text{ and } e(\D_\e) \leq C_3 \e^3 \text{ if } d =3, 
 \end{equation}
for some universal constants $C_2$, $C_3$.
\end{remark}

%%%%%%%%%
\subsection{Preliminary estimates}
%%%%%%%%%

\noindent We start with a preliminary result, which is analogous to \cref{lem.capas},  
and is essential for the derivation of our asymptotic representation formula. Let $\zeta_\e$ be the unique $H^1(\Omega)$ solution to 
\begin{equation}\label{eq.zepsneum}
\left\{ 
\begin{array}{cl}
-\Delta \zeta_\e = 0 & \text{in } \Omega~, \\
\zeta_\e = 0 & \text{on } \Gamma_D \setminus \overline{\omega_\e}~, \\
\frac{\partial \zeta_\e}{\partial n} = 1 & \text{on } \omega_\e~,\\
\frac{\partial \zeta_\e}{\partial n} = 0 & \text{on } \Gamma_N~.
\end{array}
\right.
\end{equation}
The following lemma relates the energy of $\zeta_\e$ with the quantity $e(\omega_\e)$ defined in \cref{eq.capaneu}.  

\begin{lemma}\label{lem.zetaecapaneu}
Let $\omega_\e$ be an open Lipschitz subset of the region $\Gamma_D \subset \partial \Omega$, 
which lies ``far'' from $\Gamma_N$ in the sense that \cref{assum.far} holds,
and let $\zeta_\e$ be the solution to \cref{eq.zepsneum}.
There exist two constants $ 0 < m \leq M$, which depend only on $\Omega$, $\Gamma_D$ and the lower bound $\dmin$ on the distance from $\omega_\e$ to $\Gamma_N$, but are otherwise independent of $\omega_\e$,  such that 
 $$ m\: e(\omega_\e) \leq || \zeta_\e ||_{H^1(\Omega)}^2 \leq M\:  e(\omega_\e).$$
\end{lemma}

\begin{proof} \textit{We start by looking at the right-hand inequality.}
The latter is actually quite natural, since $\zeta_\e$ can be seen as arising from the solution $z$ to (an equation like) \cref{eq.funccapneu}, for a suitable function $\kappa$, by ``adding Dirichlet boundary conditions''. An adapted version of the Poincar\'e inequality for functions with vanishing trace on the set 
$$ \left\{ x \in \Gamma_D, \:\: \dist(x,\Gamma_N) < \dmin /3 \right\}$$ 
reveals that there exists a constant $C >0$ which only depends on $\Omega$, $\Gamma_N$ and $\dmin$ such that 
\begin{equation}\label{eq.poincareze}
 || \zeta_\e ||_{H^1(\Omega)} \leq  C || \nabla \zeta_\e ||_{L^2(\Omega)^d}~.
 \end{equation} 
Let $z_\e$ be the solution to \cref{eq.funccapneu}, 
where $n$ is the unit normal vector to $\partial \Omega$ pointing outward from $\Omega$ (in particular, it is normal to $\omega_\e$) 
and $\kappa$ constantly equals $1$ on $\omega_\e$. 
An integration by parts, using the boundary conditions satisfied by $\zeta_\e$ and $z_\e$, yields
$$  \begin{array}{>{\displaystyle}cc>{\displaystyle}l}
 || \nabla \zeta_\e ||_{L^2(\Omega)^d}^2 &=&  \int_\Omega{\nabla \zeta_\e \cdot \nabla \zeta_\e \: \d x} \\[0.75em]
 &=& \int_{\omega_\e}{\frac{\partial \zeta_\e}{\partial n} \zeta_\e \: \d s} \\[0.75em]
 &=& \int_{\omega_\e}{\frac{\partial z_\e}{\partial n} \zeta_\e \: \d s} \\[0.75em]
 &=& \int_{\partial \Omega}{\frac{\partial z_\e}{\partial n} (\eta \zeta_\e) \: \d s}~, \\
 \end{array}
 $$ 
 where $\eta$ is a smooth function such that 
 $$\eta \equiv 1 \text{ on } \left\{ x \in \Gamma_D, \:\: \dist(x,\Gamma_N) > \dmin /2 \right\} \text{ and } 
 \eta \equiv 0 \text{ on } \left\{ x \in \partial \Omega, \:\: \dist(x,\Gamma_N) < \dmin /3 \right\}~.$$
 It follows that 
 $$  \begin{array}{>{\displaystyle}cc>{\displaystyle}l}
 || \nabla \zeta_\e ||_{L^2(\Omega)^d}^2 &=&\int_\Omega{\Big( \nabla z_\e \cdot \nabla (\eta \zeta_\e) + z_\e \eta \zeta_\e \Big)\: \d x} \\
 &\leq& C ||z_\e ||_{H^1(\Omega)}   ||\zeta_\e ||_{H^1(\Omega)}\\
 &\leq& C ||z_\e ||_{H^1(\R^d \setminus \overline \omega_\e)}   ||\nabla \zeta_\e ||_{L^2(\Omega)^d}~,
  \end{array}
 $$ 
where we have used the Poincar\'e inequality \cref{eq.poincareze}. 
The desired inequality now follows from the definition \cref{eq.capaneu} of $e(\omega_\e)$ and repeated use of  the Poincar\'e inequality \cref{eq.poincareze}. 
\par\medskip
\noindent \textit{Let us now turn to the left-hand inequality.}
To this end, let $z_\e$ be the $H^1(\R^d \setminus \overline \omega_\e)$ solution to (an equation like) \cref{eq.funccapneu}, 
where $\kappa$ is any ${\mathcal C}^\infty_c(\R^d)$ function taking values $1$ or $-1$ on $\omega_\e$, and $n$ again is chosen to be the unit normal to $\partial\Omega$, pointing outward $\Omega$. 
The variational formulation associated to (an equation like) \cref{eq.funccapneu} and an integration by parts immediately imply that
$$  \begin{array}{>{\displaystyle}cc>{\displaystyle}l}
 || z_\e ||_{H^1(\R^d\setminus \overline{\omega_\e})}^2 &=&  \int_{\R^d \setminus \overline \omega_\e}{\Big( z_\e ^2 + \nabla z_\e \cdot \nabla z_\e \Big) \: \d x} \\
 &=& -\int_{\omega_\e}{\kappa (z_\e^+ - z_\e^-) \: \d s}~. \\
 \end{array}
 $$ 
Here we have denoted by $z_\e^+$ and $z_\e^-$ the one-sided traces of $z_\e$ on $\omega_\e$ from the exterior and  the interior of $\Omega$, respectively (see \cref{eq.oslimits}).
 We obtain
 $$  \begin{array}{>{\displaystyle}cc>{\displaystyle}l}
 || z_\e ||_{H^1(\R^d\setminus \overline{\omega_\e})}^2 &\leq & \int_{\omega_\e}{\frac{\partial \zeta_\e}{\partial n} |z_\e^+ - z_\e^-| \: \d s}\\
 &=& \int_{\partial \Omega}{\frac{\partial \zeta_\e}{\partial n}| z_\e^+ - z_\e^-| \: \d s}~, \\
 \end{array}
 $$ 
 where we have used the fact that $z_\e$ is continuous across $\partial \Omega$ (in the sense of traces) except on $\omega_\e$. Since $\Omega$ is smooth, there exists a bounded linear extension operator $E : H^{1/2}(\partial \Omega) \to H^1(\Omega)$ such that
 $$\forall u \in H^{1/2}(\partial \Omega), \:\: || E u ||_{H^1(\Omega)} \leq C|| u ||_{H^{1/2}(\partial \Omega)}  \text{ and } E u = u \text{ on } \partial \Omega~,$$
 for a constant $C$ which depends only on $\Omega$.
Based on the previous estimate we calculate
 $$  \begin{array}{>{\displaystyle}cc>{\displaystyle}l}
 || z_\e ||_{H^1(\R^d\setminus \overline{\omega_\e})}^2 &\leq& \int_{\partial \Omega}{\frac{\partial \zeta_\e}{\partial n}| E z_\e^+ - z_\e^-| \: \d s}\\
 &=& \int_{\Omega}{\nabla \zeta_\e \cdot \nabla | E z_\e^+ - z_\e| \: \d x}\\
 &\leq& C || \nabla \zeta_\e ||_{L^2(\Omega)^d} || Ez_\e^+ -z_\e ||_{H^1(\Omega)} \\
 &\leq& C || \nabla \zeta_\e ||_{L^2(\Omega)^d} \left(|| z_\e^+ ||_{H^{1/2}(\partial \Omega)}+ ||z_\e ||_{H^1(\Omega)}\right)\\
 &\leq& C || \nabla \zeta_\e ||_{L^2(\Omega)^d} || z_\e ||_{H^1(\R^d \setminus \overline \omega_\e)}~,
 \end{array}
 $$ 
 which finally results in the desired inequality 
 $$  || z_\e ||_{H^1(\R^d\setminus \overline{\omega_\e})} \leq C || \nabla \zeta_\e ||_{L^2(\Omega)^d}~.$$
 Since this holds for any choice of the function $\kappa \in {\mathcal C}^\infty_c(\R^d)$ having values $1$ or $-1$ on $\omega_\e$, 
 the desired inequality follows by taking the maximum with respect to any such choice.
\end{proof}

We now consider the $H^1(\Omega)$ solution $v_\e$ to the boundary value problem
\begin{equation}\label{eq.vepsneu}
\left\{ 
\begin{array}{cl}
-\dv(\gamma \nabla v_\e) = 0 & \text{in } \Omega~, \\
v_\e = 0 & \text{on } \Gamma_D \setminus \overline{\omega_\e}~, \\
\gamma\frac{\partial v_\e}{\partial n} = g & \text{on } \omega_\e~,\\
\gamma\frac{\partial v_\e}{\partial n} = 0 & \text{on } \Gamma_N~,
\end{array}
\right.
\end{equation}
where $g$ is a given ${\mathcal C}^0(\overline \Omega)$ function.
Our next result provides norm bounds for $v_\e$ in terms of the expression $e(\omega_\e)$.

\begin{lemma}\label{prop.prelestneu}
Suppose $d=2$ or $d=3$. Let $\omega_\e$ be an open Lipschitz subset of the region $ \Gamma_D \subset \partial \Omega$, which lies ``far" from $\Gamma_N$ in the sense that \cref{assum.far} holds. There exists a constant $M$, which depends only on $\alpha$, $\beta$, the coercivity constants of $\gamma$, $\Omega$, $\Gamma_N$ and the lower bound $\dmin$ on the distance from $\omega_\e$ to $\Gamma_N$, but is otherwise independent of $\omega_\e$,  such that 
the function $ v_\e$ in \cref{eq.vepsneu} satisfies the following $H^1$ estimate
\begin{equation}\label{eq.H1estneu}
 || v_\e ||_{H^1(\Omega)} \leq M e(\omega_\e)^{\frac{1}{2}} \: || g ||_{{\mathcal C}^0(\overline \Omega)}~. 
 \end{equation}
In addition, the following ``improved'' $L^2$ estimate holds
\begin{equation}\label{eq.L2estneu}
 || v_\e ||_{L^2(\Omega)} \leq M e(\omega_\e)^{\frac{3}{4}} \: || g ||_{{\mathcal C}^0(\overline \Omega)}~.
 \end{equation}
The quantity $e(\omega_\e)$ is that defined in \cref{eq.capaneu}.
\end{lemma}

\begin{proof} \textit{We start by proving \cref{eq.H1estneu}.} Since $\omega_\e$ lies inside $\Gamma_D$ with $\dist(\omega_\e,\Gamma_N)>\dmin>0$, 
a variant of the  Poincar\'e's inequality for functions whose trace vanishes on the fixed region $\{ x \in \Gamma_D, \:\: \dist(x,\Gamma_N)<\dmin \}$ yields the existence of a constant $C >0$, 
depending only on $\partial\Omega$, $\Gamma_D$ and $\dmin$, such that
\begin{equation}\label{eq.poincarerepneu}
 || v_\e ||_{H^1(\Omega)} \leq C || \nabla v_\e ||_{L^2(\Omega)^d}~.
 \end{equation}
We then calculate 
$$  \begin{array}{>{\displaystyle}cc>{\displaystyle}l}
 || \nabla v_\e ||_{L^2(\Omega)}^2 &\leq& C \int_\Omega{\gamma \nabla v_\e \cdot \nabla v_\e \: \d x}  \\[0.75em]
 &=& C \int_{\partial \Omega}{\gamma \frac{\partial v_\e}{\partial n} v_\e \: \d s} \\[0.75em]
 &=& C \int_{\omega_\e}{ g v_\e \: \d s}~.
\end{array}
$$
An application of \cref{eq.poincarerepneu} and introduction of the function $\zeta_\e$ -- defined in \cref{eq.zepsneum} and estimated in \cref{lem.zetaecapaneu} -- now yields
\begin{eqnarray*}
|| v_\e ||_{H^1(\Omega)}^2  &\leq& C || g||_{{\mathcal C^0}(\overline \Omega)}  \int_{\omega_\e} | v_\e | \: \d s \\
&=& C || g||_{{\mathcal C^0}(\overline \Omega)}  \int_{\partial \Omega} | v_\e |  \frac{\partial \zeta_\e}{\partial n}\: \d s\\
&=& C || g||_{{\mathcal C^0}(\overline \Omega)} \int_{\Omega}  \nabla | v_\e | \cdot \nabla \zeta_\e \:\d x \\
&\leq& C || g||_{{\mathcal C^0}(\overline \Omega)} \Vert v_\e\Vert_{H^1(\Omega)} e(\omega_\e)^{\frac12}~,
\end{eqnarray*}
and the desired estimate \cref{eq.H1estneu} follows.\par\medskip 

\noindent \textit{Let us now consider the improved $L^2$ estimate \cref{eq.L2estneu}. }
To establish this, we proceed along the lines of the proof of \cref{lem.firstest}. As in that proof, let $w_\e$ denote the unique $H^1(\Omega)$ solution to the boundary value problem
$$ 
\left\{\begin{array}{ccl}
-\dv(\gamma\nabla w_\e) = v_\e & \text{in } \Omega~, \\
w_\e = 0 & \text{on } \Gamma_D~, \\
\gamma\frac{\partial w_\e}{\partial n} = 0 & \text{on } \Gamma_N~.
\end{array}
\right.
$$
Taking advantage of the separation assumption \cref{assum.far}, we may introduce a cut-off function $\eta \in {\mathcal C}^\infty_c(\R^d)$ with the property
$$\eta = 1 \text{ on a fixed neighborhood of all the } \omega_\e \text{ and } \eta = 0 \text{ on an open set } \calU \text{ in } \R^d \text{ with } \Gamma_N \Subset \calU~.$$
The function $w_\e$ shows improved regularity with respect to $v_\e$, away from the interface $\Sigma$ between the Dirichlet and Neumann regions $\Gamma_D$ and $\Gamma_N$. 
More precisely, arguing as in the proof of \cref{lem.firstest} (see in particular \cref{ellregwe}), one obtains that
\begin{equation}\label{ellregweneu}
\Vert  \eta w_\e  \Vert_{C^{1}(\overline \Omega)} \le C  || \eta w_\e ||_{H^3(\Omega)} \leq C || v_\e ||_{H^1(\Omega)}~.
\end{equation}
We now calculate
$$  \begin{array}{>{\displaystyle}cc>{\displaystyle}l}
\int_\Omega{v_\e ^2 \: \d x} &=&  -\int_{\Omega}\dv(\gamma\nabla w_\e)\, v_\e \: \d x\\
&=&\int_\Omega{\gamma \nabla w_\e \cdot \nabla v_\e \:\d x} -\int_{\omega_\e} \gamma \frac{\partial w_\e}{\partial n} v_\e \:\d s \\
&=& - \int_{\omega_\e}{\gamma \frac{\partial}{\partial n} (\eta w_\e) v_\e \:\d s}~. \\
\end{array}$$
Using the regularity estimate \cref{ellregweneu} for $\eta w_\e$ and introducing the function $\zeta_\e$,
 -- defined in \cref{eq.zepsneum}, and estimated in \cref{lem.zetaecapaneu} -- we are now led to
\begin{eqnarray*}
\int_\Omega{v_\e ^2 \: \d x} &\leq& C \left\lvert\left\lvert  \frac{\partial (\eta w_\e)}{\partial n} \right\lvert\right\lvert_{C^0(\partial \Omega)}   \int_{\omega_\e}| v_\e|\: \d s  \\
&=& C \left\lvert\left\lvert  \frac{\partial (\eta w_\e)}{\partial n} \right\lvert\right\lvert_{C^0(\partial \Omega)}\int_{\partial \Omega} | v_\e |  \frac{\partial \zeta_\e}{\partial n}\: \d s\\
&=&C \left\lvert\left\lvert  \frac{\partial (\eta w_\e)}{\partial n} \right\lvert\right\lvert_{C^0(\partial \Omega)}\int_{\Omega} \nabla | v_\e | \cdot \nabla  \zeta_\e\: \d x\\
&\leq& C  || v_\e ||_{H^1(\Omega)}^2 e(\omega_\e)^{\frac12}~.
\end{eqnarray*}
In combination with the already established estimate \cref{eq.H1estneu}, this yields
$$
\int_\Omega{v_\e ^2 \: \d x} \leq\, C\, e(\omega_\e)^\frac32 || g ||^2_{{\mathcal C}^0(\overline \Omega)}~,
$$
exactly as asserted in \cref{eq.L2estneu}.
\end{proof}

\begin{remark}\label{rem.prelestneuaniso}
As in \cref{sec.dir} (see \cref{rem.prelestdiraniso}) close 
inspection of the above proof reveals that both estimates \cref{eq.H1estneu,eq.L2estneu} still hold true when the function $v_\e$ from \cref{eq.vepsneu} is replaced by the solution to the following anisotropic boundary value problem
$$\left\{ 
\begin{array}{cl}
-\dv(A \nabla v_\e) = 0 & \text{in } \Omega, \\
v_\e = 0 & \text{on } \Gamma_D  \setminus \overline{\omega_\e}, \\
(A\nabla v_\e)\cdot n= g & \text{on } \omega_\e,\\
(A\nabla v_\e)\cdot n= 0 & \text{on } \Gamma_N~,\\
\end{array}
\right.
$$
where $A \in {\mathcal C}(\overline\Omega)^{d\times d}$ is a smooth conductivity matrix satisfying the bounds \cref{eq.anisomat}.  
\end{remark}

%%%%%%%%%
\subsection{The representation formula}
%%%%%%%%%

\noindent One of our main results in this section is the following representation theorem.

\begin{theorem}\label{th.repneu}
Suppose that $d=2$ or $d=3$ and that $\omega_\e$ is a sequence of non-empty, open Lipschitz subsets of $\partial \Omega$, which are all contained in $\Gamma_D$ and well-separated from $\Gamma_N$ in the sense that \cref{assum.far} holds. Let $u_\e$ denote the solution to \cref{eq.uepsneu}. Assume that the quantity $e(\omega_\e)$, given by \cref{eq.capaneu}, goes to $0$ as $\e \rightarrow 0$. Then there exists a subsequence, still labeled by $\e$, and a non-trivial distribution $\mu$ in the dual space of ${\mathcal C}^1(\partial \Omega)$ such that 
for any fixed point $x \in \Omega$, and any $\eta \in C^\infty(\partial \Omega)$ with  $\eta =1$ on $\{ y \in \partial \Omega,~\dist(y,\Gamma_N)>\dmin/2\}$ and
$\eta =0$ on $\{ y \in \partial \Omega,~\dist(y,\Gamma_N)<\dmin/3\}$ 
 \begin{equation}\label{eq.repforneu}
 u_\e (x) = u_0(x) + e(\omega_\e) \: \mu_y\left(\eta(y) \frac{\partial u_0}{\partial n}(y) \gamma(y) \frac{\partial N}{\partial n_y} (x,y) \right) + \o (e(\omega_\e))~.
 \end{equation}
The term $\o(e(\omega_\e))$ goes to zero faster than $e(\omega_\e)$, uniformly for $x$ in any fixed compact subset $K$ of $\Omega$. The distribution $\mu$ depends only on the subsequence $\omega_\e$, $\Omega$, and $\Gamma_D$.
\end{theorem}

\begin{proof}
The proof parallels that of \cref{th.repdir} with appropriate changes. We give a fairly detailed outline of it, except in a few places where we refer back to the proof of \cref{th.repdir}. 
Let $r_\e$ denote the remainder $r_\e := u_\e - u_0$, which is now the unique $H^1(\Omega)$ solution to the following problem
\begin{equation}\label{eq.repsneu}
\left\{
\begin{array}{cl}
-\dv(\gamma \nabla r_\e) = 0 & \text{in } \Omega~, \\
r_\e = 0 & \text{in } \Gamma_D \setminus \overline{\omega_\e}~, \\
\gamma \frac{\partial r_\e}{\partial n} = -\gamma \frac{\partial u_0}{\partial n} & \text{on } \omega_\e~, \\
\gamma \frac{\partial r_\e}{\partial n} = 0 & \text{on } \Gamma_N~. \\
\end{array}
\right.
\end{equation}
Let $x$ be a fixed point inside  $\Omega$. From the definition \cref{fundsol} of the fundamental solution $N(x,y)$ to the background equation, we obtain after integration by parts
$$  \begin{array}{>{\displaystyle}cc>{\displaystyle}l}
r_\e(x) &=& \int_\Omega{r_\e(y) (-\dv_y(\gamma(y) \nabla_y N(x,y))) \: \d y}\\
&=&  -\int_{\omega_\e }{r_\e(y) \gamma (y) \frac{\partial N}{\partial n_y}(x,y) \: \d s(y)} + \int_\Omega{\gamma (y) \nabla r_\e(y) \cdot \nabla_y N(x,y) \: \d y}~.\\
\end{array}
$$
Another integration by parts of the second term in the above right-hand side reveals that the latter actually vanishes, so in conclusion
\begin{equation}\label{firsteq}
\begin{array}{>{\displaystyle}cc>{\displaystyle}l}
r_\e(x) &=&  -\int_{\omega_\e }{r_\e(y) \gamma(y) \frac{\partial N}{\partial n_y}(x,y) \: \d s(y)}\\ 
&= & -\int_{\omega_\e }{r_\e(y) \eta(y)  \gamma(y) \frac{\partial N}{\partial n_y}(x,y) \: \d s(y)} ~.
\end{array}
\end{equation}
Following the proof of \cref{th.repdir}, we now proceed to calculate, for any given function $\phi \in {\mathcal C}^1(\partial\Omega)$ vanishing on $\{ y \in \partial \Omega,~\dist(y,\Gamma_N)<\dmin/3\}$,  the limit of the quantity
$$ -\int_{\omega_\e}{r_\e(y) \phi(y) \: \d s(y)}~.$$
For this purpose we introduce an extension $\psi \in {\mathcal C}^1(\overline\Omega)$ of $\phi$ satisfying the properties (see \cref{eq.extphidir}) 
$$\psi = \phi \text{ on } \partial \Omega, \text{ and } || \psi ||_{{\mathcal C}^1(\overline{\Omega})} \leq C || \phi ||_{{\mathcal C}^1(\partial \Omega)}~,$$
and we consider the unique $H^1(\Omega)$ solution $\zeta_\e$ to the boundary value problem \cref{eq.zepsneum}.
%\begin{equation}\label{eq.zepsneu}
%\left\{
%\begin{array}{cl}
%-\dv(\gamma \nabla \zeta_\e) = 0 & \text{in } \Omega~, \\
%\zeta_\e = 0 & \text{in } \Gamma_D \setminus \omega_\e~, \\
%\gamma \frac{\partial \zeta_\e}{\partial n} = 1 & \text{on } \omega_\e~, \\
%\gamma \frac{\partial \zeta_\e}{\partial n} = 0 & \text{on } \Gamma_N~. \\
%\end{array}
%\right.
%\end{equation}
We calculate
$$   \begin{array}{>{\displaystyle}cc>{\displaystyle}l}
-\int_{\omega_\e}{r_\e(y) \phi(y) \: \d s(y)} &=& -\int_{\partial\Omega}{  \frac{\partial \zeta_\e}{\partial n} r_\e \psi \: \d s } \\
&=& -\int_{\Omega}{( \nabla \zeta_\e \cdot \nabla r_\e) \psi \: \d y} -\int_{\Omega}{( \nabla \zeta_\e \cdot  \nabla\psi) r_\e \: \d y} \\
&=& -\int_{\Omega}{(\nabla \zeta_\e \cdot \gamma \nabla r_\e) \frac{\psi}{\gamma} \: \d y}  \: + \: \O(e(\omega_\e)^{\frac54})\Vert f\Vert_{H^m(\Omega)} \Vert \phi \Vert_{C^1(\partial \Omega)}~,
\end{array}
$$
where the last identity follows from the improved $L^2$ estimate (applied to $r_\e$) and the $H^1$ estimate (applied to $\zeta_\e$) from \cref{prop.prelestneu}; 
see the proof of \cref{th.repdir} for details. A repeated use of the same estimates (with the roles of $r_\e$ and $\zeta_\e$ interchanged) followed by an integration by parts yields
$$\begin{array}{>{\displaystyle}cc>{\displaystyle}l}
 -\int_{\omega_\e}{r_\e(y) \phi(y) \: \d s(y)} &=&  -\int_{\Omega}{\gamma \nabla r_\e \cdot \nabla \left(\frac{\psi \zeta_\e}{\gamma}\right)  \: \d x}  + \O(e(\omega_\e)^{\frac54})\Vert f\Vert_{H^m(\Omega)} \Vert \phi \Vert_{C^1(\partial \Omega)} \\
 &=& - \int_{\partial \Omega}{ \frac{\partial r_\e}{\partial n} \phi \zeta_\e \: \d s} + \O(e(\omega_\e)^{\frac54})\Vert f\Vert_{H^m(\Omega)} \Vert \phi \Vert_{C^1(\partial \Omega)};
 \end{array}$$
see the proof of \cref{th.repdir}.  Using the boundary conditions satisfied by $r_\e$ and $\zeta_\e$ we finally end up with
\begin{equation}
\label{secondeq}
 -\int_{\omega_\e}{r_\e(y) \phi(y) \: \d s(y)}=  \int_{\partial\Omega}{\left( \frac{\partial \zeta_\e}{\partial n} \zeta_\e \right) \frac{\partial u_0}{\partial n} \phi \: \d s} + \O(e(\omega_\e)^{\frac54})\Vert f\Vert_{H^m(\Omega)} \Vert \phi \Vert_{C^1(\partial \Omega)}~.
\end{equation}
From \cref{lem.zetaecapaneu}, we infer that the sequence $\frac{1}{e(\omega_\e)}\left(\frac{\partial \zeta_\e}{\partial n} \zeta_\e \right)$ 
has bounded norm in the dual space of ${\mathcal C}^1(\partial \Omega)$; see more precisely \cref{eq.boundedBAdir} in the proof of \cref{th.repdir}. From the Banach-Alaoglu theorem, it now follows, after extraction of a subsequence (still labeled by $\e$), that there exists a bounded linear functional $\mu $ on ${\mathcal C}^1(\partial \Omega)$ such that
\begin{equation}
\label{thirdeq}
 \forall \varphi \in {\mathcal C}^1(\partial \Omega) , \:\: \int_{\partial \Omega}{\frac{1}{e(\omega_\e)}\left( \frac{\partial \zeta_\e}{\partial n} \zeta_\e \right) \varphi \: \d s} \xrightarrow{\e\to 0} \mu(\varphi)~. 
\end{equation} 
Also, due to \cref{lem.zetaecapaneu}, it follows that $\mu(1)>0$, thus revealing that $\mu$ is non trivial.
Insertion of $\phi(y) = \eta(y)\gamma(y) \frac{\partial N}{\partial n_y}(x,y)$ into \cref{secondeq} and application of \cref{thirdeq} with $\varphi(y)=\frac{\partial u_0}{\partial n}(y)\eta(y)\gamma(y) \frac{\partial N}{\partial n_y}(x,y)$ now gives
$$
-\int_{\omega_\e }{r_\e(y) \eta(y)  \gamma(y) \frac{\partial N}{\partial n_y}(x,y) \: \d s(y)} = e(\omega_\e) \mu_y \left[  \eta(y) \frac{\partial u_0}{\partial n}(y) \gamma(y) \frac{\partial N}{\partial n_y}(x,y)\right] + \o(e(\omega_\e))~,
$$
which in combination with \cref{firsteq} leads to the desired representation formula \cref{eq.repforneu}. The uniformity of the convergence of the remainder $\o(e(\omega_\e))$, when $x$ is confined to a fixed compact subset of $\Omega$, follows as in the proof of \cref{th.repdir}.
\end{proof}

Just as in \cref{sec.dir} we may show that the distribution $\mu$ is a non negative Radon measure compactly supported ``near" the sets  $\omega_\e$; in other words, the following analogue of \cref{prop.caracmudir} holds in the present context, whose nearly identical proof is left to the reader.
\begin{proposition}
The limiting distribution $\mu$ introduced in \cref{th.repneu} is a non negative Radon measure on $\partial \Omega$.
Moreover, the support of $\mu$ is contained in any compact subset $K$ of $\partial \Omega$ such that $\omega_\e \subset K$ for $\e >0$ sufficiently small. 
\end{proposition}

This proposition immediately leads to the following corollary to \cref{th.repneu}. 
\begin{corollary}\label{repneu2}
Suppose $d=2$ or $d=3$ and suppose $\omega_\e$ is a sequence of non-empty, open Lipschitz subsets of $\partial \Omega$, which are all contained in  $\Gamma_D$ and are well-separated from $\Gamma_N$, in the sense that \cref{assum.far} holds; let $u_\e$ denote the solution to \cref{eq.uepsneu}.
Assume that  the quantity $e(\omega_\e)$, defined by \cref{eq.capaneu}, goes to $0$ as $\e \rightarrow 0$. 
Then there exists  a subsequence, still labeled by $\e$, and a non-trivial, non negative Radon measure $\mu$ on $\partial \Omega$, whose support is included in any compact subset $K \subset \partial \Omega$ containing the $\omega_\e$ for $\e >0$ small enough, such that the following asymptotic expansion 
$$ 
u_\e(x) = u_0(x) +e(\omega_\e) \int_{\partial \Omega} \frac{\partial u_0}{\partial n}(y) \gamma(y) \frac{\partial N}{\partial n_y} (x,y) \:\d \mu(y)  + \o(e(\omega_\e))~,
$$
holds at any fixed point $x \in \Omega$.
The term $\o(e(\omega_\e))$ goes to zero faster than $e(\omega_\e)$ uniformly (in $x$) on compact subsets of $\Omega$. The measure $\mu$ depends only on the subsequence $\omega_\e$, $\Omega$, and $\Gamma_D$.
\end{corollary}
\begin{remark}
\noindent 
From the physical viewpoint, the second term in the representation formula of \cref{repneu2} accounts for the potential created at $x$ by a distribution of dipoles located at the ``limiting position'' of the sets $\omega_\e$. We notice the sign change, when compared to the second term of the expansion in \cref{sec.dir}. A calculation similar to that found in \cref{rem.cply} (and under the same assumptions regarding the source term $f$) now leads to a non negative first term in the perturbation of the compliance, reflecting the intuitive fact that the compliance of $\Omega$ necessarily (asymptotically) increases when the homogeneous Dirichlet boundary condition on $\omega_\e$  is turned into a homogeneous Neumann condition.
\end{remark}

%%%%%%%%%%%%%%%%%%%%%%%%%%%%%%%%%%%%%%%%%%%%%%%%%%%%%%%
\section{An explicit asymptotic formula for the case of substituting Dirichlet conditions}\label{sec.calcdir}
%%%%%%%%%%%%%%%%%%%%%%%%%%%%%%%%%%%%%%%%%%%%%%%%%%%%%%%

\noindent In this section, we investigate a particular instance of the general situation of \cref{sec.dir}, where
the homogeneous Neumann boundary condition satisfied by the background potential $u_0$ on the whole region $\Gamma_N$ is modified to a Dirichlet boundary condition 
on a subset $\omega_\e \subset \Gamma_N$ taking the form of a vanishing ``surfacic ball''.

Without loss of generality, we assume that the origin $0$ belongs to $\Gamma_N$, 
and that the normal vector $n(0)$ at $0$ coincides with the last coordinate vector $e_d$. 
We select a smooth bounded domain ${\mathcal O} \subset \R^d$, and construct a smooth diffeomorphism $T: \R^d \to \R^d$ such that $\Omega = T({\mathcal O})$, and 
$$
\begin{minipage}{0.93\textwidth}
\begin{enumerate}[(i)]
\item The domain ${\mathcal O}$ lies inside the lower half-space $H$, and it coincides with $H$ in a fixed open neighborhood $U$ of $0$: 
$${\mathcal O} \subset H, \text{ and } {\mathcal O}\cap U= H \cap U, \text{ where } H := \left\{ x=(x_1,\ldots, x_{d}) \in \R^d , \:\: x_d < 0\right\}.$$ 
\item $T(0) = 0$ and $\nabla T(0) = \Id$. 
\end{enumerate}
\end{minipage}
$$
Given such $T$ and ${\mathcal O}$, the subset $\omega_\e \subset \Gamma_N$ is now defined as follows:
\begin{equation}\label{eq.GNflat}
 \omega_\e = T (\D_\e), \text{ where } \D_\e := \left\{ x = (x_1,\ldots, x_{d-1}, 0) \in \partial H, \:\: |x| < \e \right\},
 \end{equation}
for $\e$ sufficiently small, see \cref{fig.mapflat} for an illustration.
We denote by $\widehat{\Gamma_N}\subset \partial {\mathcal O}$ the boundary set $\widehat{\Gamma_N} := T^{-1}(\Gamma_N)$, and purely for simplicity we also assume that ${\mathcal O}$ and $T$  are selected in such aa way that $T$ coincides with the identity mapping ``far'' from $0$, so that in particular $T^{-1}(\Gamma_D) = \Gamma_D$ (in terms of the original domain $\Omega$ this is achievable through the assumption that $\Omega$ lies below its tangent plane at $0$).

\begin{figure}[!ht]
\centering
\begin{minipage}{1.0\textwidth}
\includegraphics[width=1.0\textwidth]{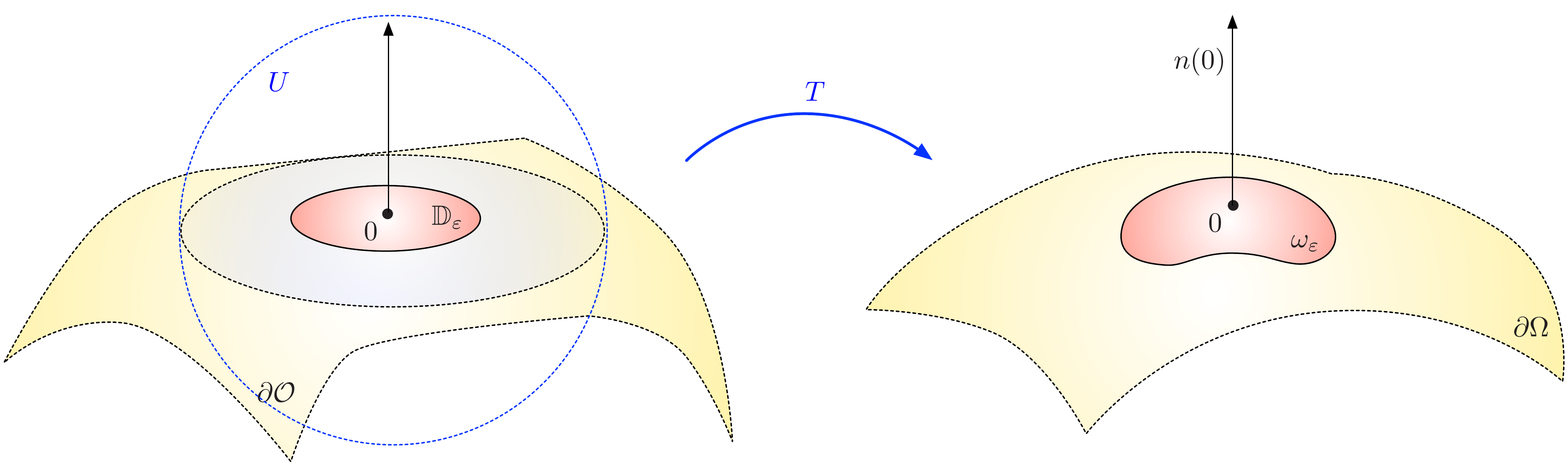}
\end{minipage} 
\caption{\it The setting in \cref{sec.calcdir}.}
\label{fig.mapflat}
\end{figure}

The ``background'' and perturbed potentials $u_0$ and $u_\e$ are the $H^1(\Omega)$ solutions to the following equations:
\begin{equation}\label{eq.ueu0si}
\left\{ 
\begin{array}{cl}
-\dv(\gamma \nabla u_0) = f & \text{in } \Omega, \\
u_0 = 0 & \text{on } \Gamma_D, \\
\gamma\frac{\partial u_0}{\partial n} = 0 & \text{on } \Gamma_N,
\end{array}
\right.
\text{ and } 
 \left\{
\begin{array}{cl}
-\dv(\gamma \nabla u_\e) = f & \text{in } \Omega, \\
u_\e = 0 & \text{on } \Gamma_D \cup \omega_\e, \\
\gamma\frac{\partial u_\e}{\partial n} = 0 & \text{on } \Gamma_N \setminus \overline{\omega_\e},
\end{array}
\right.
\end{equation}
where the source term $f \in {\mathcal C}^\infty(\overline\Omega)$ is smooth.
Invoking classical elliptic regularity results, we observe that $u_0$ and $u_\e$ are smooth, 
except in the vicinity of the points $x \in \partial \Omega$ where boundary conditions change type. More precisely, with $\Sigma= \overline{\Gamma_D} \cap \overline{\Gamma_N}$
\begin{itemize}
\item The function $u_0$ is of class ${\mathcal C}^\infty$ in a neighborhood of any point $x \in \overline\Omega \setminus \Sigma$; 
\item The function $u_\e$ is of class ${\mathcal C}^\infty$ in a neighborhood of any point $x \in \overline\Omega \setminus \left( \Sigma \cup \partial \omega_\e \right) $;
\end{itemize}
see for instance \cite{brezis2010functional}, \S 9.6, or \cite{gilbarg2015elliptic}.

We aim to derive a precise first order asymptotic expansion of $u_\e$ when $\e \to 0$, thus exemplifying the abstract structure of \cref{th.repdir}.
We start by providing the complete analysis for the two-dimensional case in \cref{sec.dir2dasym}.
The analysis for the three-dimensional case, which is quite similar in many aspects, is outlined in \cref{sec.dir3dasym}. 

%%%%%%
\subsection{Asymptotic expansion of the perturbed potential $u_\e$ in 2d}\label{sec.dir2dasym}
%%%%%%

\noindent This section deals with the case $d=2$, and our main result is 

\begin{theorem}\label{th.expdir2d}
The following asymptotic expansion holds at any point $x\in \overline{\Omega}$, $x\notin \Sigma \cup \left\{ 0 \right\}$: 
\begin{equation}\label{eq.expuedir2d}
 u_\e(x) = u_0(x) - \frac{\pi}{| \log \e|} \gamma(0) u_0(0) N(x,0) + \o\left(\frac{1}{|\log \e|}\right).
 \end{equation}
\end{theorem}

\begin{proof}
We proceed in four steps, relying on several intermediate, technical results, whose proofs are postponed to the end of the section
for the sake of clarity. \par\medskip 

\noindent \textit{Step 1.} We establish a representation formula for $u_\e$ which relates 
its value at a point $x \in \partial \Omega$ ``far'' from the inclusion set $\omega_\e$ to its values inside $\omega_\e$ by means of the fundamental solution $N(x,y)$ 
to the background operator, defined by \cref{fundsol}.

Considering a fixed point $x \in \Omega$, using the definition of $N(x,y)$ and integrating by parts twice, we obtain 
$$\begin{array}{>{\displaystyle}cc>{\displaystyle}l}
 u_\e(x) &=& - \int_\Omega{\dv_y(\gamma(y) \nabla_y N(x,y)) u_\e(y) \: \d y} \\
 &=&  \int_{\Omega}{\gamma(y) \nabla_y N(x,y) \cdot \nabla u_\e(y) \: \d y} \\[0.7em]
 &=& \int_{\partial \Omega}{\gamma(y)\frac{\partial u_\e}{\partial n_y}(y) N(x,y)\:\d s(y)} + \int_{\Omega}{f(y) N(x,y) \:\d y},
 \end{array}$$
 where the second line follows from the facts that $\gamma(y) \frac{\partial N}{\partial n_y}(x,y) = 0$ on $\Gamma_N$ and $u_\e = 0$ on $\Gamma_D$; see \cref{eq.ueu0si}. 
Using the background problem \cref{eq.ueu0si} satisfied by $u_0$ and the boundary conditions for $u_\e$ and $N$, we 
 arrive at the following formula, for any point $x \in \Omega$ 
 \begin{equation}\label{eq.repformulauesi0} 
 \begin{array}{>{\displaystyle}cc>{\displaystyle}l}
 u_\e(x) &=& u_0(x) + \int_{\omega_\e}{\gamma(y)\frac{\partial u_\e}{\partial n_y}(y) N(x,y)\:\d s(y)}\\
 &=&u_0(x) + \int_{\omega_\e}{\gamma(y)\frac{\partial r_\e}{\partial n_y}(y) N(x,y)\:\d s(y)}~.
 \end{array}
 \end{equation}
Here we have taken advantage of the fact that $u_0$ satisfies homogeneous Neumann boundary conditions on $\Gamma_N$
 to introduce the error $r_\e= u_\e - u_0$ in the last integral of the above right-hand side.
Note that
%and of \cref{article-rem.otherkernel}
the identity \cref{eq.repformulauesi0} 
 extends to the case of points $x \in \partial \Omega$, $x \notin \Sigma \cup \left\{0\right\}$ in the sense of traces, provided $\e$ is small enough, since all the quantities involved are smooth at such points.
 
Next, we introduce the mapped potentials $v_0 := u_0 \circ T$ and $v_\e := u_\e \circ T$. 
A change of variables in the variational formulations of \cref{eq.ueu0si} reveals that $v_0$ and $v_\e$ are the unique $H^1({\mathcal O})$ solutions to the problems
 \begin{equation}\label{eq.vev0si}
\left\{ 
\begin{array}{cl}
-\dv(A \nabla v_0) = g & \text{in } {\mathcal O}, \\
v_0 = 0 & \text{on } \Gamma_D, \\
(A \nabla v_0) \cdot n = 0 & \text{on } \widehat{\Gamma_N},
\end{array}
\right.
\text{ and } 
 \left\{
\begin{array}{cl}
-\dv(A \nabla v_\e) =  g & \text{in } {\mathcal O}, \\
v_\e = 0 & \text{on } \Gamma_D \cup \D_\e, \\
(A\nabla v_\e)\cdot n = 0 & \text{on } \widehat{\Gamma_N} \setminus \overline{\D_\e},
\end{array}
\right.
\end{equation}
where $g \in {\mathcal C}^\infty(\R^2)$ and $A \in {\mathcal C}^\infty(\R^2 ,\R^{2\times 2})$ are the smooth function and the matrix field defined by 
\begin{equation}\label{eq.defgA}
 g(y) = |\det (\nabla T(y))| f(T(y))~, \text{ and } A(y) = |\det (\nabla T(y))| \gamma(T(y)) \nabla T(y)^{-1} (\nabla T(y)^t)^{-1}.
\end{equation}

Recalling the definition \cref{eq.GNflat} of $\omega_\e$, we now change variables in \cref{eq.repformulauesi0} and then rescale the resulting integral to obtain
 \begin{equation}\label{eq.repformulauesi} 
 \begin{array}{>{\displaystyle}cc>{\displaystyle}l}
 u_\e(x) &=& u_0(x) + \int_{\D_\e}{(A(y) \nabla s_\e(y))\cdot n(y) N(x,T(y))\:\d s(y)}\\
 &=&u_0(x) + \int_{\D_1}{ \varphi_\e(z) N(x,T(\e z))\:\d s(z)}~,
 \end{array}
 \end{equation}
where we have introduced $s_\e := v_\e - v_0 = r_\e \circ T$, and the quantity 
\begin{equation}\label{eq.vphie2ddir}
\varphi_\e(z) = \e (A\nabla s_\e \cdot n)(\e z) \in \widetilde{H}^{-1/2}(\D_1)~.
\end{equation} 
The formula \cref{eq.repformulauesi} leads us to study the asymptotic behavior of $\varphi_\e$ as $\e \to 0$.
 \par\medskip 
 
 \noindent \textit{Step 2.} We characterize $\varphi_\e$ as the solution to an integral equation. 
 To this end, we essentially repeat the derivation of Step 1,
 except that we now use an approximate, explicit fundamental solution instead of the function $N(x,y)$.  
 
 For any symmetric, positive definite matrix $A$, and any $x\in H$, let $L_A(x,y)$ be a solution to the following equation posed on the lower half-space $H$: 
 \begin{equation}\label{eq.defLB}
\left\{
\begin{array}{cl}
-\dv_y(A \nabla_y L_A(x,y)) = \delta_{y=x}  & \text{in } H~,  \\
A\nabla_y L_A (x,y) \cdot n(y) = 0 & \text{on } \partial H~.
\end{array}
 \right.
\end{equation}
The next lemma provides an explicit expression for such a function; its proof is postponed to the end of the present section. 

\begin{lemma}\label{lem.LB}
Let $A$ be a symmetric, positive definite $2 \times 2$ matrix, and let $M := A^{-1/2}$.
Let $G(x,y)$ be the fundamental solution of the operator $-\Delta$ in the free space, defined in \cref{eq.GreenLap}.
The function
\begin{equation}\label{eq.exprLB}
  L_A(x,y) = |\det M| \left( G(Mx,My) + G\left(Mx, My - 2 y_2 \frac{M^{-1}e_2}{|M^{-1}e_2|^2}\right)\right), \quad x \neq y,
 \end{equation}
 satisfies the equation \cref{eq.defLB}.
\end{lemma}
\begin{remark}\label{rem.LAnorm}
A straightforward calculation shows that,  for $x \in \partial H$, $y \in \R^2$, $y\neq x$, 
$$ | Mx - My | = \left\lvert Mx - My + 2 y_2 \frac{M^{-1}e_2}{|M^{-1}e_2|^2} \right\lvert~.$$
\end{remark}
For a given point $x\in H$, we now consider the function $y \mapsto L_{A(x)}(x,y)$ (by substituting $A(x)$ for $A$ in  \cref{eq.exprLB}) which satisfies
 \begin{equation}\label{eq.defLAx}
\left\{
\begin{array}{cl}
-\dv_y(A(x) \nabla_y L_{A(x)}(x,y)) = \delta_{y=x}  & \text{in } H,  \\
A(x)\nabla_y L_{A(x)} (x,y) \cdot n(y) = 0 & \text{on } \partial H.
\end{array}
 \right.
\end{equation}
For a point $x \in {\mathcal O}$, we obtain from \cref{eq.defLAx} and integration by parts that
 $$\begin{array}{>{\displaystyle}cc>{\displaystyle}l}
 v_\e(x) &=& - \int_{\mathcal O}{\dv_y(A(x) \nabla_y L_{A(x)}(x,y)) v_\e(y) \: \d y} \\[0.7em]
 &=& - \int_{\partial {\mathcal O}}{ A(x) \nabla_y L_{A(x)}(x,y) \cdot n(y) v_\e(y)\:\d s(y) } + \int_{{\mathcal O}}{A(x) \nabla_y L_{A(x) }(x,y) \cdot \nabla v_\e(y) \: \d y} \\[0.7em]
 &=& - \int_{\partial {\mathcal O} \setminus {U}}{ A(x) \nabla_y L_{A(x)}(x,y) \cdot n(y) v_\e(y)\:\d s(y) }   + \int_{{\mathcal O}}{(A(x)-A(y)) \nabla_y L_{A(x) }(x,y) \cdot \nabla v_\e(y) \: \d y} \\ 
 && + \int_{{\mathcal O}}{A(y) \nabla_y L_{A(x) }(x,y) \cdot \nabla v_\e(y) \: \d y} \\[0.7em]
  &=& - \int_{\partial {\mathcal O} \setminus {U}}{ A(x) \nabla_y L_{A(x)}(x,y) \cdot n(y) v_\e(y)\:\d s(y) }   + \int_{{\mathcal O}}{(A(x)-A(y)) \nabla_y L_{A(x) }(x,y) \cdot \nabla v_\e(y) \: \d y} \\ 
 && + \int_{\partial {\mathcal O}}{(A\nabla v_\e \cdot n)(y) \: L_{A(x) }(x,y) \: \d s(y)} + \int_{{\mathcal O}}{ g(y) L_{A(x) }(x,y) \: \d y}~,
 \end{array}$$
 where the third line follows from the equation \cref{eq.defLAx} satisfied by $L_{A(x)}(x,y)$, and the fact that $\partial {\mathcal O} \cap U = \partial H\cap U$. A similar calculation applied to $v_0$ instead of $v_\e$ yields 
\begin{multline*}
v_0(x) = - \int_{\partial {\mathcal O} \setminus {U}}{ A(x) \nabla_y L_{A(x)}(x,y) \cdot n(y) v_0(y)\:\d s(y) }   + \int_{{\mathcal O}}{(A(x)-A(y)) \nabla_y L_{A(x) }(x,y) \cdot \nabla v_0(y) \: \d y} \\ 
 \hskip -50pt + \int_{\partial {\mathcal O}}{(A\nabla v_0 \cdot n)(y) \: L_{A(x) }(x,y) \: \d s(y)} + \int_{{\mathcal O}}{ g(y) L_{A(x) }(x,y) \: \d y}~.
 \end{multline*}
Forming the difference of these identities, we get 
 \begin{multline*}
s_\e(x) = - \int_{\partial {\mathcal O} \setminus {U}}{ A(x) \nabla_y L_{A(x)}(x,y) \cdot n(y) s_\e(y)\:\d s(y) }   + \int_{{\mathcal O}}{(A(x)-A(y)) \nabla_y L_{A(x) }(x,y) \cdot \nabla s_\e(y) \: \d y} \\ 
 \hskip -50pt + \int_{\Gamma_D}{(A\nabla s_\e \cdot n)(y) \: L_{A(x) }(x,y) \: \d s(y)} + \int_{\D_\e}{(A\nabla s_\e \cdot n)(y) \: L_{A(x) }(x,y) \: \d s(y)}~.
  \end{multline*}
Letting $x$ tend to $\D_\e$, and invoking the boundary continuity of single layer potentials (as in the last term), we obtain for a.e. $x \in \D_\e$
   \begin{eqnarray}\label{eq.v0De}
   -v_0(x) = - \int_{\partial {\mathcal O} \setminus {\mathcal U}}{ A(x) \nabla_y L_{A(x)}(x,y) \cdot n(y) s_\e(y)\:\d s(y) }   + \int_{{\mathcal O}}{(A(x)-A(y)) \nabla_y L_{A(x) }(x,y) \cdot \nabla s_\e(y) \: \d y}\nonumber \\ 
+ \int_{\Gamma_D}{(A\nabla s_\e \cdot n)(y) \: L_{A(x) }(x,y) \: \d s(y)} + \int_{\D_\e}{(A\nabla s_\e \cdot n)(y) \: L_{A(x) }(x,y) \: \d s(y)}~.
  \end{eqnarray}
Rescaling the above equation, we finally obtain, for a.e. $x \in \D_1$  
   \begin{multline*}
   -v_0(\e x) = - \int_{\partial {\mathcal O} \setminus {\mathcal U}}{ A(\e x) \nabla_y L_{A(\e x)}(\e x,y) \cdot n(y) s_\e(y)\:\d s(y) }   + \int_{{\mathcal O}}{(A(\e x)-A(y)) \nabla_y L_{A(\e x) }(\e x,y) \cdot \nabla s_\e(y) \: \d y} \\ 
+ \int_{\Gamma_D}{(A\nabla s_\e \cdot n)(y) \: L_{A(\e x) }(\e x,y) \: \d s(y)} + \int_{\D_1}{ \varphi_\e(z) \: L_{A(\e x) }(\e x,\e z) \: \d s(z)},
  \end{multline*}
  where $\varphi_\e \in \widetilde{H}^{-1/2}(\D_1)$ is the function \cref{eq.vphie2ddir} introduced in the course of Step 1.
  We recast this equation in the form 
  \begin{equation}\label{eq.inteq}
   T_\e \varphi_\e = -u_0(0) + \eta_\e~,
   \end{equation}
  where $T_\e: \widetilde{H}^{-1/2}(\D_1) \to H^{1/2}(\D_1) $ is the integral operator defined by 
  \begin{equation}\label{eq.defTe2ddir}
   T_\e \varphi (x) =  \int_{\D_1}{ \varphi(z) \: L_{A(\e x) }(\e x,\e z) \: \d s(z)}~,  
   \end{equation}
  and where the remainder $\eta_\e \in H^{1/2}(\D_1)$ is given by
  \begin{multline}\label{eq.defetae} 
 \eta_\e(x) = \Big( u_0(0) - v_0(\e x) \Big)  +  \int_{\partial {\mathcal O} \setminus {\mathcal U}}{ A(\e x) \nabla_y L_{A(\e x)}(\e x,y) \cdot n(y) s_\e(y)\:\d s(y) }  \\
 -  \int_{{\mathcal O}}{(A(\e x)-A(y)) \nabla_y L_{A(\e x) }(\e x,y) \cdot \nabla s_\e(y) \: \d y}  - \int_{\Gamma_D}{(A\nabla s_\e \cdot n)(y) \: L_{A(\e x) }(x,y) \: \d s(y)} ~.
 \end{multline}

 \par\medskip  
 \noindent \textit{Step 3.} We infer the asymptotic behavior of $\varphi_\e$ from the analysis of the integral equation \cref{eq.inteq}.
 The key ingredients in this direction are the next two lemmas; for clarity, their proofs are postponed to the end of this section.
 \begin{lemma}\label{lem.etae}
 The quantity $\eta_\e \in H^{1/2}(\D_1)$, defined in \cref{eq.defetae}, satisfies
 \begin{equation}\label{eq.cvetaeps}
  \eta_\e \xrightarrow{\e \to 0} 0 \text{ strongly in } H^{1/2}(\D_1)~.
  \end{equation}
 \end{lemma}
 
 \begin{lemma}\label{lem.Teps}
 The following asymptotic expansion holds 
 $$ \sup\limits_{\varphi \in \widetilde{H}^{-1/2}(\D_1) \atop || \varphi ||_{ \widetilde{H}^{-1/2}(\D_1)} \leq 1} \left\lvert\left\lvert \: T_\e \varphi - \frac{1}{\pi \gamma(0)} (  |\log \e|  + \alpha) \int_{\D_1}{\varphi \: \d s} - \frac{2}{\gamma(0)} S_1 \varphi \: \right\lvert\right\lvert_{H^{1/2}(\D_1)} \xrightarrow{\e \to 0}0~, $$
 where $\alpha= \frac{1}{2} \log \gamma(0)$ and $S_1 : \widetilde{H}^{-1/2}( \D_1) \to H^{1/2}(\D_1)$ is the self-adjoint operator defined by
\begin{equation}\label{eq.S1}
S_1\varphi(x) = - \frac{1}{2\pi}\int_{\D_1}{\log|x-y| \varphi (y) \: \d s(y)}, \:\: x\in \D_1~.
\end{equation}
 \end{lemma}\par\medskip
 
By use of these results, the integral equation \cref{eq.inteq} may be rewritten 
 \begin{equation}\label{eq.inteqSe2}
(|\log \e | + \alpha) \langle \varphi_\e,1 \rangle + 2\pi S_1 \varphi_\e + R_\e \varphi_\e= - \pi \gamma(0) u_0(0) + \pi \gamma(0) \eta_\e~,
  \end{equation}
  where $\eta_\e$ converges to $0$ strongly in $H^{1/2}(\D_1)$ and $R_\e: \widetilde{H}^{-1/2}(\D_1) \to H^{1/2}(\D_1)$ 
  is a sequence whose operator norm converges to $0$. The study of the approximate version \cref{eq.inteqSe2} of our integral equation \cref{eq.inteq}
is based on yet another lemma, whose proof is also postponed.
\begin{lemma}\label{lem.S1}
\noindent \begin{enumerate}[(i)]
\item The operator $S_1 : \widetilde{H}^{-1/2}( \D_1) \to H^{1/2}(\D_1)$ 
is invertible. 
\item For $\e >0$ small enough, the operator $V_\e : \widetilde{H}^{-1/2}(\D_1) \to H^{1/2}(\D_1)$, defined by
$$ V_\e \varphi = (|\log \e |+  \alpha ) \langle \varphi, 1 \rangle + 2\pi S_1\varphi~,$$
is invertible with the uniformly bounded inverse 
\begin{equation}\label{eq.Vem1}
 V_\e^{-1} g = \frac{1}{2\pi}S^{-1}_1 g - \frac{(|\log \e| + \alpha) \: \langle S_1^{-1} g ,1 \rangle }{2\pi + ( |\log \e| +\alpha)  \:  \langle S_1^{-1} 1 ,1 \rangle} \frac{1}{2\pi}S_1^{-1} 1~,
 \end{equation}
$\langle S_1^{-1} 1 ,1 \rangle = 2\pi/\log 2$. In particular, 
\begin{equation}\label{eq.moyVe}
 \langle V_\e^{-1} g, 1\rangle = \frac{ \langle S_1^{-1}g ,1 \rangle}{2\pi+(|\log \e| +\alpha) \langle S_1^{-1} 1, 1 \rangle}~.
 \end{equation}
\end{enumerate}
\end{lemma}

Since the operator norm of $R_\e: \widetilde{H}^{-1/2}(\D_1) \to H^{1/2}(\D_1)$ tends to $0$, 
invoking \cref{eq.inteqSe2}, \cref{lem.S1}, 
and a Neumann series for the solution of \cref{eq.inteqSe2}, 
we see that the function $\varphi_\e \in \widetilde{H}^{-1/2}(\D_1)$ satisfies 
$$ \varphi_\e = -\pi \gamma(0) u_0(0) V_\e^{-1} 1 + V_\e^{-1} \widetilde{\eta_\e}~,$$
for a sequence $\widetilde{\eta_\e}$ converging to $0$ strongly in $H^{1/2}(\D_1)$.
In particular, there exists a constant $C$ such that
\begin{equation}\label{eq.stabphie}
|| \varphi_\e ||_{\widetilde{H}^{-1/2}(\D_1)} \leq C~.
\end{equation}
Moreover, using \cref{eq.moyVe}, we calculate
\begin{equation}\label{eq.moyphie}
 \begin{array}{>{\displaystyle}cc>{\displaystyle}l}
\langle \varphi_\e,1\rangle &=& -\pi \gamma(0) u_0(0)  \langle V_\e^{-1} 1, 1 \rangle + \langle V_\e^{-1} \eta_\e , 1 \rangle\\
&= & -\frac{ \pi}{|\log \e|} \gamma(0) u_0(0) +  \o\left( \frac{1}{|\log\e |} \right)~,
\end{array}
\end{equation}
which is the needed information about $\varphi_\e$ for the following Step 4.
  \par\medskip

 \noindent \textit{Step 4.} We pass to the limit in the initial representation formula \cref{eq.repformulauesi}.
Since $x$ does not belong to $\Sigma \cup \{ 0 \}$, we obtain from the estimate \cref{eq.stabphie} and a Taylor expansion of the smooth function $y \mapsto N(x,y)$ in a neighborhood of $0$
$$
 \begin{array}{>{\displaystyle}cc>{\displaystyle}l}
\left\lvert \int_{\D_1}{\varphi_\e(z) (N(x,T(\e z)) - N(x,0)) \d s(z)}  \right\lvert & \leq & || \varphi_\e||_{\widetilde{H}^{-1/2}(\D_1)}  || N(x,T(\e \cdot)) - N(x,0) ||_{H^{1/2}(\D_1)} \\  
&\leq & C\e~, 
\end{array}
$$
and so
$$ u_\e(x) = u_0(x) + \left( \int_{ \D_1 }{ \varphi_\e (z) \d s(z)}\right) N(x,0) +  \O(\e)~.$$
Hence the desired expansion \cref{eq.expuedir2d} follows from \cref{eq.moyphie}. 
 \end{proof}
 \par\bigskip
 
 We now provide the proofs of the missing links in the above analysis. 
 \begin{proof}[Proof of \cref{lem.LB}]
 We seek a function $L_A(x,y)$ that satisfies, for any point $x \in H$, and any smooth function $\psi \in {\mathcal C}^\infty_c(\overline H)$,
\begin{equation}\label{eq.varfL2} 
 \psi(x) = \int_H{A \nabla_y L_A(x,y) \cdot \nabla \psi(y) \: \d y }~.
\end{equation}
Introducing the symmetric, positive definite matrix $M \in \R^{2\times 2}$ for which $M^{-2} = A$, we may write the latter requirement as follows
$$\forall \psi \in {\mathcal C}^\infty_c(\overline H), \:\: \psi(x) = \int_H (M^{-1} \nabla_y L_A(x,y)) \cdot (M^{-1} \nabla \psi(y)) \: \d y ~.$$
Changing variables and using test functions $\psi \in {\mathcal C}^\infty_c(\overline H)$ of the form $\psi(y) = \widetilde{\psi}(My)$, $\widetilde{\psi} \in {\mathcal C}^\infty_c(M\overline H)$, we arrive at
$$\forall x \in MH, \:\: \forall \widetilde{\psi} \in {\mathcal C}^\infty_c(M \overline H), \:\: \widetilde{\psi}(x) = \int_{MH}{ |\det M^{-1}| \: \nabla_z (L_A(M^{-1} x,M^{-1}z)) \cdot \nabla \widetilde{\psi}(z) \: \d z }~.$$
Therefore, it suffices that the function $(x,y) \mapsto \frac{1}{|\det M|} L_A(M^{-1}x, M^{-1}y)$ be a Neumann function for the Laplacian  on the rotated half-space $MH$. Such a function can easily be constructed  by reflection -- more precisely
\begin{equation}\label{eq.LAconst}
\frac{1}{|\det M|}  L_A(M^{-1}x, M^{-1}y) = G(x,y) + G(x,s_M(y)),
\end{equation}
where 
$$s_M(y) := y - 2\left(y \cdot  \frac{M^{-1}e_2}{|M^{-1}e_2|}\right)  \frac{M^{-1}e_2}{|M^{-1}e_2|} $$
is the symmetric image of a point $y \in MH$ with respect to the hyperplane $\partial (MH)$ (whose unit normal vector equals $\frac{M^{-1} e_2}{| M^{-1}e_2|}$). The desired expression \cref{eq.exprLB} for $L_A(x,y)$ follows immediately. 

\end{proof}
We next turn to the proof of \cref{lem.etae} concerning the remainder $\eta_\e$. 

\begin{proof}[Proof of \cref{lem.etae}]
The definition of $\eta_\e$ as the right-hand side of \cref{eq.defetae} features four terms, which we denote by $I^i_\e(x)$, $i=1,\ldots,4$, respectively. 
We prove that each of these converges to $0$ strongly in $H^{1/2}(\D_1)$. 

First, using the smoothness of $v_0$ near the point $0$ together with the fact that $v_0(0) = u_0(0)$, we get
\begin{equation}\label{eq.u0v0e}
I_\e^1(x) := u_0(0) - v_0(\e x) \xrightarrow{\e \to 0} 0 \text{ strongly in } H^{1/2}(\D_1)~.
 \end{equation}
\noindent
Secondly, the term
\begin{equation}\label{eq.I2rep}
 I_\e^2(x) :=  \int_{\partial {\mathcal O} \setminus {U}}{ A(\e x) \nabla_y L_{A(\e x)}(\e x,y) \cdot n(y) s_\e(y)\:\d s(y) }
 \end{equation}
 is an integral over the set $\partial {\mathcal O} \setminus {U}$, which lies ``far'' from  $\D_\e$. 
Since the function $(x,y) \mapsto A(\e x) \nabla_y L_{A(\e x)}(\e x,y)$ is smooth for $x \in \D_1$ and $y \in \partial {\mathcal O} \setminus {U}$ (uniformly with respect to $\e$), and since
$s_\e \to 0$ strongly in $H^1(\O)$ by virtue of \cref{lem.firstest} ( \cref{rem.prelestdiraniso}) and \cref{eq.capaDe}, it follows easily that $I_\e^2(x) \to 0$ strongly in $H^{1/2}(\D_1)$. 

For the very same reason, the term 
\begin{equation}\label{eq.I3rep}
 I_\e^4(x) := - \int_{\Gamma_D}{(A\nabla s_\e \cdot n)(y) \: L_{A(\e x) }(x,y) \: \d s(y)} 
 \end{equation}
also converges to $0$ strongly in $H^{1/2}(\D_1)$. Finally, we consider the term
\begin{equation}\label{eq.Kepsrep} 
I^3_\e(x) :=  - \int_{{\mathcal O}}{(A(\e x)-A(y)) \nabla_y L_{A(\e x) }(\e x,y) \cdot \nabla s_\e(y) \: \d y}~.
\end{equation}
Using \cref{lem.LB} and the subsequent \cref{rem.LAnorm}, we see that, for $x\in \D_1$ and $y \in {\mathcal O}$,   
$$ \nabla_y L_{A(\e x)}(\e x , y) = 
 \frac{-1}{\pi \sqrt{\det(A(\e x))}} \frac{M^2(\e x) (y-\e x)}{|M(\e x)(y-\e x)|^2}.
$$
As the matrix field $A(y)$ is smooth, there exists a constant $C>0$ such that:
$$\forall x \in \D_1, \: y \in {\mathcal O}, \quad ||A(\e x) - A(y)|| \leq C|\e x - y|~,$$
where $|| \cdot  ||$ denotes any matrix norm.  We then estimate
\begin{equation}\label{eq.estI4edir}
 |I^3_\e(x)| \leq C \int_{{\mathcal O}}{ |\nabla s_\e (y)|\: \d y} \leq C || s_\e ||_{H^1({\mathcal O})}~.
 \end{equation}
Invoking again \cref{lem.firstest} (\cref{rem.prelestdiraniso}) and \cref{eq.capaDe}, we conclude that
$$ |I_\e^3(x)| \xrightarrow{\e \to 0} 0 \text{ uniformly in } x \in \D_1~,$$
which implies, in particular, the strong $L^2(\D_1)$ convergence of $I_\e^3$ to $0$. It remains to prove that $I_\e^3$ converges to $0$ strongly in $H^{1/2}(\D_1)$. 
To this end, we return to the formula \cref{eq.v0De}, which reads
\begin{multline*}
 I_\e^3\left(\frac{x}{\e}\right) = -s_\e(x) - \int_{\partial {\mathcal O} \setminus {U}}{ A(x) \nabla_y L_{A(x)}(x,y) \cdot n(y) s_\e(y)\:\d s(y) }   + 
 \int_{\Gamma_D}{(A\nabla s_\e \cdot n)(y) \: L_{A(x) }(x,y) \: \d s(y)} \\
 + \int_{\widehat{\Gamma_N}}{(A\nabla s_\e \cdot n)(y) \: L_{A(x) }(x,y) \: \d s(y)}, \quad x \in \D_\e~,
 \end{multline*}
where we have taken advantage of the fact that $A\nabla s_\e \cdot n$ belongs to $\widetilde{H}^{-1/2}(\D_\e)$ (and vanishes in $\widehat{\Gamma_N}\setminus \D_\e$) to express the last integral in the above right-hand side as an integral on the whole  set $\widehat{\Gamma_N}$.
Using the mapping properties of the integral operator with kernel $L_{A(x) }(x,y)$ (see \cref{th.TK}), we obtain 
$$ \left\lvert I^3_\e\left(\frac{\cdot }{\e}\right) \right\lvert_{H^{1/2}(\D_\e)} \leq C || s_\e ||_{H^1(\mathcal O)}~,$$
where we recall the definition \cref{eq.normHs} of the semi-norm $|\cdot|_{H^{1/2}(\D_\e)}$.
Changing variables in the definition of this semi-norm to rescale the above left-hand side, we now get 
$$ \left\lvert I^3_\e \right\lvert_{H^{1/2}(\D_1)} \leq C || s_\e ||_{H^1(\mathcal O)}~.$$
We conclude from \cref{lem.firstest} ( \cref{rem.prelestdiraniso}) and \cref{eq.capaDe}, that $\left\lvert I_\e^3 \right\lvert_{H^{1/2}(\D_1)} \to 0$.
Finally, as we already know that $I_\e^3$ converges to $0$ strongly in $L^2(\D_1)$, it follows that $I_\e^3$ converges to $0$ strongly in  $H^{1/2}(\D_1)$, which completes the proof of the lemma. 
 \end{proof}
 
 We next turn to the proof of the approximation \cref{lem.Teps}.
  \begin{proof}[Proof of \cref{lem.Teps}]
Let $D \subset \R^2$ be a smooth bounded domain, 
whose boundary $\partial D$ is a closed curve containing $\D_1$ as a subset.  
Since $\widetilde{H}^{-1/2}(\D_1)$ is the space of distributions in $\D_1$ whose extension by $0$ to $\partial D$ 
belongs to $H^{-1/2}(\partial D)$, and since $H^{1/2}(\D_1)$ is the space of restrictions to $\D_1$ of elements from $H^{1/2}(\partial D)$ (see \cref{sec.Hs}), it is enough to prove that the asymptotic formula in the statement of \cref{lem.Teps} holds when all the operators at play are seen as operators from $H^{-1/2}(\partial D)$ into $H^{1/2}(\partial D)$. 
  
To this end, let us first simplify the expression \cref{eq.exprLB} for the function $L_{A(\e x)}(\e x , \e y)$ featured in the definition \cref{eq.defTe2ddir} of the operator $T_\e$
$$ \forall x,y \in \D_1, \: x \neq y, \quad L_{A(\e x)}(\e x,\e y) = \frac{-1}{\pi \sqrt{ |\det A(\e x)|}} \log \left\lvert M(\e x) ( \e x-\e y)\right\lvert~ .$$
The matrix field $A$ is given by \cref{eq.defgA}, and its definition readily implies that $A(\e x)$ tends to $\gamma(0) \I$ in ${\mathcal C}^k(V)$ for any integer $k \geq 0$ and any relatively compact open neighborhood $V$ of $0$ in $\R^2$. 
Hence, $T_\e$ may be decomposed as: 
\begin{multline}\label{eq.Tedecomp2d}
 T_\e \varphi =  \frac{1}{\pi \gamma(0)}  (|\log \e| + \alpha) \int_{\D_1}{\varphi \: \d s} + \frac{2}{\gamma(0)} S_1 \varphi \\
+ \left(  \frac{1}{\pi \sqrt{ |\det A(\e x)|}} -  \frac{1}{\pi \gamma(0)}  \right) ( |\log \e| + \alpha)  \int_{\D_1}{\varphi \: \d s} + T_{K_\e} \varphi~,
 \end{multline}
where $T_{K_\e}$ is the integral operator with kernel $K_\e(x,x-y)$, and $K_\e$ is given by
$$ K_\e(x,z) := \frac{1}{\pi \gamma(0)} \log|z| - \frac{1}{\pi \sqrt{ |\det A(\e x)|}} \log \left\lvert \sqrt{\gamma(0)} M(\e x) z \right\lvert~.$$
The first two terms in the right-hand side of \cref{eq.Tedecomp2d} correspond to the desired limiting behavior for $T_\e$, and the third term is easily seen to converge to $0$ as an operator from $H^{-1/2}(\partial D)$ into $H^{1/2}(\partial D)$. 
We then focus on the operator $T_{K_\e}$. It is easy to verify that $K_\e$ is a homogeneous kernel of class $-1$ in the sense of \cref{def.homkernel}.
Hence, \cref{th.TK} implies that $T_{K_\e}$ maps $H^{-1/2}(\partial D)$ into $H^{1/2}(\partial D)$. Note that we may modify $K_\e$, in such a way that it vanishes outside a sufficently large compact set (since the definition of $T_{K_\e}$ only involves values $K_\e(x,x-y)$ for $x,y \in \partial D$). With this modification we have 
$$\sup\limits_{|\alpha| \leq k \atop |\beta| \leq k}\sup \limits_{x\in \R^d} \sup\limits_{|z|=1} \left\lvert \frac{\partial^\alpha}{\partial x^\alpha} \frac{\partial^\beta}{\partial z^\beta} K_\e(x,z) \right\lvert \xrightarrow{\e \to 0} 0~,$$
for any integer $k$.
In view of \cref{th.TK}, $T_{K_\e}$ converges to $0$ in the operator norm
$$ \sup\limits_{\varphi \in H^{-1/2}(\partial D) \atop || \varphi ||_{H^{-1/2}(\partial D)} \leq 1} || T_{K_\e} \varphi ||_{H^{1/2}(\partial D)} \xrightarrow{\e \to 0} 0~,$$ 
which finishes the proof.  
\end{proof}

\begin{proof}[Proof of \cref{lem.S1}]

\noindent \textit{Proof of (i).} Let $D \subset \R^2$ be a smooth bounded domain, 
whose boundary $\partial D$ is a closed curve containing ${\D}_1$ as a subset. 
We also introduce another bounded Lipschitz domain $V \subset \R^2$ such that $D \Subset V$, and a smooth cut-off function $\chi \in {\mathcal C}^\infty_c(\R^2)$
such that $\chi \equiv 1$ on a neighborhood of $\overline{D}$ and $\chi \equiv 0$ on $\R^2 \setminus \overline V$. 

The proof follows an idea in \cite{stephan1986boundary,stephan1987boundary}; it relies on the connection between $S_1$ and the single layer potential
${\mathcal S}_D : H^{-1/2}(\partial D) \to H^{1}_{\text{loc}}(\R^2) $ associated with  $D$, as defined in \cref{eq.SD}. More precisely
$$\forall \varphi \in \widetilde{H}^{-1/2}({\D}_1), \:\: S_1 \varphi = ({\mathcal S}_D \varphi) \lvert_{{\D}_1}  ~,$$
where the density $\varphi$ in the right hand side is extended by $0$ outside $\D_1$. We first show that $S_1: \widetilde{H}^{-1/2}({\D}_1) \to H^{1/2}({\D}_1) $ 
is a Fredholm operator with index $0$ by adapting the argument of the proof of Th. 7.6 in \cite{mclean2000strongly}. 
The classical mapping properties of the single layer potential ${\mathcal S}_D$ imply that there exists a constant $C >0$ such that 
for any density $\varphi \in  \widetilde{H}^{-1/2}({\D}_1) $, the associated potential $u  = {\mathcal S}_D\varphi \in H^1_{\text{loc}}(\R^2)$ satisfies
$$ || \chi u ||_{H^1(\R^2)} \leq C || \varphi ||_{H^{-1/2}(\partial D)}~. $$
Conversely, we infer from the jump relations \cref{eq.jumpSD} of the single layer potential that
\begin{equation}\label{eq.estvphichiu}
 || \varphi ||_{H^{-1/2}(\partial D)} = \left\lvert \left\lvert  \frac{\partial (\chi u)^+}{\partial n} -  \frac{\partial (\chi u)^-}{\partial n}  \right\lvert\right\lvert_{H^{-1/2}(\partial D)} \leq C || \chi u||_{H^1(\R^2)}~.
 \end{equation}
Now, using again \cref{eq.jumpSD} together with integration by parts, we obtain that, for $\varphi \in \widetilde H^{-1/2}(\D_1)$, 
\begin{eqnarray*}
 \langle S_1\varphi, \varphi \rangle = \int_{\partial D}{{\mathcal S}_D \varphi \: \varphi \: \d s} &=& \int_{\partial D}{ \chi u \left(\frac{\partial (\chi u)^+}{\partial n} - \frac{\partial (\chi u)^-}{\partial n}\right) \: \d s} \\
&=& - \int_{\R^2}{|\nabla (\chi u ) |^2 \:\d x} - \int_{\R^2 \setminus \overline{D}}{\Delta (\chi u) \chi u \:\d x}\\
&=& - \int_{\R^2}{|\nabla (\chi u ) |^2 \:\d x} - \int_{\R^2 \setminus \overline{D}}{\Big(\Delta \chi u + 2\nabla \chi \cdot \nabla u \Big) \chi u \:\d x}\\
&=& - \int_{\R^2}{|\nabla (\chi u ) |^2 \:\d x} - \int_{\R^2 \setminus \overline{D}}{\Big(\chi (\Delta \chi) u^2 + 2u \nabla \chi \cdot \nabla (\chi u) - 2 u^2 |\nabla \chi|^2 \Big) \:\d x}~,\\
\end{eqnarray*}
which we rewrite as
\begin{equation}\label{eq.estchiu}
 \int_{\R^2}{|\nabla (\chi u ) |^2 \:\d x}  = -\langle S_1\varphi, \varphi \rangle - \int_{\R^2 \setminus \overline{D}}{\Big(\chi (\Delta \chi) u^2 + 2u \nabla \chi \cdot \nabla (\chi u) - 2 u^2 |\nabla \chi|^2 \Big) \:\d x}~. 
 \end{equation}
The Cauchy-Schwarz inequality (and rearrangement) now implies the existence of a constant $C$ such that 
$$|| \nabla (\chi u ) ||_{L^2(\R^2)^2}^2  \leq C\Big(|| S_1 \varphi ||_{H^{1/2}(\D_1)} || \varphi ||_{\widetilde{H}^{-1/2}(\D_1)} +  || u ||_{L^2(V)}^2  \Big)~.$$
A combination with \cref{eq.estvphichiu}, and insertion of $u= {\mathcal S}_D\varphi$, yields the existence of a constant $C$ such that, for arbitrary $\varphi \in  \widetilde{H}^{-1/2}({\D_1})$ (extended by $0$ outside $\D_1$)
\begin{equation}\label{eq.estpeetrevphi}
||\varphi ||_{\widetilde{H}^{-1/2}(\D_1)}  = ||\varphi ||_{H^{-1/2}(\partial D)} \leq C \Big( ||S_1 \varphi ||_{H^{1/2}(\D_1)} + || {\mathcal S}_D\varphi ||_{L^2(V)} \Big)~. 
\end{equation}
Since the mapping $H^{-1/2}(\partial D) \ni \varphi \mapsto {\mathcal S}_D \varphi \in H^1(V)$ is continuous
and the injection $H^1(V) \to L^2(V)$ is compact, an application of  Peetre's \cref{lem.peetre} to \cref{eq.estpeetrevphi} reveals that $S_1$ has finite dimensional kernel $\Ker(S_1) \subset \widetilde{H}^{-1/2}(\D_1)$, 
and closed range $\Ran(S_1) \subset H^{1/2}(\D_1)$. Finally, since $S_1$ is self-adjoint, it follows that
$$ \Ker(S_1) = \Ran(S_1)^\perp, \text{ and so } \Ran(S_1) = \Ker(S_1)^\perp~.$$
In summary $S_1$ is a Fredholm operator with index $0$.\par\medskip

In order to prove that $S_1$ is invertible, it thus suffices to prove that it is injective on $ \widetilde{H}^{-1/2}({\D}_1)$. 
To this end, let $\varphi \in  \widetilde{H}^{-1/2}({\D}_1)$ be such that $S_1 \varphi = 0$ on ${\D}_1$. We assume first that $\varphi$ has mean $0$, that is $\langle \varphi , 1\rangle = \int_{\partial D}{\varphi \: \d s} = 0$. 
Then, the associated single layer potential ${\mathcal S}_D \varphi : \R^2 \to \R$ satisfies the decay property 
\begin{equation}\label{eq.decaySDvphiS1}
 | {\mathcal S}_D\varphi (x) | = {\mathcal O}\left( \frac{1}{|x|}\right) \text{ as } |x|\to \infty;
 \end{equation}
see \cref{eq.decaySDmean0}.
From the same integration by parts which led to \cref{eq.estchiu} (and which can now be carried out without introducing a cut-off function $\chi$ because of the decay property \cref{eq.decaySDvphiS1}), 
we obtain 
\begin{equation}\label{eq.ippSDmean0}
 \langle S_1\varphi, \varphi \rangle = 0 = - \int_{\R^2}{|\nabla ({\mathcal S}_D\varphi )|^2 \:\d x},
 \end{equation}
and so ${\mathcal S}_D\varphi = 0$ on $\R^2$. 
As a result, 
$$ \varphi = \frac{\partial ({\mathcal S}_D\varphi)^+}{\partial n} - \frac{\partial ({\mathcal S}_D\varphi)^-}{\partial n}  = 0 \text{ on } \partial D, $$
as desired. Finally, let us consider the general case where $S_1\varphi = 0$ but $\langle \varphi, 1 \rangle$ does not necessarily vanish. 
From \cref{prop.eqdistrib}, the function $\varphi_c \in \widetilde{H}^{-1/2}({\D}_1)$ defined by:
$$ \varphi_c(x_1) = \frac{1}{\pi\sqrt{1-x_1^2}}$$
is such that: 
$$\langle \varphi_c, 1\rangle = 1, \text{ and } S_1 \varphi_c =  \frac{\log 2}{2\pi} \text{ on } \D_1.$$
Hence, the element $\varphi_0 \in  \widetilde{H}^{-1/2}(\D_1)$ defined by: 
$$ \varphi_0 = \varphi - \langle\varphi,1\rangle \varphi_c,$$
satisfies the following properties:
$$ \langle \varphi_0, 1 \rangle = 0 \text{ and } S_1\varphi_0 = -\langle \varphi, 1 \rangle S_1 \varphi_c = -\frac{ \log 2}{2\pi} \: \langle\varphi,1\rangle,$$ 
so that $\langle S_1 \varphi_0, \varphi_0 \rangle = 0$. 
The same calculation as in \cref{eq.ippSDmean0} reveals that $\varphi_0=0$, and so
$\varphi = \langle \varphi,1 \rangle \varphi_c$. 
Eventually, since $S_1 \varphi = 0$, we obtain $\langle \varphi,1\rangle = 0$, so that $\varphi = 0$, as desired. \par\medskip

\noindent \textit{Proof of (ii).} Both formulas \cref{eq.Vem1,eq.moyVe} follow from simple calculations.
\end{proof}

\begin{remark}
A significantly simpler proof of \cref{th.expdir2d} can be given, under the additional assumption that the boundary $\partial \Omega$ is completely flat in a fixed neighborhood $U$ of the $\omega_\e$ (i.e. $\partial \Omega \cap U = \partial H \cap U$) and that the conductivity $\gamma$ is constant in such a neighborhood. 
\end{remark}

%%%%%%
\subsection{Adaptation to the three-dimensional case}\label{sec.dir3dasym}
%%%%%%

\noindent We proceed with the three-dimensional version of the general problem 
described at the beginning of this section: the background and perturbed potentials $u_0$ and $u_\e$ are still characterized by the equations \cref{eq.ueu0si}, 
and we look for the asymptotic expansion of $u_\e$ as the size $\e$ of the subset $\omega_\e \subset \partial \Omega$, defined by \cref{eq.GNflat}, vanishes.  

The counterpart of \cref{th.expdir2d} is the following. Since the proof is quite similar in most aspects, we only elaborate on the differences.  
 \begin{theorem}\label{th.expdir3d}
The following asymptotic expansion holds at any point $x\in \overline \Omega$, $x\notin \Sigma\cup \left\{ 0 \right\}$: 
$$ u_\e(x) = u_0(x) - 4 \e \gamma(0) u_0(0) N(x,0) + \o(\e)~.$$
\end{theorem}

\begin{proof}
As in the two-dimensional case, we introduce the transported functions $v_0 := u_0 \circ T$ and $v_\e := u_\e \circ T$. 
These are characterized as the unique $H^1({\mathcal O})$ solutions to the problems in \cref{eq.vev0si}, 
which feature the smooth matrix field $A \in {\mathcal C}^\infty(\R^3, \R^{3\times 3})$ and source term $g \in {\mathcal C}^\infty(\mathbb{R}^3)$ defined as in \cref{eq.defgA}.
We also introduce the error $r_\e := u_\e - u_0 \in H^1(\Omega)$ and its transformed version $s_\e := v_\e - v_0 \in H^1({\mathcal O})$. The proof of the theorem again proceeds in four steps. \par\medskip

\noindent \textit{Step 1.} We construct a representation formula for $u_\e$ in terms of the values of $r_\e$ inside $\omega_\e$. 
Arguing as in the first step of the proof of \cref{th.expdir2d}, we prove that, for any point $x \in \Omega$
 \begin{equation*}\label{eq.repformulauesi3d} 
 u_\e(x) =u_0(x) + \int_{\omega_\e}{\gamma(y)\frac{\partial r_\e}{\partial n_y}(y) N(x,y)\:\d s(y)}~,
 \end{equation*}
 an identity which also holds for $x \in \partial \Omega$ in the sense of traces in $H^{1/2}(\partial \Omega)$. 
 Performing a change of variables based on the diffeomorphism $T$  we arrive at 
 \begin{equation}\label{eq.repformulauesi3d} 
 u_\e(x) = u_0(x) + \int_{\D_1}{ \varphi_\e(z) N(x,T(\e z))\:\d s(z)}~,
 \end{equation}
where the rescaled density $\varphi_\e$ is given by
$$\varphi_\e(z) = \e^2 (A\nabla s_\e \cdot n)(\e z) \in \widetilde{H}^{-1/2}(\D_1)~.$$ 
\par\medskip

\noindent \textit{Step 2.} We characterize $\varphi_\e$ as the solution to an integral equation. 
 To this end, again, we rely on a variant of the representation formula \cref{eq.repformulauesi3d} adapted to the function $v_\e$, and obtained 
 with the use of a special function $y \mapsto L_{A(x)}(x,y)$ which satisfies, for given $x \in H$
 \begin{equation}\label{eq.defLAx3d}
\left\{
\begin{array}{cl}
-\dv_y(A(x) \nabla_y L_{A(x)}(x,y)) = \delta_{y=x}  & \text{in } H,  \\
A(x)\nabla_y L_{A(x)} (x,y) \cdot n(y) = 0 & \text{on } \partial H.
\end{array}
 \right.
\end{equation}
The construction of such a function is accomplished exactly as in the two-dimensional case; see \cref{eq.defLB} and \cref{lem.LB}. 
The same calculations as in Step 2 of the proof of \cref{th.expdir2d} then yield, for a.e. $x \in \D_\e$ 
   \begin{eqnarray}\label{eq.v0De3d}
   -v_0(x) = - \int_{\partial {\mathcal O} \setminus {U}}{ A(x) \nabla_y L_{A(x)}(x,y) \cdot n(y) s_\e(y)\:\d s(y) }   + \int_{{\mathcal O}}{(A(x)-A(y)) \nabla_y L_{A(x) }(x,y) \cdot \nabla s_\e(y) \: \d y} \nonumber \\ 
+ \int_{\Gamma_D}{(A\nabla s_\e \cdot n)(y) \: L_{A(x) }(x,y) \: \d s(y)} + \int_{\D_\e}{(A\nabla s_\e \cdot n)(y) \: L_{A(x) }(x,y) \: \d s(y)}~,
  \end{eqnarray}
which, after rescaling, reads   
   \begin{multline*}
   -v_0(\e x) = - \int_{\partial {\mathcal O} \setminus {U}}{ A(\e x) \nabla_y L_{A(\e x)}(\e x,y) \cdot n(y) s_\e(y)\:\d s(y) }   + \int_{{\mathcal O}}{(A(\e x)-A(y)) \nabla_y L_{A(\e x) }(\e x,y) \cdot \nabla s_\e(y) \: \d y} \\ 
+ \int_{\Gamma_D}{(A\nabla s_\e \cdot n)(y) \: L_{A(\e x) }(\e x,y) \: \d s(y)} + \int_{\D_1}{ \varphi_\e(z) \: L_{A(\e x) }(\e x,\e z) \: \d s(z)}~,
  \end{multline*}
  for a.e. $x \in \D_1$. This can be recast in the form of an integral equation
  \begin{equation}\label{eq.inteq3d}
   T_\e \varphi_\e = -u_0(0) + \eta_\e~,
   \end{equation}
  where the operator $T_\e: \widetilde{H}^{-1/2}(\D_1) \to H^{1/2}(\D_1) $ is defined by 
  $$ T_\e \varphi (x) =  \int_{\D_1}{ \varphi(z) \: L_{A(\e x) }(\e x,\e z) \: \d s(z)}~,  $$
  and $\eta_\e \in H^{1/2}(\D_1)$ denotes the remainder 
  \begin{multline}\label{eq.defetae3d} 
 \eta_\e(x) = u_0(0) - v_0(\e x)  +  \int_{\partial {\mathcal O} \setminus {U}}{ A(\e x) \nabla_y L_{A(\e x)}(\e x,y) \cdot n(y) s_\e(y)\:\d s(y) }  \\
 -  \int_{{\mathcal O}}{(A(\e x)-A(y)) \nabla_y L_{A(\e x) }(\e x,y) \cdot \nabla s_\e(y) \: \d y}  - \int_{\Gamma_D}{(A\nabla s_\e \cdot n)(y) \: L_{A(\e x) }(\e x,y) \: \d s(y)}~.
 \end{multline}
 \par\medskip  
 \noindent \textit{Step 3.} We analyze the integral equation \cref{eq.inteq3d} 
 to obtain information about the asymptotic behavior of $\varphi_\e$.
To this end, we rely on the following two lemmata, which are the exact counterparts of \cref{lem.etae} and \cref{lem.Teps} in the present 3d situation; 
their proofs are outlined at the end of this section.
 \begin{lemma}\label{lem.etae3d}
 The remainder term $\eta_\e \in H^{1/2}(\D_1)$, defined in \cref{eq.defetae3d}, satisfies
 \begin{equation}\label{eq.cvetaeps}
  \eta_\e \xrightarrow{\e \to 0} 0 \text{ weakly in } H^{1/2}(\D_1)~.
  \end{equation}
 \end{lemma}
 
 \begin{lemma}\label{lem.Teps3d}
The following asymptotic expansion holds
\begin{equation}\label{eq.asymL3d}
\sup\limits_{\varphi \in \widetilde{H}^{-1/2}(\D_1) \atop || \varphi ||_{\widetilde{H}^{-1/2}(\D_1) } \leq 1}  \e \left\lvert \left\lvert T_\e \varphi 
- \frac{2}{\e\gamma(0)} S_1 \varphi \right\lvert  \right\lvert_{H^{1/2}(\D_1)}  \xrightarrow{\e \to 0} 0~,
\end{equation}
where the operator $S_1 : \widetilde{H}^{-1/2}({\D}_1) \to H^{1/2}(\D_1)$ is defined by
\begin{equation}\label{eq.defS13ddir}
S_1\varphi(x) = \frac{1}{4\pi}\int_{{\D}_1}{\frac{1}{|x-y|} \varphi(y) \: \d s(y) }, \:\: x\in{\D}_1~.
\end{equation}
 \end{lemma}
 
Using this result in combination with the integral equation \cref{eq.inteq3d}, we see that the function $\varphi_\e \in  \widetilde{H}^{-1/2}(\D_1)$ 
satisfies the integral equation
\begin{equation}\label{eq.inteqS1dir3d}
 S_1 \varphi_\e +  R_\e \varphi_\e = - \frac{\e}{2} \gamma(0) u_0(0) + \frac{\e\gamma(0)}{2} \eta_\e~, 
 \end{equation}
where $R_\e:  \widetilde{H}^{-1/2}(\D_1) \to H^{1/2}(\D_1)$ is a sequence of operators whose norms converge to $0$ and the sequence $\eta_\e$ converges to $0$ weakly in $H^{1/2}(\D_1)$. The study of this approximate version of our integral equation \cref{eq.inteq3d}
relies on the following lemma, whose proof is also postponed.
\begin{lemma}\label{lem.S13d}
The operator $S_1 : \widetilde{H}^{-1/2}( \D_1) \to H^{1/2}(\D_1)$ 
is invertible. 
\end{lemma}

It then follows from \cref{eq.inteqS1dir3d}, \cref{lem.S13d}, 
and the use of a Neumann series, that the function $\varphi_\e \in \widetilde{H}^{-1/2}(\D_1)$ satisfies: 
$$ \varphi_\e = -\frac{\e}{2} \gamma(0) u_0(0) S^{-1} 1 + \e\widetilde{\eta_\e}~,$$
where $\widetilde{\eta_\e}$ is a sequence converging to $0$ weakly in $H^{1/2}(\D_1)$, and 
$S_1^{-1}1$ is the equilibrium distribution associated with the operator $S_1$, which is explicitly given by  \cref{eq.eqdist3ddir} in  \cref{prop.eqdistrib} of the appendix. In particular, we infer from \cref{eq.eqdist3ddirprops} that
\begin{equation}\label{eq.moyphie3d}
\langle \varphi_\e,1\rangle = -4\e \gamma(0) u_0(0)   + \o\left( \e \right),
\end{equation}
which is the needed information about $\varphi_\e$ for the next step.
  \par\medskip

 \noindent \textit{Step 4.} We pass to the limit in the representation formula \cref{eq.repformulauesi3d}, 
 which is valid for any point $x \in \overline\Omega$, $x\notin \Sigma \cup \{0\}$.
Arguing as in the final step of the proof of \cref{th.expdir2d}, we obtain
$$ u_\e (x) = u_0(x) + \int_{\D_1}{\varphi_\e(z) N(x,0) \: \d s(z)} + \o(\e)~,$$
and the result follows from \cref{eq.moyphie3d}. 
\end{proof}

We now provide a few details about the missing ingredients in the above proof.

\begin{proof}[Proof of \cref{lem.etae3d}]
As in the proof of \cref{lem.etae}, we denote the four terms in the right-hand side of \cref{eq.defetae3d} by $I^i_\e(x)$, $i=1,\ldots,4$. 
The exact same arguments as in the two-dimensional case show that $I_\e^1$, $I_\e^2$ and $I_\e^4$ converge to $0$ strongly in $H^{1/2}(\D_1)$, 
and we focus on the treatment of the last term
\begin{equation}\label{eq.Kepsrep3d} 
I_\e^3(x) :=  - \int_{{\mathcal O}}{(A(\e x)-A(y)) \nabla_y L_{A(\e x) }(\e x,y) \cdot \nabla s_\e(y) \: \d y}~.
\end{equation}
From the explicit expression for the function $L_{A(\e x)}$ supplied by \cref{lem.LB,rem.LAnorm}, a simple calculation yields, for $x \in \D_1$  
$$ \nabla_y L_{A(\e x)}(\e x , y) = 
\frac{-1}{2\pi \sqrt{\det(A(\e x))}} \frac{M^2(\e x) (y-\e x)}{|M(\e x)(y-\e x)|^3}~.
$$
Hence, using the Cauchy-Schwarz inequality and a switch to polar coordinates, we obtain 
$$\begin{array}{>{\displaystyle}cc>{\displaystyle}l}
 |I^3_\e(x)|& \leq & C \int_{\mathcal O}{\frac{1}{|\e x-y|} |\nabla s_\e(y) | \:\d y}\\
 &\leq& C \left( \int_{{\mathcal O}}{\frac{1}{|\e x-y|^2} \: \d y }\right)^{1/2} || s_\e ||_{H^1({\mathcal O})} \\
 &\leq& C ||s_\e ||_{H^1({\mathcal O})}~.
 \end{array}$$
Invoking \cref{lem.firstest} ( \cref{rem.prelestdiraniso}) about the asymptotic behavior of $s_\e$ together with the estimate \cref{eq.capaDe}, we conclude that
$$ |I^3_\e(x)| \xrightarrow{\e \to 0} 0 \text{ uniformly in } x \in \D_1~,$$
and in particular, 
\begin{equation}\label{eq.cv0L2strong3ddir}
|| I^3_\e ||_{L^2(\D_1)} \to 0~.
\end{equation}

Let us now consider the $H^{1/2}(\D_1)$ convergence of $I_\e^3$. 
To this end, we return to the formula \cref{eq.v0De3d}, which we rewrite
\begin{multline*}
 I_\e^3\left(\frac{x}{\e}\right) = -s_\e(x) - \int_{\partial {\mathcal O} \setminus {U}}{ A(x) \nabla_y L_{A(x)}(x,y) \cdot n(y) s_\e(y)\:\d s(y) }   + 
 \int_{\Gamma_D}{(A\nabla s_\e \cdot n)(y) \: L_{A(x) }(x,y) \: \d s(y)} \\
 + \int_{\widehat{\Gamma_N}}{(A\nabla s_\e \cdot n)(y) \: L_{A(x) }(x,y) \: \d s(y)}, \quad x \in \D_\e~.
 \end{multline*}
This identity, and the mapping properties of the integral operator with kernel $ L_{A(x) }(x,y)$ stated in \cref{th.TK} readily imply that
$$ \left\lvert I^3_\e\left(\frac{\cdot }{\e}\right) \right\lvert_{H^{1/2}(\D_\e)} \leq C || s_\e ||_{H^1(\mathcal O)}~.$$
After a change of variables in the semi-norm $|\cdot |_{H^{1/2}(\D_\e)}$, the above estimate yields
$$\e^{\frac{1}{2}} | I^3_\e |_{H^{1/2}(\D_1)} \leq C  || s_\e ||_{H^1(\mathcal O)}~,$$
and since $||s_\e ||_{H^1({\mathcal O})} \leq C \e ^{\frac{1}{2}}$ as a consequence of \cref{lem.firstest} ( \cref{rem.prelestdiraniso}) and \cref{eq.capaDe}, 
it follows that the function $I^3_\e$ is bounded in $H^{1/2}(\D_1)$. 
Hence, up to a subsequence, it converges to a limit weakly in $H^{1/2}(\D_1)$, which is necessarily $0$ by virtue of \cref{eq.cv0L2strong3ddir}. 
Finally, by uniqueness of the weak limit (that is, regardless of the chosen subsequence for the weak $H^{1/2}(\D_1)$ convergence of $I_\e^3$), 
the whole sequence $I_\e^3$ converges to $0$ weakly in $H^{1/2}(\D_1)$, 
which completes the proof.
 \end{proof}
 
 We next turn to the proof of the approximation \cref{lem.Teps3d}.
  \begin{proof}[Proof of \cref{lem.Teps3d}]
  Let $D \subset \R^3$ be a smooth bounded domain whose boundary contains $\D_1$. 
The kernel $L_{A(\e x)}(\e x , \e y)$ of the operator $T_\e$ reads, for $x,y \in \D_1$, $x\neq y$
$$  L_{A(\e x)}(\e x,\e y) = \frac{1}{ 2\pi \e \sqrt{ |\det A(\e x)|}} \frac{1}{ \left\lvert M(\e x) ( x-y)\right\lvert  }~,$$
see \cref{eq.exprLB}.
Let us now recall from \cref{eq.defgA} that the matrix field $A(\e x)$ tends to $\gamma(0) \I$ in ${\mathcal C}^k(V)$ for any integer $k \geq 0$ and any open, relatively compact neighborhood $V$ of $0$ in $\R^3 $. $T_\e$ may be decomposed as
$$ T_\e \varphi =  \frac{2}{\e \gamma(0)}  S_1\varphi  + T_{K_\e} \varphi~,$$
where $T_{K_\e}$ is defined as the integral operator with kernel $K_\e(x,x-y)$, and $K_\e$ denotes the following homogeneous kernel of class $-1$, in the sense of \cref{def.homkernel},
$$ K_\e(x,z) :=  \frac{1}{2\pi \e \sqrt{\det(A(\e x))}} \frac{1}{|M(\e x) z|} - \frac{1}{2\pi \e \gamma(0) |z|}~.$$
According to \cref{th.TK}, this operator maps $H^{-1/2}(\partial D)$ into $H^{1/2}(\partial D)$. For any integer $k\ge 0$, we furthermore have
$$\sup\limits_{|\alpha| \leq k \atop |\beta| \leq k}\sup \limits_{x\in \R^d} \sup\limits_{|z|=1} \left\lvert \frac{\partial^\alpha}{\partial x^\alpha} \frac{\partial^\beta}{\partial z^\beta} K_\e(x,z) \right\lvert \xrightarrow{\e \to 0} 0~. $$
Here we have, again, ``cut off" $K_\e$ outside a sufficiently large compact set. In light of \cref{th.TK}, this limiting behaviour implies that $T_{K_\e}$ converges to $0$ as an operator from $H^{-1/2}(\partial D)$ into $H^{1/2}(\partial D)$, and so as an operator from $\widetilde{H}^{-1/2}(\D_1)$ into $H^{1/2}(\D_1)$, which is the desired result.  
\end{proof}

\begin{proof}[Sketch of proof of \cref{lem.S13d}.] The proof is very similar to that of \cref{lem.S1}, and we only point out the differences. Repeating mutatis mutandis the argument presented in the two-dimensional case,
one sees that the operator $S_1$ is still Fredholm with index $0$, and so, it suffices to prove that it is injective. 
To achieve this, let $\phi \in \widetilde{H}^{-1/2}(\D_1)$ be a density such that $S_1 \phi = 0$, and let $u ={\mathcal S}_D \phi \in H^1_{\text{\rm loc}}(\R^3)$ be the associated potential.
Because of the decay properties at infinity \cref{eq.decayS3d} of the single layer potential in three space dimensions (which hold even if $\langle \varphi, 1 \rangle \neq 0$), an integration by parts similar to that which led to \cref{eq.estchiu}, reveals that
$$ 0= \langle S_1 \phi, \phi \rangle = -\int_{\R^3}{|\nabla u|^2 \:\d x}~.$$
Hence $u$ is constant on $\R^3$. Since $|u(x)| \to 0$ as $|x| \to \infty$, it follows that $u$ vanishes identically, and  so does $\phi = -\left( \frac{\partial u}{\partial n}^+ - \frac{\partial u}{\partial n}^-\right)$. This shows the injectivity (and thus the bijectivity) of $S_1$.
\end{proof}

%%%%%%%%%%%%%%%%%%%%%%%%%%%%%%%%%%%%%%%%%%%%%%%%%%%%%%%
\section{An explicit asymptotic formula for the case of substituting Neumann conditions}\label{sec.calcneu}
%%%%%%%%%%%%%%%%%%%%%%%%%%%%%%%%%%%%%%%%%%%%%%%%%%%%%%%

\noindent This section exemplifies the general physical setting of \cref{sec.neu}: 
we consider a smooth, bounded domain $\Omega$ in $\R^d$, 
whose boundary is made of two disjoint, open Lipschitz subregions $\Gamma_D$, $\Gamma_N$: $\partial \Omega= \overline{\Gamma_D}\cup \overline{\Gamma_N}$. $\Sigma$ denotes the interface between $\Gamma_D$ and $\Gamma_N$. 
The geometric setting is exactly as in \cref{sec.calcdir}, only with the roles of $\Gamma_D$ and $\Gamma_N$ interchanged. The vanishing subset $\omega_\e \subset \Gamma_D$ is of the same nature as in \cref{eq.GNflat}: 
it is the image of the planar disk $\D_\e$ with radius $\e$ around $0$ by the smooth diffeomorphism $T$ that maps the domain ${\mathcal O}$ (whose boundary is flat in a fixed neighborhood $U$ of $0$) onto $\Omega$. We also denote $\widehat{\Gamma_D} = T^{-1}(\Gamma_D)$ and we assume for convenience that $T$ coincides with the identity mapping far from $0$, so that $T^{-1}(\Gamma_N) = \Gamma_N$.
The background potential and the perturbed potential, $u_0$ and $u_\e$, respectively, are the solutions to the equations
\begin{equation}\label{eq.ueu0sineu}
\left\{ 
\begin{array}{cl}
-\dv(\gamma \nabla u_0) = f & \text{in } \Omega, \\
u_0 = 0 & \text{on } \Gamma_D, \\
\gamma\frac{\partial u_0}{\partial n} = 0 & \text{on } \Gamma_N,
\end{array}
\right.
\text{ and } 
 \left\{
\begin{array}{cl}
-\dv(\gamma \nabla u_\e) = f & \text{in } \Omega, \\
u_\e = 0 & \text{on } \Gamma_D \setminus \overline{\omega_\e}, \\
\gamma\frac{\partial u_\e}{\partial n} = 0 & \text{on } \Gamma_N \cup \omega_\e.
\end{array}
\right.
\end{equation}
Our aim is to derive a precise first order asymptotic expansion of $u_\e$ when $\e \to 0$.
In order to emphasize the similarity of this study with that conducted in \cref{sec.calcdir}, we use the same notation whenever possible. Our main result is the following.

\begin{theorem}\label{th.asymneu}
Let $d = 2$ or $3$ and let $x\in \overline\Omega$, $x \notin \Sigma \cup \left\{ 0 \right\}$. One has the asymptotic expansion 
$$ u_\e(x) = u_0(x) + a_d \e^d \gamma(0) \frac{\partial u_0}{\partial n}(0) \frac{\partial N}{\partial n_y}(x,0) + \o (\e^d)~, $$
where the constant $a_d$ is given by
 $$ a_d = 
 \left\{
 \begin{array}{cl} 
 \frac{\pi}{2} &\text{if }d=2, \\
 \frac{1}{3} & \text{if } d=3.
 \end{array}
 \right.
 $$
\end{theorem}
\begin{proof}[Sketch of the proof.]
As in the proof of \cref{eq.expuedir2d}, we proceed in four steps, introducing the difference $r_\e := u_\e - u_0 \in H^1(\Omega)$. \par\medskip

\noindent \textit{Step 1.} We construct a representation formula for $u_\e$ which only involves the values of $r_\e$ inside $\omega_\e$, and the fundamental solution $N(x,y)$ to the background equation in \cref{eq.ueu0sineu}. 

To this end, let $x \in \Omega$ be arbitrary; using the definition of $N(x,y)$ and integrating by parts twice, we obtain
$$\begin{array}{>{\displaystyle}cc>{\displaystyle}l}
 u_\e(x) &=& - \int_\Omega{\dv_y(\gamma(y) \nabla_y N(x,y)) u_\e(y) \: \d y} \\
 &=& - \int_{\partial \Omega}{\gamma(y) \frac{\partial N}{\partial n_y}(x,y) u_\e(y)\:\d s(y) } + \int_{\Omega}{\gamma(y) \nabla_y N(x,y) \cdot \nabla u_\e(y) \: \d y} \\
 &=& - \int_{\omega_\e}{\gamma(y) \frac{\partial N}{\partial n_y}(x,y) u_\e(y)\:\d s(y) } + \int_{\Omega}{f(y) N(x,y) \:\d y}~,
 \end{array}$$
where the last line follows from the facts that
$$ \gamma(y)\frac{\partial N}{\partial n}(x,y)=\gamma(y)\frac{\partial u_\e}{\partial n}(y) = 0 \text{ for } y\in \Gamma_N, \:\: \, N(x,y) = 0 \text{ for } y\in \Gamma_D \text{ and } u_\e(y) = 0 \text{ for } y\in \Gamma_D\setminus \overline{\omega_\e}~.$$
Using that $u_\e=u_\e-u_0=r_\e$ on $\omega_\e$ in the previous equation, we get for $x \in \Omega$
\begin{equation}\label{eq.repneusi} 
u_\e(x) = u_0(x) - \int_{\omega_\e}{\gamma(y) \frac{\partial N}{\partial n_y}(x,y) r_\e(y) \: \d s(y)}~.
\end{equation}
The above identity also holds for $x \in \overline \Omega$, $x \notin \Sigma \cup \left\{0 \right\}$ provided $\e$ is small enough, since all the quantities involved are smooth in a neighborhood of such points. 

 Next, we introduce the transformed potentials $v_0 := u_0 \circ T$ and $v_\e := u_\e \circ T$ on the domain ${\mathcal O}$. A change of variables in the variational formulations for \cref{eq.ueu0sineu} reveals that $v_0$ and $v_\e$ are the unique $H^1({\mathcal O})$ solutions to the equations 
 \begin{equation}\label{eq.vev0sineu}
\left\{ 
\begin{array}{cl}
-\dv(A \nabla v_0) = g & \text{in } {\mathcal O}, \\
v_0 = 0 & \text{on } \widehat{\Gamma_D}, \\
(A \nabla v_0) \cdot n = 0 & \text{on }\Gamma_N,
\end{array}
\right.
\text{ and } 
 \left\{
\begin{array}{cl}
-\dv(A \nabla v_\e) =  g & \text{in } {\mathcal O}, \\
u_\e = 0 & \text{on } \widehat{\Gamma_D} \setminus \overline{\D_\e}, \\
(A\nabla v_\e)\cdot n = 0 & \text{on } \Gamma_N \cup \D_\e,
\end{array}
\right.
\end{equation}
where $g \in {\mathcal C}^\infty(\R^d)$ and $A \in {\mathcal C}^\infty(\R^d, \R^{d\times d})$ are the smooth function and the matrix field defined by 
\begin{equation}\label{eq.defgAneu}
 g = |\det (\nabla T(y))| f(T(y)), \text{ and } A(y) = |\det (\nabla T(y))| \gamma(T(y)) \nabla T(y)^{-1}( \nabla T(y)^t)^{-1}~.
\end{equation}
Changing variables in the integral featured in \cref{eq.repneusi} and rescaling, we arrive at 
\begin{equation}\label{eq.repuecalcneu}
u_\e(x) = u_0(x) - \int_{ \D_1}{\gamma(T(\e y)) \frac{\partial N}{\partial n_y}(x,T(\e y)) \varphi_\e(y) \: \d s(y)}~,
\end{equation}
where we have introduced the function $s_\e := r_\e \circ T$, and the quantity $\varphi_\e \in \widetilde{H}^{1/2}(\D_1)$ defined by
\begin{equation}\label{eq.vphieNeu} 
\varphi_\e(y) = \e^{d-1} s_\e(\e y)~.
\end{equation}
This is the desired representation formula. \par\medskip

 \noindent \textit{Step 2.} We characterize $\varphi_\e$ as the solution to an integral equation.
 This arises from a representation formula for $u_\e$ which differs slightly from \cref{eq.repuecalcneu}: it is obtained by repeating the derivation of Step 1, 
 except that a different, explicit fundamental solution $L_{A(x)}(x,y)$ is used in place of $N(x,y)$. For any symmetric, positive definite matrix $A$, and any $x \in H$, let $L_A(x,y)$ be a solution to the following boundary value problem  posed on the lower half-space $H$
 \begin{equation}\label{eq.defLBdir}
\left\{
\begin{array}{cl}
-\dv_y(A \nabla_y L_A(x,y)) = \delta_{y=x}  & \text{in } H,  \\
L_A(x,y) = 0 & \text{on } \partial H.
\end{array}
 \right.
\end{equation}
An explicit formula for one such function is provided by the next lemma, whose proof is completely analogous to that of \cref{lem.LB} and is therefore omitted.
\begin{lemma}\label{lem.LBneu}
Let $A$ be a symmetric, positive definite $d \times d$ matrix, and let $M := A^{-1/2}$.
Let $G(x,y)$ be the fundamental solution of the operator $-\Delta$ in free space, cf. \cref{eq.GreenLap}.
The function $L_A(x,y)$ defined by
\begin{equation}\label{eq.exprLBneu}
  L_A(x,y) = |\det M| \left( G(Mx,My) - G(Mx, My - 2 y_d \frac{M^{-1}e_d}{|M^{-1}e_d|^2}\right), \quad  \:\: x \neq y~,
 \end{equation}
satisfies \cref{eq.defLBdir}.
\end{lemma}

For a point $x \in {\mathcal O}$, we obtain from two successive integrations by parts
 $$\begin{array}{>{\displaystyle}cc>{\displaystyle}l}
 v_\e(x) &=& - \int_{\mathcal O}{\dv_y(A(x) \nabla_y L_{A(x)}(x,y)) v_\e(y) \: \d y} \\
  &=& - \int_{\partial {\mathcal O}}{ A(x) \nabla_y L_{A(x)}(x,y) \cdot n(y) v_\e(y)\:\d s(y) }   + \int_{{\mathcal O}}{(A(x)-A(y)) \nabla_y L_{A(x) }(x,y) \cdot \nabla v_\e(y) \: \d y} \\ 
 && + \int_{\partial {\mathcal O}}{(A\nabla v_\e \cdot n)(y) \: L_{A(x) }(x,y) \: \d s(y)} + \int_{{\mathcal O}}{ g(y) L_{A(x) }(x,y) \: \d y}~.
 \end{array}$$
The same calculation based on the function $v_0$, instead of $v_\e$, yields 
\begin{multline*}
v_0(x) = - \int_{\partial {\mathcal O} }{ A(x) \nabla_y L_{A(x)}(x,y) \cdot n(y) v_0(y)\:\d s(y) }   + \int_{{\mathcal O}}{(A(x)-A(y)) \nabla_y L_{A(x) }(x,y) \cdot \nabla v_0(y) \: \d y} \\ 
 + \int_{\partial {\mathcal O}}{(A\nabla v_0 \cdot n)(y) \: L_{A(x) }(x,y) \: \d s(y)} + \int_{{\mathcal O}}{ g(y) L_{A(x) }(x,y) \: \d y}~.
 \end{multline*}
Forming the difference between these identities, and using the boundary conditions for $v_0$ and $v_\e$ we obtain
 \begin{multline}\label{eq.repseneuLA}
s_\e(x) = - \int_{\Gamma_N}{ A(x) \nabla_y L_{A(x)}(x,y) \cdot n(y) s_\e(y)\:\d s(y) }   + \int_{{\mathcal O}}{(A(x)-A(y)) \nabla_y L_{A(x) }(x,y) \cdot \nabla s_\e(y) \: \d y} \\ 
 + \int_{\partial {\mathcal O} \setminus U}{(A\nabla s_\e \cdot n)(y) \: L_{A(x) }(x,y) \: \d s(y)} - \int_{\D_\e}{A(x) \nabla_y L_{A(x)}(x,y) \cdot n(y) s_\e(y) \: \d s(y)}.
  \end{multline}
We now wish to take the trace of a co-normal derivative of  the above identity on $\D_\e$. This is possible owing to the next lemma, whose proof is postponed to the end of this section. 
 \begin{lemma}\label{lem.estdivop}
Let us define the operator $M_\e: \widetilde{H}^{1/2}(\D_\e) \to H^1({\mathcal O})$ by
$$ M_\e \varphi(x) = \int_{\D_\e}{A(x) \nabla_y L_{A(x)}(x,y) \cdot n(y) \varphi(y) \: \d s(y)}~.$$
There exists a constant $C$, depending only on the matrix field $A(x)$ and the domain ${\mathcal O}$ such that, for all $\varphi \in \widetilde{H}^{1/2}(\D_\e)$, 
\begin{equation*}\label{eq.estMeHdiv}
 || M_\e \varphi ||_{H^1({\mathcal O})} + || \dv_x(A(x) \nabla_x (M_\e \varphi))||_{L^2({\mathcal O})} \leq C || \varphi||_{\widetilde{H}^{1/2}(\D_\e)}~.
 \end{equation*}
\end{lemma} 
  
%  
%  Finally, letting $x$ tend to $\D_\e$, we end up with the equation, for all $x \in \D_\e$: 
%   \begin{multline}\label{eq.v0DeNeu}
%   s_\e(x) = - \int_{\Gamma_N}{ A(x) \nabla_y L_{A(x)}(x,y) \cdot n(y) s_\e(y)\:\d s(y) }   + \int_{{\mathcal O}}{(A(x)-A(y)) \nabla_y L_{A(x) }(x,y) \cdot \nabla s_\e(y) \: \d y} \\ 
%+ \int_{\partial {\mathcal O} \setminus {\mathcal U}}{(A\nabla s_\e \cdot n)(y) \: L_{A(x) }(x,y) \: \d s(y)} - \int_{\D_\e}{A(x) \nabla_y L_{A(x)}(x,y) \cdot n(y) s_\e(y) \: \d s(y)},
%  \end{multline}
Using this lemma we obtain the following identity between elements of $H^{-1/2}(\D_\e)$:
\begin{eqnarray}\label{eq.ANv0DeNeu}
   -A(x)\nabla v_0(x)\cdot n(x) &=& - \int_{\Gamma_N}{  A(x) \nabla_x\left( A(x) \nabla_y L_{A(x)}(x,y) \cdot n(y)\right) \cdot n(x) \: s_\e(y)\:\d s(y) }  \nonumber \\
    &&\hskip 10pt  + A(x) \nabla \left( \int_{{\mathcal O}}{(A(x)-A(y)) \nabla_y L_{A(x) }(x,y) \cdot \nabla s_\e(y) \: \d y}\right)\cdot n(x) \\
&&\hskip 20pt +  \int_{\partial {\mathcal O} \setminus {\mathcal U}}{(A\nabla s_\e \cdot n)(y) \: A(x) \nabla_x ( L_{A(x) }(x,y)) \cdot n(x) \: \d s(y)}  \nonumber \\ 
&&\hskip 30pt - \int_{\D_\e}{\Big(A(x) \nabla_x \left(A(x) \nabla_y L_{A(x)}(x,y) \cdot n(y) \right) \cdot n(x) \Big)\: s_\e(y) \: \d s(y)}~. \nonumber
\end{eqnarray}
We rewrite the latter as
\begin{equation}\label{eq.ANv0DeNeu2}
   -A(x)\nabla v_0(x)\cdot n(x) = R_\e^1(x) + A(x) \nabla K_\e(x) \cdot n(x) + R_\e^2(x) -  \int_{\D_\e}{P(x,y)\: s_\e(y) \: \d s(y)}~, 
\end{equation}
where we have defined the following quantities on $\D_\e$
\begin{equation}\label{eq.Re1}
R_\e^1(x) :=  - \int_{\Gamma_N}{  A(x) \nabla_x\left( A(x) \nabla_y L_{A(x)}(x,y) \cdot n(y)\right) \cdot n(x) \: s_\e(y)\:\d s(y) }~,
\end{equation}
$$ K_\e(x) :=  \int_{{\mathcal O}}{(A(x)-A(y)) \nabla_y L_{A(x) }(x,y) \cdot \nabla s_\e(y) \: \d y}~,$$
$$ R_\e^2(x) :=  \int_{\partial {\mathcal O} \setminus {U}}{(A\nabla s_\e \cdot n)(y) \: A(x) \nabla_x ( L_{A(x) }(x,y)) \cdot n(x) \: \d s(y)}~,$$
and the kernel
 \begin{equation}\label{eq.defNneu}
  P(x,y) := A(x) \nabla_x( A(x)\nabla_y L_{A(x) }(x,y) \cdot n(y))\cdot n(x)~, \quad x, y \in \D_\e~.
  \end{equation}
Rescaling \cref{eq.ANv0DeNeu2}, we finally arrive at the following integral equation on $\D_1$  
  \begin{equation}\label{eq.inteqNeu}
   T_\e \varphi_\e =  \gamma(0) \frac{\partial u_0}{\partial n}(0) + \eta_\e,
   \end{equation}
     where the unknown $\varphi_\e \in \widetilde{H}^{1/2}(\D_1)$ is the quantity introduced in \cref{eq.vphieNeu}, 
   the operator  $T_\e: \widetilde{H}^{1/2}(\D_1) \to H^{-1/2}(\D_1) $ is defined by 
 \begin{equation}\label{eq.defTeNeu}
  T_\e \varphi (x) =  \int_{\D_1}{ \varphi(z) \: P(\e x,\e z) \: \d s(z)}~,  
  \end{equation}
  and the remainder $\eta_\e \in H^{-1/2}(\D_1)$ is given by 
  \begin{equation}\label{eq.defetaeneu} 
 \eta_\e(x) =  \Big( (A\nabla v_0\cdot n)(\e x)  - \gamma(0)\frac{\partial u_0}{\partial n}(0)  \Big)
 + R_\e^1(\e x)  + R_\e^2(\e x)
  + (A\nabla K_\e \cdot n)(\e x)~.
 \end{equation}

 \par\medskip  
 \noindent \textit{Step 3.} We study the integral equation \cref{eq.inteqNeu} to obtain information about the limiting behavior of $\varphi_\e$ as $\e \to 0$.
To this end, we estimate the remainder $\eta_\e$ and we approximate the operator $T_\e$; this is possible due to the following lemmata, whose proofs are detailed at the end of this section.
 
  \begin{lemma}\label{lem.etaeneu}
 The remainder term $\eta_\e \in H^{-1/2}(\D_1)$ defined in \cref{eq.defetaeneu} satisfies
 \begin{equation}\label{eq.cvetaeps}
  \eta_\e \xrightarrow{\e \to 0} 0 \text{ weakly in } H^{-1/2}(\D_1)~.
  \end{equation}
 \end{lemma} 
 
 \begin{lemma}\label{lem.approxd2N}
The operator $T_\e$ in \cref{eq.defTeNeu} satisfies the following expansion
\begin{equation}\label{eq.asym2dN}
\sup\limits_{\varphi \in \widetilde{H}^{1/2}(\D_1) \atop || \varphi ||_{\widetilde{H}^{1/2}(\D_1) } \leq 1} \: \e^d \left\lvert \left\lvert T_\e \varphi -
\frac{2\gamma(0)}{\e^d} R_1 \varphi \right\lvert  \right\lvert_{H^{-1/2}(\D_1)} \xrightarrow{\e \to 0} 0~,
\end{equation}
where the hypersingular operator $R_1 : \widetilde{H}^{1/2}( \D _1) \to H^{-1/2}(\D_1)$ is defined by
\begin{equation}\label{eq.R1}
R_1\varphi(x) = \left\{
\begin{array}{>{\displaystyle}c>{\displaystyle}l}
 \frac{1}{2\pi}\int_{\D_1}{\frac{1}{|x-y|^2} \varphi (y) \: \d s(y)} & \text{if } d = 2~, \\
 \frac{1}{4\pi}\int_{\D_1}{\frac{1}{|x-y|^3} \varphi (y) \: \d s(y)} & \text{if } d = 3~,
\end{array}
\right.
\end{equation}
and the above integrals are understood as finite parts; see \cref{sec.pot}.
\end{lemma}

Inserting the approximation \cref{eq.asym2dN} in the integral equation \cref{eq.inteqNeu}, 
the function $\varphi_\e \in \widetilde{H}^{1/2}(\D_1)$ satisfies: 
\begin{equation}\label{eq.inteqd2N}
R_1 \varphi_\e + R_\e \varphi_\e  = \frac{\e^d}{2} \frac{\partial u_0}{\partial n}(0) + \e^d\widetilde{\eta_\e}~, 
\end{equation}
for some sequence $\widetilde{\eta_\e} \in H^{-1/2}(\D_1)$ which converges weakly to $0$, and some operators $R_\e: \widetilde{H}^{1/2}(\D_1) \to H^{-1/2}(\D_1)$, which converge to zero in the operator norm. 
This integral equation can now be solved owing to the next lemma, whose proof is also postponed to the end of this section.
\begin{lemma}\label{lem.hypersinginvert}
The operator $R_1: \widetilde{H}^{1/2}(\D_1) \to H^{-1/2}(\D_1)$ defined in \cref{eq.R1} is invertible.
\end{lemma}
Using this result together with Neumann series to invert the integral equation \cref{eq.inteqd2N}, we obtain the existence of a constant $C >0$ such that 
\begin{equation}\label{eq.estvphied2N}
 || \varphi_\e ||_{\widetilde{H}^{1/2}(\D_1)} \leq C\e^d, 
 \end{equation}
as well as the following asymptotic expansion
\begin{equation}\label{eq.expmoyvphieneu}
 \langle \varphi_\e , 1\rangle = \frac{\e^d}{2} \frac{\partial u_0}{\partial n}(0) \langle R_1^{-1}1,1\rangle + \o(\e^d)~,
 \end{equation}
where the explicit expression for the constant $\langle R_1^{-1}1,1\rangle$ is given by \cref{prop.eqdistrib} (ii), (iv).\par\medskip

\noindent \textit{Step 4.} We pass to the limit in the representation formula \cref{eq.repuecalcneu} for $u_\e$. 
Arguing as in the proof of \cref{th.expdir2d}, that is, combining a Taylor expansion of the function $z \mapsto \gamma(T(\e z)) \frac{\partial N}{\partial n_y}(x,T(\e z))$ with the estimate \cref{eq.estvphied2N}, we obtain: 
$$ \begin{array}{>{\displaystyle}cc>{\displaystyle}l}
 u_\e(x) &=& u_0(x) - \left(\int_{\D_1}{ \varphi_\e(z) \: \d s(z)} \right)\gamma(0) \frac{\partial N}{\partial n_y}(x,0) + \o(\e^d)~,\\
 &=& u_0(x) -  \frac{\e^d}{2} \langle R_1^{-1}1,1\rangle\gamma(0) \frac{\partial u_0}{\partial n}(0) \frac{\partial N}{\partial n_y}(x,0) + \o(\e^d)~,
 \end{array}
 $$
where the second line follows from \cref{eq.expmoyvphieneu}. 
The explicit expressions for the constant $\langle R_1^{-1}1,1\rangle$ in 2d and 3d provided in \cref{prop.eqdistrib} (ii), (iv) lead to the statement of the Theorem. 
\end{proof}

We conclude this section with the missing arguments in the above proof.  

\begin{proof}[Proof of \cref{lem.estdivop}]
The intuition behind the technical argument below is the following: if $M_\e \varphi$ was the double layer potential associated with the operator $u \mapsto \dv_y(A(y)\nabla_y u)$ (see \cref{sec.pot}), the quantity $\dv(A\nabla (M_\e \varphi))$ would vanish exactly on ${\mathcal O}$. 
Unfortunately, this is not the case since $(x,y) \mapsto L_{A(x)}(x,y)$ is not the fundamental solution of this operator. However,
the following calculations show that $M_\e \varphi$ is ``not too far'' from this double layer potential, so that 
the terms of highest-order derivatives vanish in the expression of $\dv(A\nabla (M_\e \varphi))$, 
and the lower-order terms can be controlled. 

Before starting, let us introduce some notations. For the sake of clarity, we denote by
$ L(A,x,y) := L_A(x,y)$ the function defined in \cref{lem.LBneu}. The corresponding partial derivatives with respect to the entries $a_{jk}$ ($j,k=1,\ldots d$) of the matrix $A$, and with respect to the components $x_i$, $y_i$ of $x$ and $y$ ($i,=1,\ldots,d$) are denoted by $\frac{\partial L}{\partial a_{jk}}$, $\frac{\partial L}{\partial x_i}$, $\frac{\partial L}{\partial y_i}$. Throughout the proof, $r(\varphi) \in L^2({\mathcal O})$ stands for a remainder term, which may vary from one line to the other, but which consistently satisfies the following estimate 
$$ || r(\varphi) ||_{L^2({\mathcal O})} \leq C || \varphi ||_{\widetilde{H}^{1/2}(\D_e)}~.$$ 
 
At first, using the expression for $L(A(x),x,y)$ given in \cref{lem.LBneu}, we calculate
\begin{equation}\label{eq.AdLAdn}
 A(x) \nabla_y L(A(x), x , y) \cdot n(y) = \left\{
\begin{array}{cl}
 \frac{-1}{\pi \sqrt{\det A(x)}} \frac{ (y- x) \cdot n(y)}{|M( x)(y- x)|^2} & \text{if } d =2, \\
 \frac{-1}{2\pi \sqrt{\det A(x)}} \frac{ (y- x) \cdot n(y)}{|M( x)(y- x)|^3} & \text{if } d =3,
 \end{array}\right. \quad  x\in \R^d, \:\: y \in \D_\e, \:\: x \neq y~.   
\end{equation}
Recalling from \cref{sec.Hs} the definition of the space $\widetilde{H}^{1/2}(\D_\e)$, and notably the fact that the associated norm is $|| u ||_{\widetilde{H}^{1/2}(\D_\e)}  =  || u ||_{{H}^{1/2}(\partial {\mathcal O})} $, \cref{th.TKpot} then implies that $M_\e \varphi \in H^1({\mathcal O})$ and that there exists a constant $C >0$ independent of $\e$ such that 
$$ || M_\e \varphi ||_{H^1({\mathcal O})} \leq  C || \varphi||_{\widetilde{H}^{1/2}(\D_\e)}~.$$
We now proceed to prove the estimate 
 \begin{equation}\label{eq.estdivMe}
  || \dv_x (A(x) \nabla_x (M_\e \varphi))||_{L^2({\mathcal O})} \leq C || \varphi||_{\widetilde{H}^{1/2}(\D_\e)}~.
  \end{equation}
For an arbitrary point $x \in {\mathcal O}$, the definition of $M_\e \varphi$ boils down to:
$$ M_\e \varphi (x) = \sum\limits_{i,j=1}^d a_{ij}(x) \int_{\D_\e} { \frac{\partial L}{\partial y_j} (A(x),x,y ) n_i(y) \varphi(y)  \: \d s(y)}~.$$
Since $n_i(y) = 0$ on $\D_\e$ for $i=1,\ldots,d-1$, and since the matrix field $A(x)$ is smooth and the function $y \mapsto L(A(x),x,y)$ satisfies homogeneous Dirichlet boundary conditions on $\partial H$, 
the above expression actually simplifies into 
$$M_\e \varphi (x) = a_{dd}(x)  \int_{\D_\e} { \frac{\partial L}{\partial y_d} (A(x),x,y )  \varphi(y)  \: \d s(y)}~. $$
Taking derivatives, we now get, for $x \in {\mathcal O}$, and $i=1,\ldots,d$~,
\begin{equation}\label{eq.dupsi} 
\begin{array}{>{\displaystyle}cc>{\displaystyle}l} 
\frac{\partial}{\partial x_i}(M_\e\varphi) (x) &=& \frac{\partial a_{dd}}{\partial x_i}(x)  \int_{\D_\e} { \frac{\partial L}{\partial y_d} (A(x),x,y )  \varphi(y)  \: \d s(y)}  \\&& \hskip 20pt + a_{dd}(x) \frac{\partial }{\partial x_i} \left( \int_{\D_\e} { \frac{\partial L}{\partial y_d} (A(x),x,y )  \varphi(y)  \: \d s(y)} \right) =: V_i(\varphi) (x) + W_i(\varphi)(x)~,
\end{array}
\end{equation}
with obvious notations. We infer from this expression that
\begin{equation}\label{eq.divAupsi}
 \dv_x(A(x)\nabla_x (M_\e \varphi)) = \dv(A(x) V(\varphi)) + \dv(A(x) W(\varphi)).
 \end{equation}
 Each function $V_i(\varphi)$ is the multiple of a smooth function with a potential associated to the kernel $\frac{\partial L}{\partial y_d}(A(x),x,y)$, $y\in \D_\e$. 
A simple calculation, similar to \cref{eq.AdLAdn}, reveals that the latter is (the restriction of) a homogeneous kernel of class $0$ in the sense of \cref{def.homkernel}. It then follows from \cref{th.TKpot} that 
 $$ || V(\varphi) ||_{H^1({\mathcal O})^d} \leq  C || \varphi ||_{H^{1/2}(\partial {\mathcal O})} = C || \varphi ||_{\widetilde{H}^{1/2}(\D_\e)}~,$$
 and so 
 \begin{equation}\label{eq.estdivAV}
  || \dv (A(x) V(\varphi)) ||_{L^2({\mathcal O})} \leq C || \varphi ||_{\widetilde{H}^{1/2}(\D_\e)}~. 
  \end{equation} 
We then focus on the second term $\dv(A(x)W(\varphi))$ in \cref{eq.divAupsi}. A straightforward calculation yields, for $x\in{\mathcal O}$,
 \begin{equation*}
\begin{array}{>{\displaystyle}cc>{\displaystyle}l} 
\dv(A(x)W(\varphi))(x)&=& \sum\limits_{i,j=1}^d{\frac{\partial }{\partial x_i}(a_{ij}a_{dd})(x) \frac{\partial}{\partial x_j} \left( \int_{\D_\e} { \frac{\partial L}{\partial y_d} (A(x),x,y )  \varphi(y)  \: \d s(y)} \right) }  \\
&& + \:\: a_{dd}(x) \sum\limits_{i,j=1}^d{ a_{ij} (x) \frac{\partial^2}{\partial x_i\partial x_j} \left( \int_{\D_\e} { \frac{\partial L}{\partial y_d} (A(x),x,y )  \varphi(y)  \: \d s(y)} \right) } \\ 
&=&  r ( \varphi) + a_{dd}(x) \sum\limits_{i,j=1}^d{ a_{ij} (x) \frac{\partial^2}{\partial x_i\partial x_j} \left( \int_{\D_\e} { \frac{\partial L}{\partial y_d} (A(x),x,y )  \varphi(y)  \: \d s(y)} \right) }~,
\end{array}
\end{equation*}
where we have used a similar argument to that used in the treatment of the functions $V_i(\varphi)$ to pass from the first line to the second. 
Using the chain rule to proceed, we obtain 
 \begin{equation}\label{eq.divAW}
\begin{array}{>{\displaystyle}cc>{\displaystyle}l} 
\dv(A(x)W(\varphi)) &=& a_{dd}(x) \sum\limits_{i,j=1}^d a_{ij}(x) \frac{\partial }{\partial x_i} \Bigg( \sum\limits_{k,l=1}^d \frac{\partial a_{kl}}{\partial x_j}(x)  \int_{\D_\e} { \frac{\partial^2 L}{\partial a_{kl} \partial y_d} (A(x),x,y )  \varphi(y)  \: \d s(y)}   \\
&& +  \int_{\D_\e} { \frac{\partial^2 L}{\partial x_j \partial y_d} (A(x),x,y )  \varphi(y)  \: \d s(y)}  \Bigg) + r(\varphi) \\ 
 &=& a_{dd}(x) \sum\limits_{i,j=1}^d \sum\limits_{k,l=1}^d a_{ij}(x)  \frac{\partial^2 a_{kl}}{\partial x_i\partial x_j}(x)  \int_{\D_\e} { \frac{\partial^2 L}{\partial a_{kl} \partial y_d} (A(x),x,y )  \varphi(y)  \: \d s(y)}   \\
  &&+ a_{dd}(x) \sum\limits_{i,j=1}^d \sum\limits_{k,l=1}^d a_{ij}(x)  \frac{\partial a_{kl}}{\partial x_j}(x) \frac{\partial}{\partial x_i} \left( \int_{\D_\e} { \frac{\partial^2 L}{\partial a_{kl} \partial y_d} (A(x),x,y )  \varphi(y)  \: \d s(y)} \right)  \\  
  &&+ a_{dd}(x) \sum\limits_{i,j=1}^d \sum\limits_{k,l=1}^d a_{ij}(x)  \frac{\partial a_{kl}}{\partial x_i}(x) \int_{\D_\e} { \frac{\partial^3 L}{\partial a_{kl} \partial x_j \partial y_d} (A(x),x,y )  \varphi(y)  \: \d s(y)}   \\  
&& +  a_{dd}(x) \sum\limits_{i,j=1}^d a_{ij}(x) \int_{\D_\e} { \frac{\partial^3 L}{\partial x_i \partial x_j \partial y_d} (A(x),x,y )  \varphi(y)  \: \d s(y)}  + r(\varphi), \\[1.5em]
&=:& a_{dd}(x)\Big(Z_1(\varphi) + Z_2(\varphi) + Z_3(\varphi) + Z_4(\varphi )\Big)+ r(\varphi)~, 
\end{array}
\end{equation}
with obvious notations for $Z_m$, $m=1,~\ldots~,4.$ 

 We now remark that the function $Z_1(\varphi)$ is a linear combination of integral operators with smooth coefficients; the kernels of these operators are $\frac{\partial^2 L}{\partial a_{kl} \partial y_d}$, $y\in \D_\e$, and they are (restrictions of) homogeneous kernels of class $0$ in the sense of \cref{def.homkernel}. Here, we use the fact that taking derivatives with respect to one of the matrix entries $a_{kl}$ changes neither the order, the homogeneity, nor the parity of the function involved. It then follows from \cref{th.TKpot} that 
 \begin{equation}\label{eq.estZ1}
  || Z_1(\varphi) ||_{H^1({\mathcal O})} \leq C || \varphi ||_{\widetilde{H}^{1/2}(\D_\e)}~.
  \end{equation}
  By the same token, we obtain
  \begin{equation}\label{eq.estZ2}
 || Z_2(\varphi) ||_{L^2({\mathcal O})} \leq C || \varphi ||_{\widetilde{H}^{1/2}(\D_\e)}~.
 \end{equation}
In order to estimate $Z_3(\varphi)$, we rewrite this quantity as
\begin{multline*}
Z_3(\varphi) = \sum\limits_{i,j=1}^d \sum\limits_{k,l=1}^d a_{ij}(x)  \frac{\partial a_{kl}}{\partial x_i}(x) \frac{\partial}{\partial x_j} \left( \int_{\D_\e} { \frac{\partial^2 L}{\partial a_{kl}  \partial y_d} (A(x),x,y )  \varphi(y)  \: \d s(y)} \right) \\
-  \sum\limits_{i,j=1}^d \sum\limits_{k,l=1}^d \sum\limits_{k^\prime,l^\prime=1}^d a_{ij}(x)  \frac{\partial a_{kl}}{\partial x_i}(x) \frac{\partial a_{k^\prime l^\prime}}{\partial x_j}(x)   \int_{\D_\e} { \frac{\partial^3 L}{\partial a_{kl}  \partial a_{k^\prime l^\prime}\partial y_d} (A(x),x,y )  \varphi(y)  \: \d s(y)}~,
\end{multline*} 
and arguing as above, we obtain 
 \begin{equation}\label{eq.estZ3}
 || Z_3(\varphi) ||_{L^2({\mathcal O})} \leq C || \varphi ||_{\widetilde{H}^{1/2}(\D_\e)}~.
\end{equation}
This leaves us with the task of estimating $Z_4(\varphi)$. 
To accomplish this, we rewrite the equation \cref{eq.defLBdir} satisfied by $L(A,x,y)$ as
  \begin{equation*}
   \sum\limits_{i,j=1}^d a_{ij} \frac{\partial ^2 L}{\partial y_i \partial y_j} (A,x,y) = 0, \quad x, y  \in \overline{H}, \:\: x \neq y~;
   \end{equation*}
note that this holds for an arbitrary, symmetric, positive definite matrix $A \in \R^{d\times d}$, with entries $a_{ij}$. Due to the symmetry property
$$\forall x,y \in \overline{H}, \: x \neq y, \quad L(A,x,y) = L(A,y,x)~, $$
it follows that 
  \begin{equation*}
   \sum\limits_{i,j=1}^d a_{ij} \frac{\partial ^2 L}{\partial x_i \partial x_j} (A,x,y) = 0, \quad  x, y  \in \overline{H}, \:\: x \neq y~.
   \end{equation*}
Substituting $A(x)=\left\{a_{ij}(x) \right\}$ for $A$ and taking a derivative with respect to the $y_d$ variable, we get
   $$   \sum\limits_{i,j=1}^d a_{ij}(x) \frac{\partial ^3 L}{\partial x_i \partial x_j \partial y_d} (A(x),x,y) =0~.$$
It follows that
  \begin{equation}\label{eq.estZ4}
Z_4(\varphi) = 0~. 
\end{equation}
Combining the estimates \cref{eq.estZ1,eq.estZ2,eq.estZ3,eq.estZ4} with \cref{eq.divAupsi,eq.estdivAV,eq.divAW} we obtain the desired conclusion.
\end{proof}

We proceed with the proof of \cref{lem.etaeneu}. 
\begin{proof}[Sketch of the proof of \cref{lem.etaeneu}] 
Let us denote by $I_\e^i(x)$, $i=1,\ldots,4$ the four terms in the right-hand side of \cref{eq.defetaeneu}. We prove that each of these contributions 
tends to $0$ weakly in $H^{-1/2}(\D_1)$ as $\e \to 0$. 

At first, since $v_0$ is smooth and $(A\nabla v_0 \cdot n)(0) = \gamma(0) \frac{\partial u_0}{\partial n}(0)$, 
the difference 
$$I_\e^1(x) = \Big(  (A\nabla v_0\cdot n)(\e x) - \gamma(0)\frac{\partial u_0}{\partial n}(0) \Big)$$
 is easily seen to converge to $0$ strongly in $H^{-1/2}(\D_1)$. Furthermore, since the support of the integral 
$$R_\e^1(x)=  - \int_{\Gamma_N}{  A(x) \nabla_x\left( A(x) \nabla_y L_{A(x)}(x,y) \cdot n(y)\right) \cdot n(x) \: s_\e(y)\:\d s(y) }$$
 is ``far'' from $\D_\e$, 
the convergence properties of $s_\e$ expressed in \cref{prop.prelestneu} (\cref{rem.prelestneuaniso}) and \cref{eq.capaneuDe} imply that $R_\e^1(x)$ converges to $0$ uniformly for $x$ in a fixed neighborhood of the sets $\D_\e$; 
in particular, $I_\e^2(x) = R_\e^1(\e x)$ converges to $0$ strongly in $H^{-1/2}(\D_1)$.
The same argument shows that $I_\e^3(x) = R_\e^2(\e x)$ also converges to $0$ strongly in $H^{-1/2}(\D_1)$.\par\medskip

This leaves us with the task of proving that 
$ I_\e^4(x) = (A\nabla K_\e \cdot n)(\e x)$ converges to $0$ weakly in $H^{-1/2}(\D_1)$. 
Let us introduce a smooth bounded domain $D \subset {\mathcal O}$, whose boundary contains $\D_1$. Furthermore, select $D$ so that $D$ is bounded away from $\Gamma_N$ and $\partial {\mathcal O}\setminus U$.
Recalling the definition of $H^{-1/2}(\D_1)$ as the space of restrictions to $\D_1$ of distributions in $H^{-1/2}(\partial {\mathcal O})$ (see \cref{sec.Hs}), 
it suffices to show that the vector-valued function
$$ \sigma_\e(x) := (A\nabla K_\e)(\e x), \quad x \in D~,$$
converges to $0$ weakly in the Hilbert space 
$$ H_{\dv}(D) := \left\{ \sigma \in L^2(D)^d, \:\: \dv \sigma \in L^2(D) \right\}~.$$
We proceed in two steps to achieve this. \par\medskip

\noindent \textit{Step 1.} We prove that $\sigma_\e$ is a bounded sequence in $H_{\dv}(D)$. 
To this end, we return to \cref{eq.repseneuLA}, which, for $x \in D \subset {\mathcal O}$, reads 
 \begin{multline}\label{eq.repseneuLA2}
s_\e(x) = - \int_{\Gamma_N}{ A(x) \nabla_y L_{A(x)}(x,y) \cdot n(y) s_\e(y)\:\d s(y) }   \\
 + \int_{\partial {\mathcal O} \setminus U}{(A\nabla s_\e \cdot n)(y) \: L_{A(x) }(x,y) \: \d s(y)} + K_\e(x) - M_\e s_\e (x),
  \end{multline}
with
$$K_\e(x) = \int_{{\mathcal O}}{(A(x)-A(y)) \nabla_y L_{A(x) }(x,y) \cdot \nabla s_\e(y) \: \d y}~,$$
and the quantity $M_\e s_\e$ is as in \cref{lem.estdivop}.
It follows from \cref{lem.estdivop} that $M_\e s_\e$ satisfies the following estimate
$$ || M_\e s_\e(x) ||_{H^1({\mathcal O})} +  || \dv(A(x) \nabla (M_\e s_\e)) ||_{L^2({\mathcal O})} \leq C || s_\e ||_{H^{1/2}(\partial {\mathcal O})}~.$$
From \cref{eq.repseneuLA2}, and the fact that $D$ is bounded away from $\Gamma_N$ and $\partial {\mathcal O}\setminus U$, we now see that the function $ K_\e(x)$ satisfies the similar estimate
$$ || K_\e ||_{H^{1}(D)} + || \dv(A(x) \nabla K_\e) ||_{L^2(D)} \leq C || s_\e ||_{H^1({\mathcal O})}~.$$
Rescaling the above inequality (note that $\e D \subset D$ for $\e$ sufficiently small) and using the estimate
$$ || s_\e||_{H^1({\mathcal O})} \leq C \e^\frac{d}{2}~,$$
which follows readily from  \cref{prop.prelestneu} (\cref{rem.prelestneuaniso}) and \cref{eq.capaneuDe}, we now obtain
$$ \e^{\frac{d}{2}} || \sigma_\e ||_{L^2(D)^d} +  \e^{\frac{d-2}{2}} ||\dv \sigma_\e ||_{L^2(D)} \leq C \e ^{\frac{d}{2}}~. $$
Hence, $\sigma_\e$ is a bounded sequence in $H_{\dv}(D)$, and so, up to a subsequence (which we still index by $\e$) it converges weakly to a limit $\sigma^*$ in this space. 
\par\medskip
\noindent \textit{Step 2.} We prove that the weak limit $\sigma^*$ is $0$, and this task requires separating the cases $d=2$ and $d=3$. 

When $d=3$, we observe that, by definition, 
\begin{equation}\label{eq.relsigmaeKte}
 \sigma_\e(x) = A(\e x) \nabla \widetilde{K_\e}(x), \text{ where } \widetilde{K_\e}(x) := \frac{1}{\e} K_\e(\e x)~,
 \end{equation}
and the same calculation as in the proof of \cref{lem.etae} (see notably \cref{eq.estI4edir}) reveals that the quantity $K_\e(x)$ satisfies 
$$ |K_\e(x) | \leq C || s_\e ||_{H^1({\mathcal O})}, \text{ for all } x \in D~.$$
Hence, we obtain 
$$ \int_{D}{|K_\e (\e z)|^2 \: \d z} \leq C \e^3,$$
which proves that
$$ || \widetilde{K_\e} ||_{L^2(D)} \xrightarrow{\e \to 0} 0~.$$
It follows from \cref{eq.relsigmaeKte} and the continuity of derivatives in the sense of distributions that $\sigma^* = 0$. \par\medskip

The case where $d=2$ is a little more involved, and we need to estimate the quantity $K_\e$ more carefully. 
The argument performed for $d=3$ in this case only allows us to infer that $\widetilde{K}_\e$ is a bounded sequence in $L^2(D)$; 
we also know from Step 1 that its gradient is bounded in $L^2(D)^2$, and so (up to a subsequence) 
$\widetilde{K}_\e$ converges strongly to a function $K^* \in L^2(D)$, which we need to analyze further. 
For any point $x \in \R^d$ and positive real number $h >0$, we denote by $B(x,h)$ the open ball with radius $h$ centered at $x$. 

We observe that, for $x,y\in \D_\e$,  
\begin{multline*}
 |K_\e(x) - K_\e(y)| = \frac{1}{\pi} \left\lvert \int_{\mathcal O}\left((A(x)-A(z)) \frac{1}{ \sqrt{\det(A(x))}} \frac{M^2(x) (x-z)}{|M(x)(z- x)|^2} \cdot \nabla s_\e(z)  \right. \right.\\
 \left.\left. -(A( y)-A(z)) \frac{1}{ \sqrt{\det(A( y))}} \frac{M^2( y) (y- z)}{|M( y)(z- y)|^2} \cdot \nabla s_\e(z) \right) \: \d z \right\lvert~.
  \end{multline*}
 Denoting by $h := |x-y|$ we get, since $B(y,h) \subset B(x,2h) \subset B(y, 4h)$,
  $$\begin{array}{>{\displaystyle}cc>{\displaystyle}l}
 |K_\e(x) - K_\e(y)| & \leq & \frac{1}{\pi}  \int_{B(x, 2h)\cap {\mathcal O}}\left\lvert (A(x)-A(z)) \frac{1}{ \sqrt{\det(A(x))}} \frac{M^2(x) (x-z)}{|M( x)(z-x)|^2} \cdot \nabla s_\e(z)  \right\lvert   \: \d z\\
 & &+ \frac{1}{\pi}  \int_{B(y, 4h)\cap {\mathcal O}}\left\lvert (A(y)-A(z)) \frac{1}{ \sqrt{\det(A(y))}} \frac{M^2(y) (y-z)}{|M(y)(z- y)|^2} \cdot \nabla s_\e(z) \right\lvert  \: \d z  \\
 &&+ \frac{1}{\pi}  \int_{{\mathcal O}\setminus B(x,2h)} \left\lvert \left((A(x)-A(z)) \nabla_z L_{A(x)}(x,z)   -(A(y)-A(z)) \nabla_z L_{A(y)} (y , z)  \right)\cdot \nabla s_\e(z)\right\lvert  \: \d z \\
 &=:& J_1+J_2+J_3,
 \end{array}$$ 
 with obvious notations. 
 Due to the smoothness of the matrix field $A$
  $$
   |J_1 | \leq C \int_{B(x,2h)\cap {\mathcal O}}{|\nabla s_\e (z)| \: \d z}
 \leq C h ||\nabla s_\e ||_{L^2({\mathcal O})^d}~,  
   $$
 and a similar estimate holds for $J_2$. When it comes to $J_3$, we remark that for $z \notin B(x, 2h)$
 $$ 2h \leq |x - z |, \quad h \leq |y - z |, \text{ and } \frac{1}{2}|x - z | \leq |y -z| \leq \frac{3}{2} | x - z|~.$$
 We now decompose
$$ \left((A(x)-A(z)) \nabla_z L_{A(x)}(x,z)   -(A(y)-A(z)) \nabla_z L_{A(y)} (y , z)  \right) = b_1 + b_2~,$$
where 
$$ b_1 =  (A(x)-A(y)) \nabla_z L_{A(x)}(x,z) \text{ and } b_2 = (A(y)-A(z)) \left( \nabla_z L_{A(x)}(x,z) - \nabla_z L_{A(y)} (y , z) \right).$$
A simple calculation yields that
$$| b_1 | \leq \frac{C h }{|x - z|}~,$$
and regarding $b_2$, we calculate 
$$ \begin{array}{>{\displaystyle}cc>{\displaystyle}l}
| b_2| &\leq & C|z- y|  \left\lvert  \nabla_z L_{A(x)}(x,z) - \nabla_z L_{A(y)} (y , z)\right\lvert  \\
&= & C|z- x| \left\lvert  \frac{1}{ \sqrt{\det(A(x))}} \frac{M^2(x) (x-z)}{|M(x)(z-x)|^2} -  \frac{1}{ \sqrt{\det(A(y))}} \frac{M^2(y) (y-z)}{|M(y)(z-y)|^2} \right\lvert \\
&\leq& C |z- x|  \left\lvert  \frac{1}{ \sqrt{\det(A(x))}} -   \frac{1}{ \sqrt{\det(A(y))}}  \right\lvert  \left\lvert \frac{M^2(x) (x-z)}{|M(x)(z-x)|^2} \right\lvert \\
&&+C|z- x|  \left\lvert   \frac{M^2(x) (x-z)}{|M(x)(z-x)|^2} -  \frac{M^2(y) (y-z)}{|M(y)(z- y)|^2} \right\lvert \\
&\leq & Ch + C \frac{|z-  x| }{|M(x)(z-x)|^2} \left\lvert  M^2(x) (x-z) - M^2( y) (y- z) \right\lvert \\
&&+ C|z-  x|  |M^2( y) (y- z)| \left\lvert  \frac{1}{|M( x)(z- x)|^2}  -  \frac{1}{|M( y)(z- y)|^2} \right\lvert  \\
&\leq & Ch + C \frac{|z-  x| }{|M( x)(z-x)|^2} \left\lvert  M^2( y) ( x-y)  \right\lvert + C|z-  x|^2  \left\lvert  \frac{1}{|M( x)(z- x)|^2}  -  \frac{1}{|M( y)(z- y)|^2} \right\lvert \\
&\leq & Ch +  \frac{Ch}{|z-  x| }  + C\frac{1}{|z-  x|^2}  \left\lvert  |M( x)(z- x)|^2  -  |M( y)(z- y)|^2 \right\lvert \\
&\leq & Ch +  \frac{Ch}{|z-  x| }  + C\frac{1}{|z-  x|}  \left\lvert  |M( x)(z- x)|  -  |M( y)(z- y)| \right\lvert \hskip 10pt \leq   Ch+ \frac{Ch}{|z-  x| }~. 
\end{array}$$ 
Summarizing, we now have 
$$ |J_3| \leq Ch \int_{{\mathcal O} \setminus B( x,2h)} {\left(\frac{1}{|z -  x| }+1 \right) |\nabla s_\e| \: \d z} ~,$$
and so
\begin{equation*}
 \begin{array}{>{\displaystyle}cc>{\displaystyle}l}
 |J_3| &\leq& Ch \left( \int_{{\mathcal O} \setminus B(x,2h)}{ \frac{1}{|z- x|^2} \:\d z} \right)^{\frac{1}{2}} || \nabla s_\e ||_{L^2({\mathcal O})^2} \\ 
 &\leq& Ch \left( \int_{2h}^M{ \frac{\d r} {r}} \right)^{\frac{1}{2}} || \nabla s_\e ||_{L^2({\mathcal O})^2} \\
 &\leq& C h |\log h |^{\frac{1}{2}} || \nabla s_\e ||_{L^2({\mathcal O})^2}~.
\end{array}
 \end{equation*}
With $x$ and $y$ replaced by $\e x$ and $\e y$, for $x, y \in \D$, we now conclude 
$$ |K_\e(\e x) - K_\e(\e y) | \leq C\e | x -  y|\big| \log|\e x - \e y| \big|^{1/2}\,|| \nabla s_\e ||_{L^2({\mathcal O})^2}~,$$
and so
$$  \begin{array}{>{\displaystyle}cc>{\displaystyle}l}
 |\widetilde{K}_\e(x) - \widetilde{K}_\e(y) |^2 &\leq& C |x-y|^2  (|\log \e | + \big|\log | x - y| \big|\,)\, || \nabla s_\e ||^2_{L^2({\mathcal O})^2}\\
 &\leq& C \e^2 \left( |x-y|^2 |\log \e | +|x-y|^2 \big|\log |x-y|\big|\right) ~,
 \end{array} $$
 where we have used again \cref{prop.prelestneu} (\cref{rem.prelestneuaniso}) and \cref{eq.capaneuDe} to estimate $|| \nabla s_\e ||^2_{L^2({\mathcal O})^2}$.
Integrating the terms in the previous inequality and passing to the limit as $\e \to 0$, we obtain
$$\int_{D}\int_{D}{|K^*(x) - K^*(y)|^2 \: \d x \d y} = \lim_{\e \rightarrow 0}\int_{D}\int_{D}{|\widetilde{K}_\e(x) - \widetilde{K}_\e(y)|^2 \: \d x \d y}=0~,$$
which proves that $K^*$ is a constant function over $D$. This completes the proof of the fact that $\sigma^*=0$, for $d=2$.

\end{proof}

\begin{proof}[Proof of \cref{lem.approxd2N}]
We only provide the proof in the two-dimensional case, the three-dimensional proof being very similar.

Using the definition of the fundamental solution $L_A(x,y)$ given by \cref{lem.LBneu}, we get, for arbitrary $y \in \D_\e$ and $x \in \overline{H}$, $x \neq y$,
$$ A(x) \nabla_y L_{A(x)}(x,y) \cdot n(y) = -\frac{1}{\pi\sqrt{\det A(x)}} \frac{y_2 - x_2}{|M(x)(y-x)|^2}~.$$
Hence, a straightforward calculation yields the following expression of the kernel $P(\e x,\e y)$ of the operator $T_\e$, defined in \cref{eq.defNneu}
$$ P(\e x,\e y) =  \frac{1}{\pi\sqrt{\det A(\e x)}}\frac{A(\e x)e_2 \cdot e_2}{|M(\e x) (\e y-\e x)|^2} = \e^{-2} \frac{A(\e x)e_2 \cdot e_2}{\pi\sqrt{\det A(\e x)} |M(\e x) e_1 \cdot e_1|^2}\frac{1}{|y-x|^2}~ , \quad x, y \in \D_1, \:\: x\neq y~,$$
and this immediately leads to 
$$
T_\e \varphi (x) -\frac{2 \gamma(0)}{\e^2} R_1 \varphi (x) =\e^{-2}\left( \frac{2 A(\e x)e_2 \cdot e_2}{\sqrt{\det A(\e x)} |M(\e x) e_1 \cdot e_1|^2} -2 \gamma(0)\right) R_1 \varphi (x)~.
$$
Since the matrix fields $A(x)$ and $M(x)$ are smooth, with values $A(0) = \gamma(0)\I$ and $M(0)=\gamma(0)^{-1/2}I$ at $x =0$, we have that 
$$
\left\lvert\left\lvert \frac{2 A(\e x)e_2 \cdot e_2}{\sqrt{\det A(\e x)} |M(\e x) e_1 \cdot e_1|^2} - 2\gamma(0)\right\lvert\right\lvert_{{\mathcal C}^1(\overline{\D_1})} \le C\e~;
$$
in order to verify \cref{lem.approxd2N} it thus suffices to show that the operator 
$$
R_1 \varphi (x) = \frac{1}{2 \pi} \int_{\D_1}  \frac {1}{|x-y|^2} \varphi(y)\, ds(y)
$$
(interpreted in terms of finite parts) is a bounded operator from $\widetilde{H}^{1/2}(\D_1)$ into $H^{-1/2}(\D_1)$. For this purpose we can, unfortunately, not directly use the results from \cref{sec.homogkernel}, since the hypersingular kernel of the above operator does not fit within that framework. To remedy this, we rely on a classical trick for hypersingular operators of the form $R_1$, using an alternate representation in terms of a homogeneous kernel operator, and a surface differentiation operator (see e.g. \cite{hsiao2008boundary}, \S 1.2). More precisely, we observe that
$$ \frac{1}{|x-y |^2 } = - \frac{\partial}{\partial y_1} \left( \frac{y_1-x_1}{| x-y|^2}\right) \text{ for } x,y \in \D_1, \:\: x\neq y,$$
due to the fact that $|x-y| = |x_1-y_1 |$ when $x,y \in \D_1$. 
It follows that, for an arbitrary density $\varphi \in \widetilde{H}^{1/2}(\D_1)$, 
$$ R_1 \varphi = \frac{1}{2 \pi}\int_{\D_1}{\frac{y_1-x_1}{|x-y|^2} \frac{\partial \varphi}{\partial y_1}(y)  \:\d s(y)}~,$$
where the right hand side represents a Cauchy principal value. The kernel $\frac{y_1-x_1}{|x-y|^2}$ fits within the framework of \cref{sec.homogkernel}, and it gives rise to an operator of class $0$, {\it i.e.}, a bounded operator from $\widetilde{H}^{-1/2}(\D_1)$ into $H^{-1/2}(\D_1)$. Since the 
operator $\varphi \rightarrow \frac{\partial \varphi}{\partial y_1}$ is bounded from $\widetilde{H}^{1/2}(\D_1)$ to $\widetilde{H}^{-1/2}(\D_1)$, we conclude that $R_1$ is a bounded operator from $\widetilde{H}^{1/2}(\D_1)$ into $H^{-1/2}(\D_1)$, as needed.
\end{proof}

\begin{proof}[Proof of \cref{lem.hypersinginvert}]
As in the proof of \cref{lem.S1} we introduce a smooth bounded domain $D \subset \R^d$, 
whose boundary contains the set $\D_1$, and a bounded Lipschitz domain $V$ with $D \Subset  V$.
We first prove that $R_1$ is a Fredholm operator with index $0$.
To achieve this, let $\varphi$ be an arbitrary element in $\widetilde{H}^{1/2}(\D_1)$ (extended by $0$ to all of $\partial D$) and set $u = {\mathcal D}_D \varphi \in H^1_{\text{\rm loc}}(\R^d \setminus \partial D)$.
Using the jump relations \cref{eq.jumpDD}, and then integrating by parts on all of $\mathbb{R}^d$ (which is possible because of the decay properties \cref{eq.decayDD}) we obtain
\begin{equation}\label{eq.interm1}\begin{array}{>{\displaystyle}cc>{\displaystyle}l}
 \langle R_1 \varphi,\varphi \rangle &=& \int_{\partial D}{\frac{\partial u}{\partial n} (u^+ - u^-) \: \d s}\\
 &=& - \int_{\R^d\setminus \partial D}{|\nabla u|^2 \: \d x}~.
 \end{array}
\end{equation}
 Since 
 $$ || \varphi ||_{\widetilde{H}^{1/2}(\D_1)} =  ||(u^+ - u^-) ||_{H^{1/2}(\partial D)} \leq  C \Big(\Vert \nabla u \Vert_{L^2(\R^d\setminus \partial D)} + \Vert u \Vert_{L^2(V)}\Big)~,  $$
 it follows from \cref{eq.interm1} that 
 $$
 || \varphi ||_{\widetilde{H}^{1/2}(\D_1)} \leq C \left(\Vert R_1\varphi \Vert_{H^{-1/2}(\D_1)}+ \Vert {\mathcal D}_D \varphi \Vert_{L^2(V)}\right)~.
 $$
 It now follows as in the proof of \cref{lem.S1} that $R_1$ is Fredholm with index $0$. Hence, we are left to show that $R_1$ is injective. But if $R_1 \varphi = 0$ for some $\varphi \in \widetilde{H}^{1/2}(\D_1)$, the previous calculation with $u = {\mathcal D}_D\varphi$ yields
$$ \langle R_1 \varphi , \varphi \rangle =- \int_{\R^d\setminus \partial D}{|\nabla u |^2 \: \d x} = 0,$$
so that $u$ is constant on $D$ and on $\R^d \setminus \overline{D}$. Since $u \to 0$ as $|x | \to \infty$, the value of this constant on $\R^d \setminus \overline D$ must be $0$. 
Since $\varphi = u^+ - u^-$ vanishes on $\partial D \setminus{\D}_1$, the value of this constant inside $D$ is also $0$; hence, $u = 0$ and $\phi = u^+ - u^- = 0$, which completes the proof. 
\end{proof}

%%%%%%%%%%%%%%%%%%%%%%%%%%%%%%%%%%%%%%%%%%%%%%%%%%%%%%%
\section{Conclusion and future Directions}\label{sec.persp}
%%%%%%%%%%%%%%%%%%%%%%%%%%%%%%%%%%%%%%%%%%%%%%%%%%%%%%%

\noindent 
In this article, we have analyzed the asymptotic behavior of the solution to an elliptic partial differential equation posed on a domain $\Omega \subset \R^d$
when the accompanying boundary conditions change type on a vanishing subset $\omega_\e$ of the boundary $\partial \Omega$.
More precisely, in the model context of the conductivity equation complemented with mixed homogeneous Dirichlet and Neumann boundary conditions on the respective regions $\Gamma_D, \Gamma_N \subset \partial \Omega$, 
we have derived a general representation formula
for the asymptotic structure of the potential $u_\e$ when the homogeneous Neumann boundary condition is replaced with a homogeneous Dirichlet boundary condition on an arbitrary ``small'' subset $\omega_\e \subset \Gamma_N$ (and vice-versa, when the homogeneous Dirichlet condition is replaced with a homogeneous Neumann condition on $\omega_\e \subset \Gamma_D$). 
Furthermore, in the particular situation where $\omega_\e$ is a vanishing surfacic ball, we have given precise, explicit asymptotic formulas for $u_\e$. The present findings suggest various directions for further investigations. 
\begin{itemize}
\item A natural extension of the present work is to investigate the case where the homogeneous Dirichlet boundary condition on $\Gamma_D$, 
or the homogeneous Neumann boundary condition on $\Gamma_N$, 
is replaced by yet another type of boundary condition on $\omega_\e$, for instance an \textit{inhomogeneous} Neumann boundary condition, or an inhomogeneous Dirichlet boundary condition. A perhaps even more interesting setting involves a Robin boundary condition, and thus consists in investigating the asymptotic behavior of the solution to the problem
$$
\left\{ 
\begin{array}{cl}
-\dv(\gamma \nabla u_\e) = f & \text{in } \Omega~, \\
u_\e = 0 & \text{on } \Gamma_D~, \\
\gamma\frac{\partial u_\e}{\partial n} = 0 & \text{on } \Gamma_N \setminus \overline{\omega_\e}~,\\
\gamma\frac{\partial u_\e}{\partial n} + k u_\e = 0 & \text{on } \omega_\e~,
\end{array}
\right.
$$
or the solution to the problem
$$
\left\{ 
\begin{array}{cl}
-\dv(\gamma \nabla u_\e) = f & \text{in } \Omega~, \\
u_\e = 0 & \text{on } \Gamma_D \setminus \overline{\omega_\e}~, \\
\gamma\frac{\partial u_\e}{\partial n} = 0 & \text{on } \Gamma_N~,\\
\gamma\frac{\partial u_\e}{\partial n} + k u_\e = 0 & \text{on } \omega_\e~.
\end{array}
\right.
$$ 
The understanding of  this limiting process, uniformly with respect to the parameter $k$ of the Robin condition, would provide a key insight into the nature of the transition between the Dirichlet and Neumann behaviors (established in this paper). In this spirit, see for instance \cite{charnley2019uniformly1,dapogny2017uniform,nguyen2009representation} concerning 
small volume asymptotic formulas, which are uniform with respect to the properties of the material occupying the vanishing inclusions. 

\item Beyond the realm of the conductivity equation, the present study could be extended to other, more challenging physical contexts, e.g., the system of linear elasticity -- where homogeneous Dirichlet boundary conditions account for ``clamping" and homogeneous Neumann boundary conditions represent absence of traction.

\item Last but not least, it would be interesting to explore the practical applications of these results. 
As we have illustrated at in \cref{rem.cply}, asymptotic formulas for the solution to ``small'' perturbations of a ``background'' boundary value problem 
allow to appraise the sensitivity of a quantity of interest (or a performance criterion) with respect to such perturbations. This idea plays into the concepts of topological derivative \cite{novotny2012topological} and ``topological ligaments'' \cite{nazarov2004topological,dapogny2020topolig} in optimal design. 
In our context it would allow us to appraise the sensitivity of a performance criterion with respect to the introduction of a new, ``small'' region supporting Dirichlet or Neumann boundary condition in the physical boundary value problem. Such a program appears 
especially interesting in the context of linear elasticity, where it  would significantly complement the study in \cite{dapogny2019optimization}; see also \cite{lalainamaster}.
\end{itemize}\par\medskip

\noindent \textbf{Acknowledgements.} The authors are grateful to Jean-Claude N\'ed\'elec, who pointed several useful references
concerning Sections 5 and 6 of this work. EB is partially supported by the ANR Multi-Onde. The work of CD is partially supported by the project ANR-18-CE40-0013 SHAPO, financed by the French Agence Nationale de la Recherche (ANR).
This work was carried out while MSV was on sabbatical at the University of Copenhagen and the Danish Technical University. 
This visit was in part supported by the Nordea Foundation and the Otto Mo\hskip -4.6pt $/$nsted Foundation. The work of MSV was also partially supported by NSF grant DMS-12-11330.

%%%%%%%%%%%%%%%%%%%%%%%%%%%%%%%%%%%%%%%%%%%%%%%%%%%%%%%
\appendix
\section{A closer look to the quantity $e(\omega)$}\label{app.eom}
%%%%%%%%%%%%%%%%%%%%%%%%%%%%%%%%%%%%%%%%%%%%%%%%%%%%%%%

\noindent The purpose of this appendix is to analyze more in depth the quantity $e(\omega)$, defined in \cref{eq.capaneu}, and  
used in \cref{sec.neu,sec.calcneu} to assess the ``smallness'' of a subset $\omega$ of $\partial \Omega$, when homogeneous Dirichlet boundary conditions are replaced by homogeneous Neumann conditions. More precisely, we construct explicit quantities which bound $e(\omega)$, and which are not excessively conservative --  quantities that  do not require the solution of any boundary value problems.

%%%%%%%%%%%%%
\subsection{Some differential geometry facts}\label{sec.geom}
%%%%%%%%%%%%%

\noindent We shall need some basic facts from differential geometry on hypersurfaces in $\R^d$. All of these results are well-known, however,
some are not so easily found in the literature, and for the convenience of the reader we include their proofs in this section. We refer to classical books, such as \cite{do1992riemannian,lang2012fundamentals}, for further details.  \par\medskip

Let $\Omega \subset \R^d$ be a smooth bounded domain. We first recall some terminology: 
\begin{itemize}
\item The tangent plane to $\partial \Omega$ at a point $x \in \partial\Omega$ is the hyperplane of $\R^d$ which is orthogonal to the unit normal vector $n(x)$. 
The orthogonal projection $P_x v \in T_x \partial \Omega$ of a vector $v \in \R^d$ onto $T_x \partial \Omega$ is given by
$$P_x v := v-(v \cdot n(x))n(x)~.$$
\item The length of a piecewise differentiable curve $\gamma: [a,b] \to \partial \Omega$ is defined by 
$$ \ell(\gamma) := \int_a^b{|\gamma^\prime (t)| \: \d t}~,$$
where the derivative of $t\mapsto \gamma(t)$ is calculated as that of an $\R^d$-valued function.
This quantity is obviously independent of the particular parametrization chosen for $\gamma$. 
\item A differentiable curve $\gamma: [a,b] \to \partial \Omega$ is called a (constant speed) geodesic segment joining the endpoints $\gamma(a)$ and $\gamma(b)$ if it satisfies: 
$$\forall t \in (a,b), \:\:  P_{\gamma(t)}(\gamma^{\prime\prime}(t))  = 0~.$$
%here, the derivative is the standard Jacobian matrix for a vector field. Note that this vector field is only defined along a curve and not on an open set, 
%but since we are taking the derivative in the direction of the tangent vector to this curve, this does not matter. 
%the expression at the left- hand side being that of the Levi-Civita connection on a hypersurface of $\R^d$ equipped with the Euclidean metric. 
%Geodesics can be equivalently characterized as minimizers to en energy functional. 
\item A geodesic segment $\gamma:[a,b] \to \partial \Omega$ is called minimizing if $\ell(\gamma) \leq \ell(\widetilde \gamma)$ for any piecewise differentiable curve $\widetilde\gamma(t)$
 joining $\gamma(a)$ to $\gamma(b)$.
 \item The geodesic distance between two points $x, y \in \partial \Omega$ is defined by:
$$ d^{\partial \Omega}(x,y) = \inf\limits_{\gamma:[a,b] \to \partial \Omega \atop \gamma(a) = x, \: \gamma(b) = y}{\ell(\gamma)}~.$$
\item Likewise, the geodesic distance $\dist^{\partial \Omega}(x,K)$ from a point $x \in \partial \Omega$ to a closed subset $K \subset \partial \Omega$ is: 
$$ \dist^{\partial \Omega}(x,K)=  \inf\limits_{y\in K} d^{\partial \Omega}(x,y)~.$$
\end{itemize}

The distance between two points $x, y \in\partial \Omega$ can be measured either in terms of the (extrinsic) Euclidean distance $|x-y|$ of $\R^d$ 
or by means of the (intrinsic) geodesic distance $d^{\partial\Omega}(x,y)$. 
It turns out that these notions are equivalent in the present context, as stated in the next lemma.

\begin{lemma}\label{lem.eqdist}
There exists a constant $c>0$ which only depends on $\Omega$ such that the following inequalities hold: 
$$\forall x, y \in \partial \Omega~, \:\: c \: d^{\partial \Omega}(x,y) \leq |x-y | \leq d^{\partial \Omega}(x,y)~. $$
\end{lemma}
\begin{proof}
The right inequality is obvious, and we focus on the proof of the left one. 
To this end, we introduce a finite open covering $\left\{ U_i \right\}_{i=1,\ldots,N}$ of the smooth, compact hypersurface $\partial \Omega$ with the following property: for each $i=1,\ldots,N$, there exist a convex open subset $V_i \subset \mathbb{R}^{d-1}$ and a function $f_i : V_i \to \R$ which is smooth on an open neighborhood of $\overline{V_i}$,
such that (up to a relabeling of coordinates in $\R^d$) the mapping
$$ \sigma_i : V_i \ni (x_1,\ldots,x_{d-1}) \mapsto (x_1,\ldots,x_{d-1},f_i(x_1,\ldots,x_{d-1})) \in U_i$$ 
realizes a diffeomorphism from $V_i$ onto $U_i$. 
We also denote by $\delta >0$ a Lebesgue number associated to this covering, that is
$$ \forall \omega \subset \partial \Omega, \:\: \text{\rm diam}(\omega) \leq \delta \:\: \Rightarrow \:\: \omega \subset U_i \text{ for some } i=1,\ldots,N~,$$
where the diameter $\text{diam}(\omega) := \sup_{x,y\in \omega}{|x-y|}$ is understood in the sense of the Euclidean distance. 

Now considering two given points $x, y \in \partial \Omega$, we distinguish two cases. \par\medskip

\noindent \textit{Case 1: $|x-y| > \delta$.} By introducing the quantity $M := \sup\limits_{p,q \in \partial \Omega} d^{\partial \Omega}(p,q)$, we obtain
$$ d^{\partial \Omega}(x,y) \leq M \leq \frac{M}{\delta} |x-y|~.$$
\par\medskip
\noindent \textit{Case 2: $|x-y| \leq \delta$}. Then $x$ and $y$ belong to a common open subset $U_i$, and we let $\widehat{x}, \widehat{y}$ be the points in $V_i$ such that $x = \sigma_i (\widehat{x})$ and $y = \sigma_i(\widehat{y})$. We also introduce the differentiable curve $\gamma(t) = \sigma_i(\widehat{x} + t(\widehat{y}-\widehat{x}))$ connecting $x$ to $y$. It follows from the very definition of the geodesic distance $d^{\partial \Omega}(x,y)$ that
$$ \begin{array}{>{\displaystyle}cc>{\displaystyle}l}
d^{\partial \Omega}(x,y) &\leq& \int_0^1{\sqrt{\gamma^\prime(t) \cdot \gamma^\prime(t)} \: \d t} \\
&=& \int_{0}^1{\sqrt{\nabla\sigma_i(\widehat{x} + t(\widehat{y}-\widehat{x}))^t \nabla\sigma_i (\widehat{x} + t(\widehat{y}-\widehat{x})) (\widehat{y}-\widehat{x}) \cdot (\widehat{y}-\widehat{x})}\: \d t}~.
\end{array}$$
For $i=1,\ldots,N$ and $\widehat z \in V_i$, we introduce the eigenvalues $1 \leq \lambda^i_1(\widehat z) \leq \ldots \leq \lambda^i_{d-1}(\widehat z)$ of the $(d-1)\times (d-1)$ matrix $\nabla \sigma_i(\widehat z)^t\nabla \sigma_i(\widehat z)$, and
$$ M:=\max\limits_{i=1,\ldots,N} \sup_{\widehat z \in V_i}\lambda_{d-1}^i(\widehat z) < \infty~. $$
The bound $M$ depends only on the properties of the hypersurface $\partial \Omega$.
We now have
$$d^{\partial \Omega}(x,y) \leq \sqrt{M} |\widehat{x} - \widehat{y} | \leq \sqrt{M}|x-y|~,$$
as desired.
\end{proof}

\noindent
We recall the definition and the main properties of the \textit{exponential map}, $\exp_x$, at a point $x \in \partial \Omega$:
\begin{itemize}
\item The mapping $\exp_x: U \to \partial \Omega$ is defined on an open neighborhood $U$ of $0$ in the tangent plane $T_x \partial \Omega$ by the formula:
$$\forall v \in U, \:\:  \exp_x(v) = \gamma(1,x,v),$$
where $t\mapsto \gamma(t,x,v)$ is the unique geodesic curve on $\partial \Omega$ passing through $x$ at $t=0$ with velocity $v$:
\begin{equation}\label{eq.geodx}
 \gamma(0,x,v) = x, \text{ and } \gamma^\prime(0,x,v) = v.
 \end{equation}
\item For any point $x \in \partial \Omega$, there exists a number $\rinj(x) >0$ -- the injectivity radius of $x$ -- such that $\exp_x$ is a diffeomorphism from the $(d-1)$ dimensional ball $B(0,\rinj(x)) \subset T_x\partial \Omega$ onto the geodesic ball 
$$B^{\partial \Omega}(x,\rinj(p)) :=\{ y \in \partial \Omega,\: d^{\partial \Omega}(x,y)< \rinj(x)\}$$ 
on $\partial \Omega$. 
In particular, $\exp_x : B(0,\rinj(x)) \subset T_x\partial \Omega \to \partial \Omega$ is a local chart for $\partial \Omega$ around $x$. 
\item At an arbitrary point $x \in \partial \Omega$, the following identity holds: 
\begin{equation}\label{eq.distexp}
d^{\partial \Omega}(x,\exp_x(v)) = |v|, \quad v \in B(0,\rinj(x)) \subset T_x\partial\Omega.
\end{equation}
\item Since $\partial \Omega$ is smooth and compact, there exists a number $\rinj>0$ -- the injectivity radius of $\partial \Omega$ -- such that for all $x\in \partial \Omega$, 
$\rinj < \rinj(x)$.\par\medskip
\end{itemize}

Let us finally state a useful consequence of the change of variables formula, applied to the exponential mapping. 

\begin{lemma}\label{lem.chgvarexp}
Let $f \in L^1_{\text{\rm loc}}(\partial \Omega)$; 
then for any point $x \in \partial \Omega$ and $r < \rinj(x)$, 
$$ \int_{B^{\partial \Omega}(x,r)}{f(y) \: \d s(y)} = \int_{B(0,r)}{f(\exp_x(v))g(v) \: \d v}~,$$
where $B(0,r)$ is the ball with center $0$ and radius $r$ in $T_x \partial \Omega$, and $g$ is given by 
$$g(v) := \sqrt{\det(M_{ij}(v))}, \:\: M_{ij}(v) := \Big(\d \exp_x(v)(e_i)\Big) \cdot \Big(\d \exp_x(v) (e_j)\Big), \:\: i,j=1\ldots,d-1~,$$
is bounded uniformly from below and above by positive constants which depend only on the properties of $\partial \Omega$. 
The tangent vectors $\d \exp_x(v)(e_i) \in T_{\exp_x(v)}\partial \Omega$ featured in these equations are given by
$$\d \exp_x(v)(e_i)= \left. \frac{\d}{\d s} \gamma(1,x,v+s e_i) \right\lvert_{s=0}~,$$
where $t\mapsto \gamma(t,x,v)$ is the unique geodesic passing through $x$ at $t=0$ with velocity $v$, see \cref{eq.geodx}.
\end{lemma}

%%%%%%%%%%%%%%%%%%
\subsection{Derivation of ``geometric'' upper bounds for the quantity $e(\omega)$}
%%%%%%%%%%%%%%%%%%

\noindent Throughout this section $\omega$ is an open Lipschitz subset of $\partial \Omega$. $\omega$ lies strictly inside $\Gamma_D$, and the setting is as in \cref{sec.neu}. We start with the following result.

\begin{lemma}
\label{lem.explupbd}
Let $\omega$ be an open Lipschitz subset of $\Gamma_D \subset \partial \Omega$, which is well-separated from $\Gamma_N$, i.e. \cref{assum.far} holds. 
There exists a constant $C >0$, depending only on $\Omega$, $\Gamma_D$ and $\dmin$ such that
$$
e(\omega) \le C \int_{\omega} \frac{1}{\rho_{\omega}(x)}~\d s(x)
$$
where $\rho_{\omega}(x)$ denotes the weight function defined by
$$
\forall x \in \omega , \quad \rho_{\omega}(x) := \int_{\partial \Omega \setminus \overline{ \omega}}{\frac{1}{|x-y|^d}\:\d s(y)}~.
$$
\end{lemma}

\begin{proof}
Let us introduce the solution $\zeta \in H^1(\Omega)$ to \cref{eq.zepsneum}; 
it follows from a simple adaptation of \cref{lem.zetaecapaneu} and integration by parts that
\begin{equation}
\label{eq.firstestome}
e(\omega) \leq C \int_\Omega  |\nabla \zeta|^2 \d x = C\int_{\omega} \zeta~\d s~,
\end{equation}
where the constant $C$ depends only on $\Omega$, $\Gamma_D$ and $\dmin$.

A slight generalization of the argument leading to  the estimate \cref{normcalc} in \cref{sec.Hs}, 
using that $\zeta$ vanishes on $\Gamma_D \setminus \overline \omega$, gives that for some constant $C$, depending on $\Omega$, $\Gamma_D$ and $\dmin$
\begin{equation}
\label{eq.secestome}
\left(\int_{\omega} |\zeta(x)|^2  \rho_{\omega}(x)~\d s(x)\right)^{1/2} \le C \Vert \zeta \Vert_{H^{1/2}(\partial \Omega)}~. 
\end{equation}
A combination of \cref{eq.firstestome} and \cref{eq.secestome} now yields
\begin{eqnarray*}
e(\omega)& \le& C \int_{\omega} \zeta(x)~\d s(x) \\
&\le& C \left( \int_{\omega} \rho_{\omega}(x)^{-1}~\d s(x) \right)^{1/2}\left(\int_{\omega} |\zeta(x)|^2  \rho_{\omega}(x)~\d s(x)\right)^{1/2} \\
&\le& C \left( \int_{\omega} \rho_{\omega}(x)^{-1}~\d s(x) \right)^{1/2} \Vert \zeta \Vert_{H^{1/2}(\partial \Omega)} \\
&\le& C\left( \int_{\omega} \rho_{\omega}(x)^{-1}~\d s(x) \right)^{1/2} \Vert \zeta \Vert_{H^{1}( \Omega)}~.
\end{eqnarray*}
Adapting the proof of \cref{lem.zetaecapaneu}, we may prove
$$
\Vert \zeta \Vert_{H^1(\Omega)} \le  C e (\omega)^{1/2}~,
$$
and after insertion of this into the last line of the previous estimate (and cancellation), we obtain
$$
e(\omega)^{1/2} \le  C \left( \int_{\omega} \rho_{\omega}(x)^{-1}~\d s(x) \right)^{1/2}~,
$$
which is the desired conclusion.
\end{proof}

Let us introduce the notation
$$
D(\omega) = \int_{\omega} \frac{1}{\rho_{\omega}(x)}~\d s(x)~.
$$
It follows from \cref{lem.explupbd} that $D(\omega)$ is an upper bound for $e(\omega)$ (up to constants involving the chosen domain $\Omega$ and the regions $\Gamma_D$, $\Gamma_N$ of its boundary), 
which has the appealing feature that it depends solely on the geometry of $\omega$. 
We believe there are many interesting examples where $D(\omega)$ is equivalent to $e(\omega)$; actually we have provided such an example in \cref{sec.calcneu}. For this reason we also find it useful to derive an equivalent, but simpler, expression for the measure $D(\omega)$. The remainder of this subsection is devoted to this task, and we start with a lemma.

\begin{lemma}
\label{lem.lowerDbound}
Let $\omega$ be an open Lipschitz subset of $\partial \Omega$. There exists a constant $c>0$, which depends only on $\partial \Omega$, such that
$$ c \int_\omega{\dist^{\partial \Omega}(x,\partial \omega) \: \d s(x)}\le D(\omega)~.$$
\end{lemma}
\begin{proof}
The lemma follows immediately, by integration over $\omega$, if we prove that, for all points $x \in \omega$
\begin{equation}\label{eq.targubD}
c \, \dist^{\partial \Omega}(x,\partial \omega)\le \rho_\omega(x)^{-1}~.
\end{equation}
To achieve this goal, we distinguish between two cases, depending on the size of $\dist^{\partial \Omega}(x,\partial \omega)$ relative to the injectivity radius $\rinj$ of $\partial\Omega$.\par\medskip

\noindent \textit{Case 1: $\dist^{\partial\Omega}(x,\partial \omega)\geq \rinj$.} From the definition of $\rho_\omega(x)$ and \cref{lem.eqdist}, we have 
 $$
 \begin{array}{>{\displaystyle}cc>{\displaystyle}l}
 \rho_\omega(x) &\leq & C \int_{\partial \Omega \setminus \omega}{\frac{\d s(y)}{d^{\partial \Omega}(x,y)^d}} \\
 &\leq & \frac{C}{\dist^{\partial \Omega}(x,\partial \omega)^d} \int_{\partial \Omega}{\d s(y)} \\[0.5em]
 &\leq&  \frac{C}{\rinj^{d-1}} \frac{1}{\dist^{\partial \Omega}(x,\partial \omega)} \\
 &=&\frac{C}{\dist^{\partial \Omega}(x,\partial \omega)}~,
 \end{array}
 $$
where the constant $C$ is changing from one instance to the next, but depends only on $\partial \Omega$, and not on $\omega$.  
Hence, \cref{eq.targubD} holds in this case.\par\medskip

\noindent \textit{Case 2: $ \dist^{\partial \Omega}(x,\partial \omega)< \rinj$.} The exponential mapping $\exp_x$ induces a diffeomorphism from the ball $B(0,\dist^{\partial \Omega}(x,\partial \omega)) \subset T_x\partial \Omega$ onto the geodesic ball $B^{\partial \Omega}(x,\dist^{\partial \Omega}(x,\partial \omega))$. Since $B^{\partial \Omega}(x,\dist^{\partial \Omega}(x,\partial \omega))$ lies inside $\omega$, it follows that
\begin{equation}\label{eq.majrho}
 \begin{array}{>{\displaystyle}cc>{\displaystyle}l}
\rho_\omega(x) &\leq& C \int_{\partial \Omega\setminus B^{\partial \Omega}(x,\dist^{\partial \Omega}(x,\partial \omega))} \frac{\d s(y)}{d^{\partial \Omega}(x,y)^{d}} \\
& = & C \left( \int_{\partial \Omega\setminus B^{\partial \Omega}(x,\rinj)} \frac{\d s(y)}{d^{\partial \Omega}(x,y)^{d}} +  \int_{ B^{\partial \Omega}(x,\rinj) \setminus B^{\partial\Omega}(x,\dist^{\partial \Omega}(x,\partial \omega))} \frac{\d s(y)}{d^{\partial \Omega}(x,y)^{d}} \right)~.
\end{array}
\end{equation}
As in Case 1, the first integral in the above right-hand side is easily estimated by
\begin{equation}\label{eq.majrho1int}
 \int_{\partial \Omega\setminus B^{\partial \Omega}(x,\rinj)} \frac{\d s(y)}{d^{\partial \Omega}(x,y)^{d}}  \leq \frac{C}{\rinj^d} \leq \frac{C}{\rinj^{d-1}} \frac{1}{\dist^{\partial \Omega}(x,\partial \omega)}=  \frac{C}{\dist^{\partial \Omega}(x,\partial \omega)}~.
\end{equation}
As for the second integral, the exponential change of variables of \cref{lem.chgvarexp}, followed by a change to polar coordinates yields 
\begin{equation}\label{eq.majrho2int}
 \begin{array}{>{\displaystyle}cc>{\displaystyle}l}
 \int_{ B^{\partial \Omega}(x,\rinj) \setminus B^{\partial\Omega}(x,\dist^{\partial \Omega}(x,\partial \omega))} \frac{\d s(y)}{d^{\partial \Omega}(x,y)^{d}}  &\leq& C \int_{B(0,\rinj) \setminus B(0,\dist^{\partial\Omega}(x,\partial \omega))} {\frac{\d u}{d^{\partial\Omega}(x,\exp_x(u))^d}}\\
 &=& C \int_{B(0,\rinj) \setminus B(0,\dist^{\partial \Omega}(x,\partial \omega))} {\frac{\d u}{|u|^d}}\\[1em]
&\leq& C \int_{\dist^{\partial \Omega}(x,\partial \omega))}^{\rinj}{\frac{t^{d-2}}{t^d} \: \d t} \\
&=& C \left(\frac{1}{\dist^{\partial \Omega}(x,\partial \omega))} - \frac{1}{\rinj}\right).
 \end{array}
\end{equation}
A combination of \cref{eq.majrho,eq.majrho1int,eq.majrho2int} leads to
$$
\rho_\omega(x) \le \frac{C}{\dist^{\partial \Omega}(x,\partial \omega)}~,
$$
which is exactly \cref{eq.targubD}, thus completing the proof of the lemma.
\end{proof}
%
%\begin{proof}
%Given $x\in \omega$, let $\delta = \text{dist}(x,\partial \omega)$. The ``disk" $B_\delta(x) \cap \partial \Omega$ lies inside $\omega$, and so
%\begin{eqnarray*}
%\rho_\omega(x) &=& \int_{\partial \Omega\setminus \omega} \frac{1}{|x-y|^{d}}~dy \le \int_{\partial \Omega\setminus (B_\delta(x) \cap \partial \Omega)} \frac{1}{|x-y|^{d}}~dy \\
%&=&\int_{\partial \Omega\setminus B_\delta(x)} \frac{1}{|x-y|^{d}}~dy \le C_1 \int_{\delta}^\infty r^{-d} r^{d-2}~dr \\ 
%&=& C_1 \delta^{-1} = C_1 \text{dist}(x,\partial \omega)^{-1}~,
%\end{eqnarray*}
%where $C_1$ is independent of $x$ and $\omega$, and only depends on $\partial \Omega$. It follows that 
%$$
%c_1 \text{dist}(x,\partial \omega) \le \rho_\omega(x)^{-1}~,
%$$
%and thus
%\begin{equation}
%\label{lowerbound}
%c_1 \int_{\omega} \text{dist}(x,\partial \omega)\, ds_x \le \int_{\omega} \rho_\omega(x)^{-1}\, ds_x = D(\omega)~,
%\end{equation}
%with a positive constant $c_1$ that only depends on $\partial \Omega$.
%
%\end{proof}

The reverse inequality is more subtle, and it holds only under additional assumptions on the set $\omega \subset \partial \Omega$. Let us introduce a few related definitions.
% "convexe" seulement suffirait sans doute...  
\begin{definition}
Let $\omega \subset \partial \Omega$ be an open Lipschitz subset.
% On a forcement une param constante, et l'intervalle $[0,1]$ fixe la vitesse. 
The set $\omega$ is called geodesically convex if for any two points $p,q \in \omega$, there exists a unique
 minimizing geodesic segment $\gamma: [0,1] \to \partial \Omega$ joining $p$ to $q$, with $\gamma([0,1]) \subset \omega$. 
\end{definition}

\begin{definition}
Let $\omega \subset \partial \Omega$ be a geodesically convex, open Lipschitz subset.
\begin{itemize}
\item For any $p \in \partial \omega$, the tangent cone $C_p \subset T_p \partial \Omega$ to $\omega$ at $p$ is defined by
$$ C_p := \left\{ v \in T_p\partial\Omega, \:\: \exp_p\left(t \frac{v}{|v|} \right) \in \omega \text{ for some } 0 < t < \rinj(p) \right\} \cup \left\{ 0 \right\}~.$$
\item For any $p \in \partial \omega$, an open half-space $H \subset T_p\partial \Omega$ is called a supporting half-space for $\omega$ at $p$ if $C_p \subset \overline H$.
\end{itemize}
\end{definition}

%
%Let us recall that there exists a number $r >0$, the convexity radius of $\partial \Omega$, such that every geodesic ball with radius less than $p$ is strongly convex, 
%and normal, in the sense that the exponential mapping $\exp_p : B(0,r) \subset T_p\partial \Omega \to B^{\partial \Omega}(p,r)$ is a smooth diffeomorphism. 
% 

The following result generalizes well-known properties of convex subsets of the Euclidean space $\R^d$,
in terms of supporting hyperplanes, to the setting of geodesically convex subsets of $\partial \Omega$. It is  a summary of the contents of Proposition 1.8 and Lemma 1.7 in \cite{cheeger1972structure}; 
see \cref{fig.cvx} for an illustration.

\begin{proposition}\label{th.cheeger}
Let $\omega \subset \partial \Omega$ be a geodesically convex, open Lipschitz subset of $\partial \Omega$, and let $p \in \partial \omega$. 
Then, the tangent cone $C_p \subset T_p \partial \Omega$ to $\omega$ at $p$ satisfies
$$ C_p \setminus \left\{ 0 \right\} = \bigcap{H_j}~,$$
where the intersection is taken over all the supporting half-spaces of $\omega$ at $p$.

In addition, if there exists $q \in \omega$ and a minimal geodesic segment $\gamma:[0,1] \to \partial\Omega$ from $q$ to $p$ such that $\ell(\gamma) = \dist^{\partial \Omega}(q,\partial \omega)$, 
then $C_p \setminus \left\{ 0 \right\}$ is exactly the open half-space 
\begin{equation}\label{eq.supporths}
 H = \left\{ v \in T_p\partial \Omega,\:\: v \cdot (-\gamma^\prime(1)) > 0 \right\}.
 \end{equation} 
\end{proposition}

\begin{figure}[!ht]
\centering
\begin{minipage}{0.95\textwidth}
\includegraphics[width=1.0\textwidth]{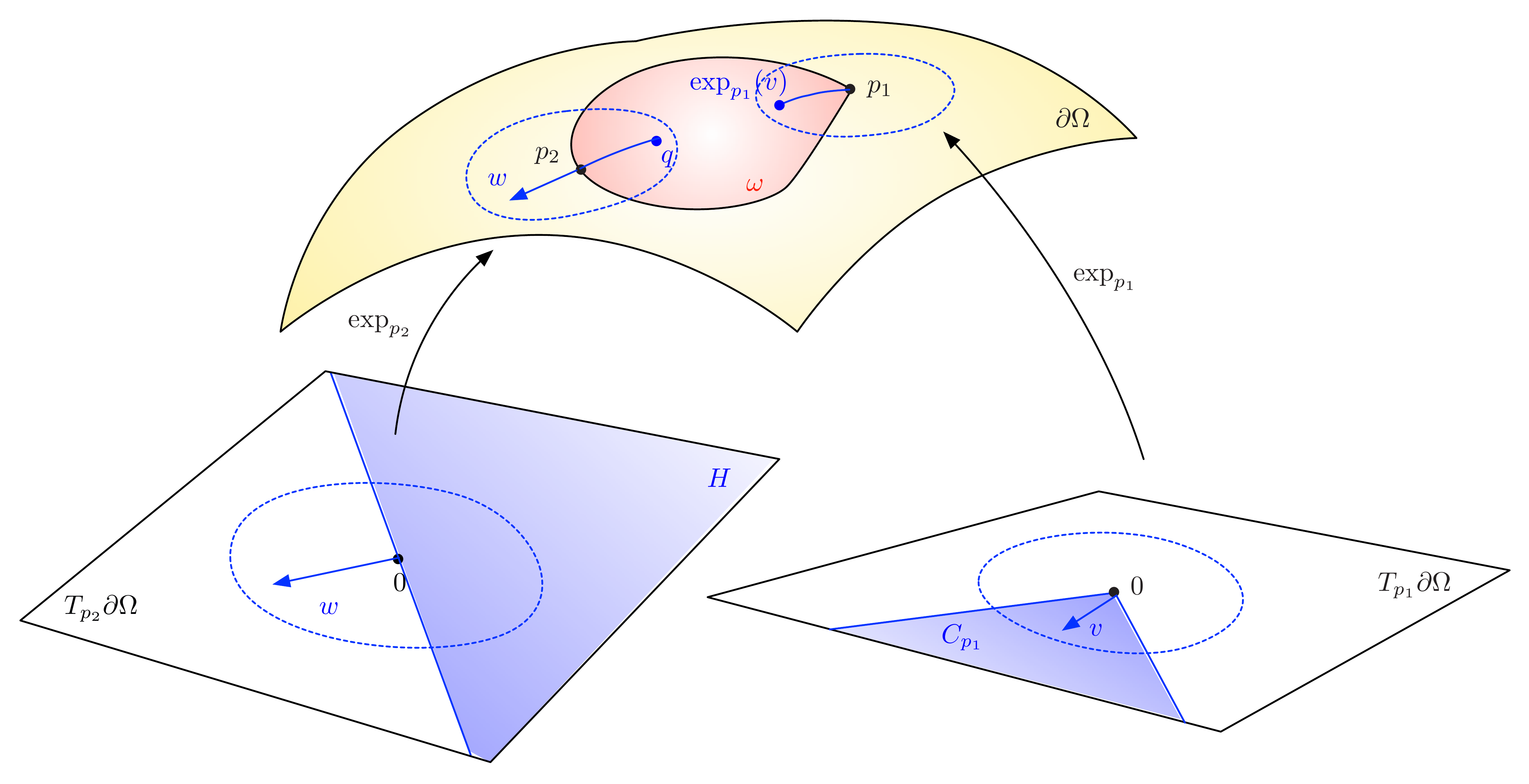}
\end{minipage} 
\caption{\it Illustration of the various objects attached to a geodesically convex open subset $\omega \subset \partial \Omega$; (on the right) the tangent cone $C_{p_1}$ to the point $p_1$; (on the left) the point $q$ is such that $d^{\partial \Omega}(p_2,q)= d^{\partial \Omega}(q,\partial \omega)$, and so the half-space $H \subset T_{p_2}\partial \Omega$ which is normal to the velocity vector $w = \gamma^\prime(1)$ of the associated minimizing geodesic $\gamma$ is exactly the tangent cone to $\omega$ at $p_2$.}
\label{fig.cvx}
\end{figure}

We are now in a position to establish an equivalence between $D(\omega)$ and a much simpler integral over $\omega$.

\begin{lemma}\label{lem.ineqcvx}
Let $\omega$ be a geodesically convex, open Lipschitz subset of $\partial \Omega$. 
Then
$$ c \int_\omega{\dist^{\partial \Omega}(x,\partial \omega) \: \d s(x)}\leq D(\omega) \leq C \int_\omega{\dist^{\partial \Omega}(x,\partial \omega) \: \d s(x)}~,$$
where the positive constants $c$ and $C$ depend only on $\partial \Omega$. 
\end{lemma}
\begin{proof}
The lower bound was already established in \cref{lem.lowerDbound}, so it only remains to prove the upper bound.
Let $x \in \omega$ be given, and let $p\in \partial \omega$ be a point minimizing the geodesic distance  from $x$ to $\partial \omega$:
$$ \delta := \dist^{\partial \Omega}(x,\partial \omega) = d^{\partial \Omega}(x,p)~.$$
Let ${\mathcal H}$ denote the set
$$
{\mathcal H} := \left\{ \exp_p\left(t\frac{v}{\vert v\vert}\right)~:~ v\in H \hbox{ and } 0<t<\rinj ~ \right\}~,
$$
where $H$ is the supporting half-space to $\omega$ at $p$ characterized by \cref{eq.supporths}.
Due to \cref{th.cheeger}, we also obtain the following estimate
\begin{equation}\label{eq.majhs}
 \rho_\omega(x) = \int_{\partial \Omega \setminus \omega}{\frac{\d s(y)}{|x-y|^d}} \geq \int_{B^{\partial \Omega}(p,\rinj) \setminus {\mathcal H}} {\frac{\d s(y)}{|x-y|^d}}~.
 \end{equation}
Since the geodesic distance $d^{\partial \Omega}(x,y)$ between $x$ and any point $y \in \partial \Omega$ is always larger than the corresponding Euclidean distance $|x-y|$, we conclude that
$$ \rho_\omega(x) \geq  \int_{B^{\partial \Omega}(p,\rinj) \setminus {\mathcal H}} {\frac{\d s(y)}{d^{\partial \Omega}(x,y)^d}}~.$$
Using the change of variables of \cref{lem.chgvarexp} based on the exponential mapping $B(0,\rinj) \subset T_p\partial \Omega \to B^{\partial \Omega}(p,\rinj) \subset \partial \Omega$, we then get 
$$ \rho_\omega(x) \geq C \int_{\left\{y \in B(0,\rinj), \: y_1 > 0 \right\}}{\frac{\d y}{d^{\partial\Omega}(x,\exp_p(y))^d}}~, $$
where $y = (y_1,\ldots,y_{d-1})$ are the coordinates of the integration variable $y$ in an orthonormal frame of $T_p\partial\Omega$, with the first coordinate vector being the outer normal to $H$. The triangle inequality for the geodesic distance, and a change of variables yields
$$   \begin{array}{>{\displaystyle}cc>{\displaystyle}l}
\rho_\omega(x) &\geq& C \int_{\left\{y \in B(0,\rinj), \:y_1 > 0 \right\}}{\frac{\d y}{(\delta + d^{\partial \Omega}(p,\exp_p(y))^d}}\\[1em]
&=& \frac{C}{\delta^d} \int_{\left\{y \in B(0,\rinj), \:y_1 > 0 \right\}}{\frac{\d y}{\left(1 + \frac{1}{\delta}d^{\partial \Omega}(p,\exp_p(y)\right)^d}}\\[1em]
&=& \frac{C}{\delta} \int_{\left\{ z \in B(0,\frac{\rinj}{\delta}),\: z_1 > 0 \right\}}{\frac{\d z}{\left(1 + \frac{1}{\delta}d^{\partial \Omega}(p,\exp_p(\delta z)\right)^d}}~.
\end{array} $$
Let $\ell >0$ be the maximum length of a geodesic segment on $\partial \Omega$; obviously $\delta \leq \ell$, and so
\begin{equation} \label{eq.nexttolast}  \begin{array}{>{\displaystyle}cc>{\displaystyle}l}
\rho_\omega(x) &\geq& \frac{C}{\delta} \int_{\left\{ z \in B(0,\frac{\rinj}{\ell}), \: z_1 > 0 \right\}}{\frac{\d z}{\left(1 + \frac{1}{\delta}d^{\partial \Omega}(p,\exp_p(\delta z)\right)^d}}\\[1em]
&=&  \frac{C}{\delta} \int_{\left\{ z \in B(0,\frac{\rinj}{\ell}),\: z_1 > 0 \right\}}{\frac{\d z}{\left(1 + |z|\right)^d}}~.
\end{array}
\end{equation}
In the last line we have used that, according to \cref{eq.distexp}
$$ d^{\partial \Omega}(p,\exp_p(\delta z)) = \delta d^{\partial \Omega}(p,\exp_p(z)) = \delta |z|, \text{ as long as } \delta z \in B(0,\rinj)~,$$ 
together with the fact that $|z|< \frac{\rinj}{\ell}$ implies $\delta |z| < \frac{\delta}{\ell} \rinj \le \rinj$.
The estimate \cref{eq.nexttolast} immediately shows that there exists a constant $c >0$ which depends only on the properties of $\partial \Omega$ (and not on the set $\omega$) such that
$$\forall x \in \omega, \:\:  \rho_\omega(x) \geq \frac{c}{d^{\partial \Omega}(x,\partial \omega)}~.$$
Finally, this gives
$$ D(\omega) = \int_{\omega}{\frac{1}{\rho_\omega(x)} \d s(x)} \leq C \int_{\omega}{d^{\partial \Omega}(x,\partial \omega) \: \d s(x)}~,$$
which is the desired upper bound.
\end{proof}

\begin{remark}\label{rem.eDe}
Combining \cref{lem.explupbd} with \cref{lem.ineqcvx}, we immediately obtain that the ``capacity'' $e(\D_\e)$ of the planar disk $\D_\e$ with center $0$ and radius $\e$ defined in \cref{eq.defDecapa} satisfies the estimate
$$ e(\D_\e) \leq C_2 \e^2 \text{ if } d =2, \text{ and } e(\D_\e) \leq C_3 \e^3 \text{ if } d=3~,$$
for some universal constants $C_2$ and $C_3$.
\end{remark}

%%%%%%%%%%%%%%%%%%%%%%%%%%%%%%%%%%%%%%%%%%%%%%%%%%%%%%%
\section{The Peetre lemma}\label{app.peetre}
%%%%%%%%%%%%%%%%%%%%%%%%%%%%%%%%%%%%%%%%%%%%%%%%%%%%%%%

\noindent For the convenience of the reader, we recall Peetre's lemma, which provides a convenient sufficient condition for an operator to be Fredholm; 
see for instance \cite{peetre1961another}, \cite{lions1968problemes} (Chap. 2, \S 5.2), and also \cite{tartar1987lemme} for the precise version below, an interesting proof and useful application examples.

\begin{lemma}\label{lem.peetre}
Let $(E,|| \cdot ||_E)$ be a Banach space, $(F,|| \cdot ||_F)$ and $(G,|| \cdot ||_G)$ be normed vector spaces. 
Let $A: E \to F$ and $B: E \to G$ be bounded operators satisfying the following conditions:
\begin{enumerate}[(i)]
\item There exists a constant $C >0$ such that, 
$$\forall u \in E, \:\:  || u ||_E \leq C \Big(|| Au ||_F + || Bu ||_G \Big).$$
\item The operator $B$ is compact. 
\end{enumerate}
Then, $A$ has closed range in $F$ and finite-dimensional kernel in $E$. 
\end{lemma}

%%%%%%%%%%%%%%%%%%%%%%%%%%%%%%%%%%%%%%%%%%%%%%%%%%%%%%%
\section{Equilibrium distributions}\label{app}
%%%%%%%%%%%%%%%%%%%%%%%%%%%%%%%%%%%%%%%%%%%%%%%%%%%%%%%

\noindent In this appendix, we collect some useful results from the literature about the \textit{equilibrium distributions} associated with certain integral operators.

\begin{proposition}\label{prop.eqdistrib}
Let $\D_1 \subset \R^d$ be defined by $\D_1 := \left\{ x = (x_1,\ldots,x_{d-1}, 0) \in \R^d, \:\: |x| < 1 \right\}$. Then, 
\noindent \begin{enumerate}[(i)]
\item If $d=2$, the function $\phi \in \widetilde{H}^{-1/2}({\D}_1)$ defined by
$$ \forall (x,0) \in {\D}_1, \:\: \phi(x) = \frac{2}{\log 2\sqrt{1-x^2}}$$
satisfies 
$$ -\frac{1}{2\pi} \int_{{\D}_1}{\log | x- y| \phi(y) \: \d y} = 1 \text{ for } x \in \D_1, \text{ and } \int_{{\D}_1}{\phi(x) \: \d x} = \frac{2\pi}{\log 2}~.$$
\item If $d=2$, the function $\phi \in \widetilde{H}^{1/2}( \D)$ defined by
$$ \forall (x,0) \in{\D}_1, \:\: \phi(x) = -2 \sqrt{1-x^2}$$
satisfies
$$ \frac{1}{2\pi} \: \underset{\eta\downarrow 0}{ \text{\rm f.p.}} \int_{{\D}_1 \setminus (x-\eta,x+\eta)} {\frac{1}{ | x- y|^2} \phi(y) \: \d y} = 1 \text{ for } x \in \D_1, \text{ and } \int_{{\D}_1}{\phi(x) \: \d x} = -\pi~.$$
\item If $d=3$, the function $\phi \in \widetilde{H}^{-1/2}({\D}_1)$ defined by
\begin{equation}\label{eq.eqdist3ddir}
 \forall (x,0) \in{\D}_1, \:\: \phi(x) = \frac{4}{\pi\sqrt{1-|x|^2}}
 \end{equation}
satisfies 
\begin{equation}\label{eq.eqdist3ddirprops}
 \frac{1}{4\pi} \int_{{\D}_1}{\frac{1}{| x- y|} \phi(y) \: \d y} = 1 \text{ for } x \in \D_1, \text{ and } \int_{{\D}_1}{\phi(x) \: \d s(x)} = 8~.
 \end{equation}
\item If $d=3$, the function $\phi \in \widetilde{H}^{1/2}({\D}_1)$ defined by
$$ \forall x \in \D_1, \:\: \phi(x) = -\frac{1}{\pi} \sqrt{1-|x|^2}$$
satisfies 
$$ \frac{1}{4\pi} \: \underset{\eta \to 0}{\text{\rm f.p.}} \int_{{\D}_1 \setminus B(x,\eta)} {\frac{1}{ | x- y|^3} \phi(y) \: \d y} = 1 \text{ for } x \in \D_1, \text{ and } \int_{\D_1}{\phi(x) \: \d s(x)} = -\frac{2}{3}~.$$
\end{enumerate}
\end{proposition}

The two-dimensional results (i) and (ii) can be proved by means of conformal mapping techniques; see \cite{mclean2000strongly}, Exercises 8.15 and 8.16. 
The item (iii) is a fairly well-known result about the capacitance of a flat disk in 3d, and we refer to \cite{jackson2007classical} Exercise 3.3, or to \cite{copson1947problem} 
for an elegant proof using the connection with Abel's integral equation. 
Finally, for the item (iv), we refer to the articles \cite{krenk1979circular,martin1986orthogonal}; see also \cite{ramaciotti2016theoretical} 
% Th .2.7.1 in there
where these results are used to build a series expansions for the hypersingular operator, for the purpose of operator preconditioning. 

\begin{remark}
Let us comment about the physical significance of the formulas in \cref{prop.eqdistrib}. 
The points $(i)$ and $(iii)$ are concerned with the Newtonian potential. In particular, 
the equilibrium distribution $\phi$ is the charge distribution on $\D_1$ which ensures that the induced electrostatic potential is constant (equals $1$) on $\D_1$. 
The total charge $\int_{\D_1}{\varphi \: \d s}$ associated with this distribution corresponds to the Newtonian version of the capacity of $\D_1$. 

The points $(ii)$ and $(iv)$ are perhaps a little more unfamiliar. The function $\phi$ is the dipole distribution on $\D_1$ which ensures that the induced electric current through $\D_1$ is constant (equals $1$). The quantity $\int_{\D_1}{\phi \: \d s}$ is the associated total dipole charge.
\end{remark}

%%%%%%%%%%%%%%%%%%%%%%%%%%%%%%%
\section{Some useful results about integral operators with homogeneous kernels}\label{sec.homogkernel}
%%%%%%%%%%%%%%%%%%%%%%%%%%%%%%%

\noindent In this section, we collect some useful properties of integral operators
whose kernels satisfy specific homogeneity properties. This material is taken from Chap. 4 in \cite{nedelec2001acoustic}. 

\begin{definition}\label{def.homkernel}
Let $m$ be a non negative integer; a homogeneous kernel of class $-m$ is 
a function $K(x,z) \in {\mathcal C}^\infty(\R^d \times \R^d\setminus \left\{0\right\})$ which satisfies the following properties: 
\begin{itemize}
\item For all multi-indices $\alpha, \beta \in \N^d$,
$$\sup \limits_{x\in \R^d} \sup\limits_{|z|=1} \left\lvert \frac{\partial^\alpha}{\partial x^\alpha} \frac{\partial^\beta}{\partial z^\beta} K(x,z) \right\lvert < \infty~.$$ 
\item For all $x \in \R^d$ and all index $\beta$ with $|\beta|=m$, the function $z \mapsto \frac{\partial^\beta}{\partial z^\beta} K(x,z)$ 
is odd and homogeneous of degree $-(d-1)$, {\it i.e.},
\begin{equation*}
\forall x \in \R^d, \: z\in \R^d\setminus \left\{ 0\right\}, \:\: t >0, \quad \frac{\partial^\beta}{\partial z^\beta} K(x,-z) = -\frac{\partial^\beta}{\partial z^\beta} K(x,z), \text{ and } \\
\frac{\partial^\beta}{\partial z^\beta} K(x,tz) = t^{-(d-1)}\frac{\partial^\beta}{\partial z^\beta} K(x,z)~. 
\end{equation*}
\end{itemize}
\end{definition}

To each homogeneous kernel, it is possible to associate an integral operator $T_K$, acting on functions $\varphi : \partial D \to \R$, via the formula 
\begin{equation}\label{eq.defTK}
 T_K \varphi(x) = \int_{\partial D}{K(x,x-y) \varphi(y) \: \d y}. 
 \end{equation}
The following result specifies the mapping properties of this integral operator.

\begin{theorem}\label{th.TK}
Let $D \subset \R^d$ be a smooth bounded domain, and let $K(x,z)$ be a homogeneous kernel of class $-m$, 
with associated operator $T_K$ defined in \cref{eq.defTK}. Then for each $s \in \R$, the mapping $T_K$ defines a bounded operator
$$ T_K : H^s(\partial D) \to H^{s+m}(\partial D)~, $$
that is, there exists a constant $C_{s,D,K}$ such that 
$$ \forall \varphi \in H^s(\partial D), \:\: || T_K \varphi ||_{H^{s+m}(\partial D)} \leq C_{s,D,K}  || \varphi ||_{H^{s}(\partial D)} ~.$$
The constant  $C_{s,D,K}$, that is, the operator norm of $ T_K : H^s(\partial D) \to H^{s+m}(\partial D)$, 
depends only on $s$, $D$, and the kernel $K$. It can be estimated by 
\begin{equation}\label{eq.CsDK}
 C_{s,D,K} \leq C_{s,D} \sup\limits_{|\alpha| \leq k \atop |\beta| \leq k}\sup \limits_{x\in \partial D} \sup\limits_{|z|=1} \left\lvert \frac{\partial^\alpha}{\partial x^\alpha} \frac{\partial^\beta}{\partial z^\beta} K(x,z) \right\lvert,
 \end{equation}
where $k$ is a non negative integer which only depends on the space dimension $d$, and $C_{s,D}$ is a constant which depends only on $d$ and the domain $D$. 
\end{theorem}

\begin{remark}
The above statement is Th. 4.3.1 in \cite{nedelec2001acoustic}. 
In that reference, the continuity constant of the mappings $T_K : H^{s}(\partial D) \to H^{s+m}(\partial D)$ is not stated explicitly, 
but formula \cref{eq.CsDK} is obtained by tracking the dependence of this constant with respect to $K$ throughout the proof. 
\end{remark}

We finally state the following result about the potential operator induced by a homogeneous kernel of class $-m$; see \cite{costabel2012shape} and \cite{eskin1980boundary} (Lemma 21.7).

\begin{theorem}\label{th.TKpot}
Let $D \subset \R^d$ be a smooth bounded domain, and let $K(x,z)$ be a homogeneous kernel of class $-m$.
Then for each $s \in \R$, the associated potential operator: 
$$ T_K \varphi(x) = \int_{\partial D}{K(x,x-y) \varphi(y) \: \d y}, \quad x \notin \partial D, $$
is a bounded mapping from $H^s(\partial D)$ into $H^{s+m+\frac{1}{2}}(D)$ and $H^{s+m+\frac{1}{2}}_{\text{loc}}(\R^d \setminus \overline D)$.
For any compact subset $L \Subset \R^d \setminus D$, there exists a constant $C_{s,D,K,L}$ such that,
$$ \forall \varphi \in H^s(\partial D), \:\: || T_K \varphi ||_{H^{s+m+\frac{1}{2}}(D)} + || T_K \varphi ||_{H^{s+m+\frac{1}{2}}(L)} \leq C_{s,D,K,L}  || \varphi ||_{H^{s}(\partial D)} .$$
%The constant  $C_{s,D,K,L}$
%depends only on $s$, $D$, the compact set $L$ and the kernel $K$. It can be estimated as: 
%\begin{equation}\label{eq.CsDKL}
% C_{s,D,K} \leq  \sup\limits_{|\alpha| \leq k \atop |\beta| \leq k}\sup \limits_{x\in L} \sup\limits_{|z|=1} \left\lvert \frac{\partial^\alpha}{\partial x^\alpha} \frac{\partial^\beta}{\partial z^\beta} K(x,z) \right\lvert,
% \end{equation}
%where $k$ is a non negative integer which only depends on the space dimension $d$. 
\end{theorem}

%%%%%%%%%%%%%%%%%%%%%%%%%%%%%%%
%
\bibliographystyle{siam}
\bibliography{/Users/charlesdapogny/Documents/genbib.bib}
%%%%%%%%%%%%%%%%%%%%%%%%%%%%%%%

\end{document}